\documentclass[11pt]{article}
\usepackage{amssymb}
\usepackage{amsfonts}
\usepackage{amsmath}

\newtheorem{theorem}{Theorem}[section]

\newtheorem{definition}[theorem]{Definition}

\newtheorem{lemma}[theorem]{Lemma}

\newtheorem{proposition}[theorem]{Proposition}
\newtheorem{remark}[theorem]{Remark}

\newenvironment{proof}[1][Proof]{\textbf{#1.} }{\hfill\rule{0.5em}{0.5em}}
{\catcode`\@=11\global\let\AddToReset=\@addtoreset
\AddToReset{equation}{section}

\AddToReset{theorem}{section}

\input{tcilatex}
\begin{document}

\title{Elliptic Hamilton-Jacobi systems and Lane-Emden Hardy-H\'{e}non
equations }
\author{Marie-Fran\c{c}oise Bidaut-V\'{e}ron \and Marta Garcia Huidobro}
\maketitle

\begin{abstract}
Here we study the solutions of any sign of the system 
\begin{equation*}
\left\{ 
\begin{array}{c}
{-\Delta u}_{1}=\left\vert \nabla u_{2}\right\vert ^{p}, \\ 
{-\Delta u}_{2}=\left\vert \nabla u_{1}\right\vert ^{q},%
\end{array}%
\right.
\end{equation*}%
in a domain of $\mathbb{R}^{N},$ $N\geqq 3$ and $p,q>0,$ $pq>1..$ We show
their relation with Lane-Emden Hardy-H\'{e}non equations 
\begin{equation*}
-{\Delta }_{\mathbf{p}}^{\mathbf{N}}w{=\varepsilon r}^{\sigma }w^{\mathbf{q}%
},\qquad \varepsilon =\pm 1,
\end{equation*}%
where $u\mapsto {\Delta }_{\mathbf{p}}^{\mathbf{N}}{u}$ $(\mathbf{p}>1)$ is
the $\mathbf{p}$-Laplacian in dimension $\mathbf{N,}$ $\mathbf{q}>\mathbf{p}%
-1$ and $\sigma \in \mathbb{R}.$ This leads us to explore these equations in
not often tackled ranges of the parameters $\mathbf{N,p},\sigma $. We make a
complete description of the radial solutions of the system and of the
Hardy-Henon equations and give nonradial a priori estimates and Liouville
type results for the system.
\end{abstract}

\tableofcontents

\section{Introduction}

Here we study the solutions of any sign of the Hamilton-Jacobi type system 
\begin{equation}
\left\{ 
\begin{array}{c}
{-\Delta u}_{1}=\left\vert \nabla u_{2}\right\vert ^{p}, \\ 
{-\Delta u}_{2}=\left\vert \nabla u_{1}\right\vert ^{q},%
\end{array}%
\right.  \label{one}
\end{equation}%
in a domain $\Omega $ of $\mathbb{R}^{N},$ $N\geqq 3$, and 
\begin{equation*}
p,q>0,\quad pq>1,\text{ and we can assume }p\geq q,
\end{equation*}%
so that $p>1.$ Our purpose is to describe the behaviour and the existence of
solutions when $\Omega =B_{r_{0}}\backslash \left\{ 0\right\} ,$ $\mathbb{R}%
^{N}\backslash \left\{ 0\right\} ,\mathbb{R}^{N}\backslash \overline{%
B_{r_{0}}}$ or $\mathbb{R}^{N}.$ Note that if $({u}_{1},{u}_{2})$ is a
solution, then $({u}_{1}+C_{1},{u}_{2}+C_{2})$ is a solution, in particular
the constants are solutions. The study of the radial solutions is
fundamental for the understanding of the system. As shown below, it appears
that they are governed by the solutions of Lane-Emden Hardy-H\'{e}non
equations 
\begin{equation}
-{\Delta }_{\mathbf{p}}^{\mathbf{N}}w{=\varepsilon r}^{\sigma }w^{\mathbf{q}%
},\qquad \varepsilon =\pm 1,  \label{hql}
\end{equation}%
where $u\mapsto {\Delta }_{\mathbf{p}}^{\mathbf{N}}{u}$ $(\mathbf{p}>1)$ is
the $\mathbf{p}$-Laplacian in dimension $\mathbf{N,}$ $\mathbf{q}>\mathbf{p}%
-1$ and $\sigma \in \mathbb{R}.$ This justifies to make the point on the
actual knowledge of these equations, and give a complete study in any range
of the parameters $\mathbf{N,p},\sigma $. \bigskip

In the last decades, a great number of elliptic systems deal with positive
solutions of semilinear or quasilinear, with source terms in the right hand
side, involving powers of the solutions. Our study is motivated by the well
known Lane-Emden system 
\begin{equation}
\left\{ 
\begin{array}{c}
{-\Delta u}_{1}=\left\vert x\right\vert ^{a}u_{2}^{p}, \\ 
{-\Delta u}_{2}=\left\vert x\right\vert ^{b}u_{1}^{q},%
\end{array}%
\right.  \label{LE}
\end{equation}%
which has developed an extremely rich literature, starting from the
conjecture of nonexistence of solutions in $\mathbb{R}^{N}$ for $a=b=0$ when 
\begin{equation*}
\frac{1}{p+1}+\frac{1}{q+1}>\frac{N-2}{N},
\end{equation*}%
first proved in the radial case for any $N\geq 3$ in \cite{Mi}, obtained in
of \cite{SeZo2} for $N=3,$ extended to $N=4$ in \cite{So2}, and still open
for $N\geq 5.$ The existence of solutions in $\mathbb{R}^{N}\backslash
\left\{ 0\right\} $ of such elliptic systems with a possible singularity at
the origin, or in an exterior domain $\mathbb{R}^{N}\backslash \overline{%
B_{r_{0}}},$ and the question of a priori estimates has been developed in
many works, with extensions to quasilinear operators, with possible weight
functions, or to analogous problems on manifolds, and it is impossible to
quote all of them. Let us also mention a significant amount of results of
non existence of supersolutions in $\mathbb{R}^{N}$, started in \cite{MiPo}, 
\cite{BiPo}.\bigskip

In contrast, only few results are available in the literature for elliptic
systems with gradient terms. A general study of nonexistence of positive
supersolutions of systems with source terms is given in \cite{Fi}. We can
also mention some works on more specific systems in \cite{FiVi}, \cite%
{GhGiSi}, \cite{BaGiWa}, \cite{BuGa}, and \cite{DiLaSc}, \cite{AtBe}, \cite%
{AbAtLa}, \cite{Si}. Up to our knowledge the study system (\ref{one}) began
very recently, with the publication of \cite{CoRa}, where the existence of
positive singular solutions of the Dirichlet problem in $\Omega \backslash
\left\{ 0\right\} $ was proved when $\Omega $ is a small perturbation of a
ball of center $0,$ and the work of \cite{AtBeLa}, where existence and non
existence results are given for the system with forcing terms.\medskip
\medskip

In the sequel, we distinguish three types of solutions: the positive ones,
solutions of (\ref{one}), with source type gradient terms 
\begin{equation}
\left\{ 
\begin{array}{c}
{-\Delta u}_{1}=\left\vert \nabla u_{2}\right\vert ^{p}, \\ 
{-\Delta u}_{2}=\left\vert \nabla u_{1}\right\vert ^{q},%
\end{array}%
\right. \qquad {u}_{1},{u}_{2}\geq 0,  \label{sou}
\end{equation}%
the negative ones, equivalently, setting $\widetilde{u}_{1}=-{u}_{1},%
\widetilde{u}_{2}=-{u}_{2},$ the solutions of a system with absorption terms 
\begin{equation}
\left\{ 
\begin{array}{c}
{-\Delta }\widetilde{u}_{1}+\left\vert \nabla \widetilde{u}_{2}\right\vert
^{p}=0, \\ 
{-\Delta }\widetilde{u}_{2}+\left\vert \nabla \widetilde{u}_{1}\right\vert
^{q}=0,%
\end{array}%
\right. \qquad \widetilde{u}_{1},\widetilde{u}_{2}\geq 0,  \label{abs}
\end{equation}%
and the mixed ones, equivalently setting $\widehat{u_{1}}=u_{1},$ $\widehat{%
u_{2}}=-{u}_{2},$ the solutions of the mixed type system 
\begin{equation}
\left\{ 
\begin{array}{c}
{-\Delta }\widehat{u_{1}}=\left\vert \nabla \widehat{u_{2}}\right\vert ^{p},
\\ 
{\Delta }\widehat{u_{2}}=\left\vert \nabla \widehat{u_{1}}\right\vert ^{q},%
\end{array}%
\right. \qquad \widehat{u_{1}},\widehat{u_{2}}\geq 0.  \label{mixt}
\end{equation}%
$\bigskip $

In \textbf{Section \ref{prop}} we study the existence of particular
solutions of system (\ref{one}). We briefly mention the main results
concerning the scalar case of the well known Hamilton-Jacobi equation 
\begin{equation}
-{\Delta u}=\left\vert \nabla u\right\vert ^{q},  \label{eq}
\end{equation}%
where $q>1.$ The existence of particular positive or negative radial
solutions $u^{\ast }(r)=A^{\ast }r^{\frac{q-2}{q-1}},$ $r=\left\vert
x\right\vert ,$ for $q\neq 2,$ where $A^{\ast }$ has the sign of $%
(2-q)((N-1)q-N),$ puts in evidence two critical values $q=\frac{N}{N-1}$ and 
$q=2.$ Concerning system (\ref{one}), these two critical values are replaced
by four critical conditions linking the parameters, namely $q=q_{i}=q_{i}(p)$%
, $i=1,2,3,4,$ defined by the relations 
\begin{equation}
(N-1)pq_{1}=N+q_{1},\qquad (N-1)pq_{2}=N+p,\qquad q_{3}(p-1)=2,\qquad
p(q_{4}-1)=2.  \label{com}
\end{equation}%
In the problem, moreover another value is involved, defined by the relation 
\begin{equation}
(N-1)pq^{\ast }=N+\frac{p+q^{\ast }}{2},  \label{comstar}
\end{equation}%
corresponding to the Sobolev exponent relative to equation (\ref{hql}%
).\bigskip

In \textbf{Section \ref{rad}} we study the radial signed solutions of system
(\ref{one}). We show that they can be completely described. Indeed the
system is invariant by the scaling transformation $T_{\ell },\ell >0,$
defined by 
\begin{equation*}
T_{\ell }\left[ (u_{1},u_{2})\right] (x)=(l^{\frac{p(q-1)-2}{pq-1}%
}u_{1}(\ell x),l^{\frac{q(p-1)-2}{pq-1}}u_{2}(\ell x)),
\end{equation*}%
and thus the radial case reduces to an autonomous system of order 4. Due to
the particular form of system (\ref{one}), which involves only the gradients
of the functions, we can reduce to a system of order 2.\ The radial study
offers a double interest:\bigskip

$\bullet $ \textit{The radial system reduces to a quadratic Lotka-Volterra
type system: } 
\begin{equation*}
\left\{ 
\begin{array}{c}
S_{t}=S(N-(N-1)p+S+pZ), \\ 
Z_{t}=Z(N-(N-1)q-qS-Z).%
\end{array}%
\right.
\end{equation*}%
where 
\begin{equation*}
S=r\frac{\left\vert u_{2}^{\prime }\right\vert ^{p}}{u_{1}^{\prime }}=-r%
\frac{(r^{N-1}u_{1}^{\prime })^{\prime }}{r^{N-1}u_{1}^{\prime }},\qquad Z=-r%
\frac{\left\vert u_{1}^{\prime }\right\vert ^{q}}{u_{2}^{\prime }}=-r\frac{%
(r^{N-1}u_{2}^{\prime })^{\prime }}{r^{N-1}u_{2}^{\prime }},\qquad t=\ln r.
\end{equation*}

Moreover system (\ref{one}) is governed by the Hardy-H\'{e}non equations (%
\ref{hql})in the following way:\bigskip

$\bullet $ \textit{For any radial solution }$(u_{1},u_{2})$\textit{\ of
system (\ref{one}), the function }$w=r^{N-1}\left\vert u_{1}^{\prime
}\right\vert $\textit{\ satisfies equation (\ref{hql}) with the specific
values } 
\begin{equation*}
\mathbf{N}=1+\frac{(N-1)(p-1)}{p},\qquad \mathbf{p}=1+\frac{1}{p},\qquad 
\mathbf{q}=q,\qquad \sigma =(N-1)\frac{1-pq}{p},\qquad \varepsilon =-\text{%
sign}(u_{1}^{\prime }u_{2}^{\prime }).
\end{equation*}

As an unexpected and remarkable fact, the system (\ref{sou}), which presents
two \textit{source} terms $\left\vert \nabla u_{2}\right\vert
^{p},\left\vert \nabla u_{1}\right\vert ^{q},$ is not linked with a Hardy-H%
\'{e}non equation with a source term ($\varepsilon =1$) but an equation with
an \textit{absorption} term ($\varepsilon =-1$), see Remark \ref{souab}%
.\bigskip

\textbf{Section \ref{app}} is devoted to a complete study of the radial
local and global positive solutions of the Hardy-Henon equations (\ref{hql})
for $\varepsilon =1$. They are well known in the case $\mathbf{N>p}>-\sigma ,
$ and we cannot quote the inmense literature on the subject, starting from
the works of \cite{GiSp} and \cite{CaGiSp}. The main purposes of the
articles are the obtention of Liouville type results, symmetry properties
and a priori estimates for the equation with source ($\varepsilon =1$), and
the study of large solutions for the equation with absorption ($\varepsilon
=-1$),. The other ranges of the parameters $\mathbf{N,p},\sigma ,$ such as $%
\mathbf{p}>\mathbf{N}$ or $\sigma <-\mathbf{p,}$ $\sigma <-\mathbf{N}$ seem
to be less considered. We first mention the remarkable work of \cite{SeZo}
valid for $\sigma =0$ and any $\mathbf{N,p.}$ Let us also quote the
nonexistence of possibly nonradial solutions, even in a very weak sense,
proved in the pioneer paper \cite{BrCa} for $\varepsilon =1,$ $\mathbf{N}%
\geq 3,$ $\mathbf{p}=2,$ $\sigma \leq -2$; and recent interesting symmetry
results of \cite{AvBr} in the case $\mathbf{p}=2.$ see also  \cite{Ga} for
an extension to any extension to any $\mathbf{p}>1$, $\sigma \leq -\max (%
\mathbf{N},\mathbf{p})$ under some regularity assumptions.\medskip\ 

Here we treat all the cases of distinct $\mathbf{N,p},\sigma $ at Theorems %
\ref{orange}, \ref{hypF}, \ref{rose}, \ref{white}, \ref{yellow}, \ref{hypE}.
The limit cases $\sigma =-\mathbf{p,}$ $\sigma =-\mathbf{N,}$ $\mathbf{p}=%
\mathbf{N,}$ where logarithmic type solutions appear, are described at
Theorems \ref{sigmap},\ref{sigman}, \ref{pegaln}, and the critical case $%
\mathbf{p}=\mathbf{N=-}\sigma $ at Theorem \ref{pnsigma}. Independently of
their application to system (\ref{one}), they offer a remarkable diversity
of types of behaviour. We hope that they constitute a solid basis for a
future nonradial analysis. \bigskip 

Our study is based on the analysis of the phase plane of the associated
quadratic Lotka-Volterra\textit{\ }system, 
\begin{equation*}
\left\{ 
\begin{array}{c}
s_{t}=s(\frac{\mathbf{p}-\mathbf{N}}{\mathbf{p}-1}+s+\frac{z}{\mathbf{p}-1}),
\\ 
z_{t}=z(\mathbf{N}+\sigma -\mathbf{q}s-z),%
\end{array}%
\right.
\end{equation*}%
where $s=-r\frac{w^{\prime }}{w}$ and $z=-\varepsilon r^{1+\sigma }w^{%
\mathbf{q}}\left\vert w^{\prime }\right\vert ^{-\mathbf{p}}w^{\prime },$
already started in \cite{BiGi}. We observe that this system is neither
competitive, nor cooperative as soon as $(\mathbf{p}-1)\mathbf{q}>0$. Note
also that our study gives local existence results without involving fixed
point methods, and global ones without introducing energy functions. Among
all the results, let us mention a few striking examples:\medskip

$\bullet $ For $\mathbf{p>N}>-\sigma $ and $\varepsilon =\pm 1$, we get
existence of solutions of equation (\ref{hql}) in $C^{0}(B_{r_{0}})\cap
C^{2}(B_{r_{0}}\backslash \left\{ 0\right\} )$ of three types : such that $%
\lim_{r\rightarrow 0}w=c>0,$ and either $\lim_{r\longrightarrow 0}r^{-\frac{%
\sigma +1}{\mathbf{p}-1}}w^{\prime }=$ $d\neq 0$ or $\lim_{r\rightarrow 0}r^{%
\frac{\mathbf{N}-1}{\mathbf{p}-1}}w^{\prime }=d\neq 0,$ and solutions such
that $\lim_{r\rightarrow 0}w=0$ and $\lim_{r\rightarrow 0}r^{\frac{\mathbf{N}%
-\mathbf{p}}{\mathbf{p}-1}}w=k>0.$\medskip

$\bullet $ For $-\sigma >\max (\mathbf{p},\mathbf{N),}$ there exist
solutions in an exterior domain $\mathbb{R}^{N}\backslash \overline{B_{r_{0}}%
}$ such that $\lim_{r\rightarrow \infty }w=c>0$, see Theorems \ref{hypF}, %
\ref{white}, \ref{pegaln}.\medskip

$\bullet $ It is well known that for $\mathbf{N>p}>-\sigma $ there exist
explicit solutions in $C^{0}(\mathbb{R}^{N})\cap C^{2}(\mathbb{R}%
^{N}\backslash \left\{ 0\right\} ),$ often called ground states, when $q$ is
the critical Sobolev exponent $q_{S}=\frac{\mathbf{N}(\mathbf{p}-1)+\mathbf{p%
}+\mathbf{p}\sigma }{\mathbf{N}-\mathbf{p}},$ given by $w=c(d+r^{\frac{%
\mathbf{p}+\sigma }{\mathbf{p}-1}})^{-\frac{\mathbf{N}-\mathbf{p}}{\mathbf{p}%
+\sigma }},$ where $c>0$ and $d=d(c,\mathbf{N,p},\sigma ).$ In fact the same
phenomena holds when $-\sigma >\mathbf{p}>\mathbf{N}$ (see Theorem \ref{hypF}%
).\medskip

$\bullet $ For $\varepsilon =-1,$ $\sigma =-\frac{\mathbf{p}(N-1)}{p-1},$ in
particular when $\mathbf{p}=\mathbf{N=-}\sigma ,$ there also exists a family
of explicit solutions, see Proposition \ref{expli}.\bigskip

For convenience the proofs of the main results of Section \ref{app}, using
various techniques of dynamical systems, are given in the Appendix. $%
\medskip $

In \textbf{Section \ref{exi}} we describe the behaviour of the radial
solutions of system (\ref{one}), from the results of Section \ref{app}.
Because of the great diversity of the possible solutions, we concentrate our
study on constant sign solutions in $B_{r_{0}}\backslash \left\{ 0\right\} $
or $\mathbb{R}^{N}\backslash \overline{B_{r_{0}}},$ and above all on global
solutions in $\mathbb{R}^{N}\backslash \left\{ 0\right\} .\medskip $

Comparing to the scalar case, the situation is much more intricated: even in
the case $p=q,$ where we show that the system can be completely integrated,
there exists an infinity of solutions such that $u_{1}\neq u_{2}$. In the
general case $p\geq q$ we give all the local and global nonconstant
solutions of systems (\ref{sou}), (\ref{abs}) and (\ref{mixt}). It appears
that the components $u_{1},u_{2}$ of a solution in $(0,\infty )$ can admit
different kinds of singularties near $0$: either $\lim_{r\longrightarrow
0}\left\vert u_{i}\right\vert =\infty ,$ and we say that $u_{i}$ is $\infty $%
\textit{-singular}, or if $\lim_{r\longrightarrow 0}u_{i}=c\in \mathbb{R}$
and $\lim_{r\longrightarrow 0}\left\vert u_{i}^{\prime }\right\vert =\infty
, $ and we say that $u_{i}$ is a \textit{cusp-solution, }see Remark \ref{2.3}%
. Concerning (\ref{sou}), we get the following, from Theorems \ref{Dirac}, %
\ref{thsou} and \ref{thall}:\ 

\begin{theorem}
\label{posi}Let $pq>1,p\geq q,$ and let $q_{1},q_{2},q_{3},q_{4}$ be defined
in (\ref{com}). Consider the system (\ref{sou}).

(1) Local solutions near 0: There exists radial solutions in $%
B_{r_{0}}\backslash \left\{ 0\right\} $ such that 
\begin{equation*}
\left\{ 
\begin{array}{c}
{-\Delta u}_{1}=\left\vert \nabla u_{2}\right\vert ^{p}+C_{1}\delta _{0}, \\ 
{-\Delta u}_{2}=\left\vert \nabla u_{1}\right\vert ^{q}+C_{2}\delta _{0},%
\end{array}%
\right.
\end{equation*}%
with $C_{1},C_{2}>0$ when $p<\frac{N}{N-1},$ with $C_{1}=0,C_{2}>0$ when $%
q<q_{1},$ with $C_{1}>0,C_{2}=0$ when $q<q_{2}.$

(2) Global solutions: up to additive constants,

$\bullet $ For $q_{2}<q<q_{3},$ there exists a $\infty $-singular solution $%
(u_{1}^{\ast },u_{2}^{\ast })=(A_{1}^{\ast }r^{\frac{p(q-1)-2}{pq-1}%
},A_{2}^{\ast }r^{\frac{q(p-1)-2}{pq-1}}),$ with $A_{1}^{\ast },A_{2}^{\ast
}>0,$ and also solutions such that 
\begin{equation*}
(u_{1},u_{2})\sim _{r\rightarrow 0}(u_{1}^{\ast },u_{2}^{\ast }),\;\left\{ 
\begin{array}{c}
\lim_{r\rightarrow \infty }r^{N-2}u_{1}=c_{1}>0,\;\lim_{r\rightarrow \infty
}r^{(N-1)q-2}u_{2}=c(c_{1})>0,\;\text{if }q<\frac{N}{N-1}, \\ 
\lim_{r\rightarrow \infty }r^{N-2}u_{1}=c_{1}>0,\quad \lim_{r\rightarrow
\infty }r^{N-2}u_{2}=c_{2}>0,\quad \quad \;\text{if }q>\frac{N}{N-1}.%
\end{array}%
\right.
\end{equation*}

$\bullet $ For $q>q_{4},$ there exist cusp-solutions $(u_{1},u_{2})$ such
that 
\begin{equation*}
(u_{1},u_{2})\sim _{r\rightarrow 0}(\left\vert A_{1}^{\ast }\right\vert r^{%
\frac{p(q-1)-2}{pq-1}},\left\vert A_{2}^{\ast }\right\vert r^{\frac{q(p-1)-2%
}{pq-1}}),\qquad \lim_{r\rightarrow \infty }u_{1}=c_{1}>0,\quad
\lim_{r\rightarrow \infty }u_{2}=c_{2}>0,
\end{equation*}%
and there is no other global solution.
\end{theorem}

We also give the existence of solutions of (\ref{sou}) in $\mathbb{R}%
^{N}\backslash \overline{B_{r_{0}}}$ at Theorem \ref{exterior}. \bigskip

The cases of system (\ref{abs}) and (\ref{mixt}), treated at Theorems \ref%
{Dirac}, \ref{thabs}, \ref{thmix} when $q<q_{4}$ are even richer. Note the
existence of a family of explicit solutions of system (\ref{mixt}) in the
case $q=q^{\ast }$, which corresponds to the critical Sobolev exponent for
the Hardy-Henon equation, see Theorem \ref{orange}. The case $q>q_{4},$
treated at Theorem \ref{thall} is remarkable, since it gives the existence
of bounded solutions of any of the three systems (\ref{sou}), (\ref{abs})
and (\ref{mixt}). \bigskip

In \textbf{Section \ref{nonrad}} we extend the study of system (\ref{one})
to the nonradial case. Our main aim is finding upper estimates, which
appears to be a good challenge. Indeed the system is not variational, and
does not offer monotony or comparison properties which could be used as it
was done in \cite{Bi2}, \cite{So2}. Moreover the Bernstein technique,
developped in \cite{GiSp}, \cite{BiVe}, \cite{SeZo} for semilinear or
quasilinear equations, appears not to be efficient for systems, as in the
case of system (\ref{LE}), except in special cases, as in \cite{BiRa}.
Finally we cannot use the very performant methods of moving planes or moving
spheres, see the pioneer results of \cite{CaGiSp}, \cite{ReZo}, even with
boundedness assumptions on the gradients, because in $\mathbb{R}^{N}$ there
aways exist constant nontrivial solutions. Here we give a first estimate,
and in the same way a nonexistence result, see Theorem \ref{esti},
Proposition \ref{esto} and Theorem \ref{Osserman}:

\begin{theorem}
\label{estexi}Let $pq>1,p\geq q\geq 1.$

(i) If $(N-1)pq<\max (N+p,N+q),$ \textbf{\ }then any supersolution of system
(\ref{one}) in $\Omega =\mathbb{R}^{N}$, without condition of sign, is
constant.\medskip

(ii) Let $(u_{1},u_{2})$ be any supersolution of system (\ref{sou}) in $%
B_{r_{0}}\backslash \left\{ 0\right\} $ (resp. in $\mathbb{R}^{N}\backslash 
\overline{B_{r_{0}}}).$ Then there exists $\rho \in (0,r_{0})$ and $C>0$,
depending on $u_{1},u_{2}$, such that for any $r\in \left( 0,\rho \right) ,$ 
\begin{eqnarray*}
\left\vert \overline{u_{1}}(r)\right\vert &\leq &C\left\{ 
\begin{array}{c}
\max (1,r^{-\frac{2-p(q-1)}{pq-1}})\text{ if }p(q-1)\neq 2, \\ 
\left\vert \ln r\right\vert \text{ if }p(q-1)=2,%
\end{array}%
\right. \text{ } \\
\left\vert \overline{u_{2}}(r)\right\vert &\leq &C\left\{ 
\begin{array}{c}
\max (1,r^{-\frac{2-q(p-1)}{pq-1}})\text{ if }q(p-1)\neq 2, \\ 
\left\vert \ln r\right\vert \text{ if }q(p-1)=2.%
\end{array}%
\right.
\end{eqnarray*}

(iii) Let $(\widetilde{u_{1}},\widetilde{u_{2}})$ be any positive
subsolution of system (\ref{abs}) in $B_{r_{0}}\backslash \left\{ 0\right\} $
(resp. in $\mathbb{R}^{N}\backslash \overline{B_{r_{0}}}).$ Then there
exists $\rho \in (0,r_{0})$ and $C>0$, depending on $\widetilde{u_{1}},%
\widetilde{u_{2}}$, such that for $0<\left\vert x\right\vert <\rho ,$ 
\begin{eqnarray*}
\widetilde{u_{1}}(x) &\leq &C\left\{ 
\begin{array}{c}
\max (1,\left\vert x\right\vert ^{-\frac{2-p(q-1)}{pq-1}})\text{ if }%
p(q-1)\neq 2, \\ 
\left\vert \ln \left\vert x\right\vert \right\vert \text{ if }p(q-1)=2,%
\end{array}%
\right. \\
\widetilde{\text{ }u_{2}}(x) &\leq &C\left\{ 
\begin{array}{c}
\max (1,\left\vert x\right\vert ^{-\frac{2-q(p-1)}{pq-1}})\text{ if }%
q(p-1)\neq 2, \\ 
\left\vert \ln \left\vert x\right\vert \right\vert \text{ if }q(p-1)=2.%
\end{array}%
\right.
\end{eqnarray*}
\end{theorem}

Note that the range $(N-1)pq<\max (N+p,N+q)$ in assertation (i) is optimal
for supersolutions. Nevertheless we think that the result is much more
general concerning solutions: \medskip

\textbf{Conjecture}: for any $p,q$ such that $pq>1,$ all the \textit{%
solutions }of system (\ref{one}) in $\Omega =\mathbb{R}^{N}$, without
condition of sign\textit{, }are constant. \medskip

We end this section by the study of the existence of supersolutions of
system (\ref{sou}) in an exterior domain, see Theorem \ref{nonext}, and the
existence of solutions in a domain $\Omega $ with measure data, see Theorem %
\ref{orig}.\bigskip

In \textbf{Section \ref{exten} }we conclude the paper by suggesting some
extensions of our results to more general systems.

\section{Critical exponents of the system\label{prop}}

We first recall some results on the scalar case of equation (\ref{eq}),
which can be found for example in \cite{BiGaVe}.

\subsection{On the scalar case}

Clearly a function $u$ is a negative solution of (\ref{eq}) if and only if $%
\widetilde{u}=-u$ is a positive solution of the equation with absorption%
\begin{equation}
-{\Delta }\widetilde{u}+\left\vert \nabla \widetilde{u}\right\vert ^{q}=0.
\label{eqa}
\end{equation}%
As we mentioned above, two critical values are involved: $q=\frac{N}{N-1},$
and $q=2.$ When $q=2$ the equation is equivalent to $\Delta (e^{u})=0,$ so
all the solutions are decribed; in particular all the solutions in $\mathbb{R%
}^{N}\backslash \left\{ 0\right\} $ are radial, and given by $u=\ln
(cr^{2-N}+d),$ $c,d\in \mathbb{R}.$\medskip

(1) For $q\neq 1,\frac{N}{N-1},2,$ there exist particular solutions of (\ref%
{eq}) on $(0,\infty )$, given by 
\begin{equation}
u^{\ast }(r)=A^{\ast }r^{\frac{q-2}{q-1}}+c,\qquad r=\left\vert x\right\vert
,\quad c\in \mathbb{R},  \label{hli}
\end{equation}%
with $\left\vert A^{\ast }\right\vert ^{q-2}A^{\ast
}=(2-q)((N-1)q-N)\left\vert 2-q\right\vert ^{1-q}\left\vert q-1\right\vert
^{2-q}.$ This function is $\infty $-singular for $q<2:$ positive for $\frac{N%
}{N-1}<q<2,$ such that $\lim_{r\rightarrow 0}u^{\ast }=\infty ,$ and
negative either for $1<q<\frac{N}{N-1}$ where $\lim_{r\rightarrow 0}u^{\ast
}=-\infty ;$ it is a cusp-solution for $q>2$ : $\lim_{r\rightarrow 0}u^{\ast
}=0,$ $\lim_{r\rightarrow 0}u^{\prime \ast }=-\infty .$\medskip

(2) In the radial case, equation (\ref{eq}) only depends on the derivative,
and it can be completely integrated: it reduces to 
\begin{equation*}
W^{\prime }=r^{(N-1)(1-q)}\left\vert W\right\vert ^{q},\text{ where }%
W=-r^{N-1}u^{\prime },
\end{equation*}%
hence there exists a union of disjoint intervals $(\rho ,R)$ where for $%
q\neq \frac{N}{N-1}$ setting $b=\frac{q-1}{(N-1)q-N},$%
\begin{equation}
u^{\prime }=-r^{1-N}(C+br^{N-(N-1)q})^{-\frac{1}{q-1}}<0,\qquad \text{or }%
u^{\prime }=r^{1-N}(C-br^{N-(N-1)q})^{-\frac{1}{q-1}}>0.  \label{eur}
\end{equation}%
Then we deduce $u$ by integration. The case $C=0$ corresponds to the
particular solutions.

When $q=\frac{N}{N-1},$ there exist negative solutions near $0$ with a
logarithmic behaviour, given by 
\begin{equation}
u^{\prime \ast }(r)=r^{1-N}((q-1)\left\vert \ln r\right\vert +C)^{\frac{-1}{%
q-1}}.  \label{logi}
\end{equation}%
From (\ref{eur}) we get the existence of other radial solutions in $\mathbb{R%
}^{N}\backslash \left\{ 0\right\} :\medskip $

$\bullet $ for $\frac{N}{N-1}<q<2$, for any $c\geq 0,k>0,$ there exists
positive solutions of (\ref{eq}) such that 
\begin{equation*}
\lim_{r\rightarrow 0}r^{\frac{2-q}{q-1}}(u-c)=A^{\ast },\qquad
\lim_{r\rightarrow \infty }r^{N-2}(u-c)=k>0,
\end{equation*}%
$\bullet $ for $1<q<\frac{N}{N-1}$, and for any $c<0,k>0,$ there exists a
negative solution $u$ such that 
\begin{equation*}
\lim_{r\rightarrow 0}r^{N-2}(u-c)=-k,\qquad \lim_{r\rightarrow \infty }r^{%
\frac{2-q}{q-1}}(u-c)=-\left\vert A^{\ast }\right\vert ,
\end{equation*}%
$\bullet $ for $q>2,$ and for any $c>d$ there exist decreasing bounded
solutions of (\ref{eq}), such that 
\begin{equation*}
\lim_{r\rightarrow 0}u=c,\quad \lim_{r\rightarrow 0}r^{\frac{2-q}{q-1}%
}(u-c)=-\left\vert A^{\ast }\right\vert ,\qquad \lim_{r\rightarrow \infty
}u=d,
\end{equation*}%
and they are either positive for $d>0$ or negative for $c<0$.\medskip

(3) In the nonradial case, the first upperestimates are due to \cite{Lio}:
if $u$ is any solution (with no condition of sign) in a domain $\Omega ,$
then for any $q>1,$ 
\begin{equation*}
\left\vert \nabla u(x)\right\vert \leq C_{N,q}\text{dist}(x,\partial \Omega
)^{-\frac{1}{q-1}}.
\end{equation*}%
As a consequence, if $\Omega =\mathbb{R}^{N},$ then $u$ is constant. Note
that the result is false if $q<1,$ since $u^{\ast }$ given at (\ref{hli})
still exists and $u^{\ast }\in C^{2}(\mathbb{R}^{N}).$ The estimates have
been extended to a quasilinear equation $-{\Delta }_{m}{u=-\func{div}(}%
\left\vert \nabla u\right\vert ^{m-2}\nabla u)=\left\vert \nabla
u\right\vert ^{q},$ $m>1$ in \cite{BiGaVe}, where one can find a complete
classification of the negative solutions $u\in C^{1}(\Omega \backslash
\left\{ 0\right\} )$ near $0,$ and a partial classification of the positive
ones. The negative solutions of the equation (\ref{eq}) have been studied in 
\cite{NgTa}. When $q\geq 2,$ the first author gives in \cite{Bi2} a detailed
behaviour of the positive solutions in $B_{r_{0}}\backslash \left\{
0\right\} $ or in $\mathbb{R}^{N}\backslash \overline{B_{r_{0}}}$ in a more
general context of equation $-{\Delta }_{m}{u}=u^{p}\left\vert \nabla
u\right\vert ^{q},$ $p\geq 0.$

\subsection{Particular solutions and critical values of the parameters}

When searching particular solutions of system (\ref{one}), some critical
values of the parameters $p,q$ are involved in the problem, showing the
complexity of the system compared to the scalar case :

\begin{definition}
Let $pq>1,p\geq q.$ Following (\ref{com}) and (\ref{comstar}), we define 
\begin{eqnarray*}
q_{1} &=&q_{1}(p)=\frac{N}{(N-1)p-1},\qquad q_{2}=q_{2}(p)=\frac{N+p}{(N-1)p}%
, \\
q_{3} &=&q_{3}(p)=\frac{2}{p-1},\qquad q_{4}=q_{4}(p)=\frac{p+2}{p},\qquad
q^{\ast }=q^{\ast }(p)=\frac{2N+p}{2(N-1)p-1}.
\end{eqnarray*}%
corresponding to five curves in the set $\left\{ (p,q):pq>1,p\geq
q>0\right\} $, namely 
\begin{eqnarray*}
\mathcal{L}_{1} &=&\left\{ q=q_{1}\right\} =\left\{ N-1)pq=N+q\right\}
,\qquad \mathcal{L}_{2}=\left\{ q=q_{2}\right\} =\left\{ N-1)pq=N+p\right\} ,
\\
\mathcal{L}_{3} &=&\left\{ q=q_{3}\right\} =\left\{ q(p-1)=2\right\} ,\qquad 
\mathcal{L}_{4}=\left\{ q=q_{4}\right\} =\left\{ p(q-1)=2\right\} , \\
\mathcal{L}^{\ast } &=&\left\{ q=q^{\ast }\right\} =\left\{ (N-1)pq=N+\frac{%
p+q}{2}\right\} .
\end{eqnarray*}%
Note that $\mathcal{L}_{1},\mathcal{L}_{2},\mathcal{L}^{\ast }$ intersect
the diagonal $\left\{ p=q\right\} $ for $p=\frac{N}{N-1},$ and $\mathcal{L}%
_{3},\mathcal{L}_{4}$ for $p=2.$ And $\mathcal{L}_{2}$ intersect $\mathcal{L}%
_{3}$ for $p=2(N-1),$ $\mathcal{L}^{\ast }$ intersects $\mathcal{L}_{3}$ for 
$p=2(N-1),$ see the figures below in Section \ref{exi}.
\end{definition}

\begin{proposition}
\label{parti}Let $pq>1,p\geq q.$ Then for $q\neq q_{1},q_{2},$ the system (%
\ref{one}) admits particular solutions such that 
\begin{equation}
u_{1}^{\ast \prime }=\eta _{1}a_{1}r^{-\frac{p+1}{pq-1}},\qquad u_{2}^{\ast
\prime }=\eta _{2}a_{2}r^{-\frac{q+1}{pq-1}},  \label{pri}
\end{equation}%
where $a_{i}=a_{i}(N,p,q)>0$ for $i=1,2,$ and 
\begin{equation}
\eta _{1}=\text{sign}(q_{2}-q),\qquad \eta _{2}=\text{sign}(q_{1}-q).
\label{pro}
\end{equation}%
If moreover $q\neq q_{3},q_{4}$, it has solutions 
\begin{equation}
u_{1}^{\ast }=\varepsilon _{1}A_{1}^{\ast }r^{\frac{p(q-1)-2}{pq-1}%
}+c_{1}=\varepsilon _{1}A_{1}^{\ast }r^{\frac{p(q-q_{4})}{pq-1}%
}+c_{1},\qquad u_{2}^{\ast }=\varepsilon _{2}A_{2}^{\ast }r^{\frac{q(p-1)-2}{%
pq-1}}+c_{2}=\varepsilon _{2}A_{2}^{\ast }r^{\frac{(p-1)(q-q_{3})}{pq-1}%
}+c_{2},  \label{fun}
\end{equation}%
where $A_{i}^{\ast }=A_{i}^{\ast }(N,p,q)>0$, $c_{i}\in \mathbb{R}$ for $%
i=1,2,$ and 
\begin{equation*}
\varepsilon _{1}=\text{sign}(q_{2}-q)(q-q_{4}),\qquad \varepsilon _{2}=\text{%
sign}(q_{1}-q)(q-q_{3}).
\end{equation*}
\end{proposition}

\begin{proof}
We search particular solutions of the form $(u_{1}^{\prime },u_{2}^{\prime
})\equiv (C_{1}r^{-\lambda _{1}},C_{2}r^{-\lambda _{2}})$. Then 
\begin{equation*}
C_{1}(\lambda _{1}-N+1)r^{-\lambda _{1}-1}=r^{-p\lambda _{2}}\left\vert
C_{2}\right\vert ^{p},\qquad C_{2}(\lambda _{2}-N+1)r^{-\lambda
_{2}-1}=r^{-q\lambda _{1}}\left\vert C_{2}\right\vert ^{p},
\end{equation*}%
hence we find $\lambda _{1}=\frac{p+1}{pq-1},\lambda _{2}=\frac{q+1}{pq-1},$
and setting 
\begin{equation*}
\beta _{1}=\lambda _{1}-N+1=\frac{N+p-(N-1)pq}{pq-1},\qquad \beta
_{2}=\lambda _{2}-N+1=\frac{N+q-(N-1)pq}{pq-1},
\end{equation*}%
we get the results with $\left\vert C_{1}\right\vert =(\left\vert \beta
_{1}\right\vert \left\vert \beta _{2}\right\vert ^{p})^{\frac{1}{pq-1}%
},\;\left\vert C_{2}\right\vert =(\left\vert \beta _{2}\right\vert
\left\vert \beta _{1}\right\vert ^{q})^{\frac{1}{pq-1}}\mathrm{.}$ By
integration, if the denominator is nonzero, we get in particular the
solution $(u_{1},u_{2}),$ with $\left\vert A_{i}^{\ast }\right\vert
=\left\vert \frac{C_{i}(pq-1)}{p(q-1)-2}\right\vert ,$ $i=1,2.$
\end{proof}

\begin{remark}
\label{2.3}Note that $u_{1}^{\ast }$ is $\infty $-singular for $q<q_{4}$ and
it is a cusp-solution for $q>q_{4}.$ And $u_{2}^{\ast }$ is $\infty $%
-singular for $q<q_{3}$ and it is a cusp-solution for $q>q_{3};$ so there
exist solutions $(u_{1},u_{2})$ of system (\ref{one}) with components of
different types when $q_{3}<q<q_{4}.$ $\medskip $
\end{remark}

\begin{remark}
\label{rela}It is clear that the exponents $q_{1},q_{2},q_{3},q_{4},$ are
strongly involved in the existence of particular solutions. The exponent $%
q^{\ast }$ corresponds to the Sobolev exponent for equation (\ref{hql}),
offering other types of explicit solutions, see Theorem \ref{regA}. Note the
relations 
\begin{eqnarray*}
q &\leq &q_{1}\Longleftrightarrow (N-1)pq\leq N+q,\qquad q\leq
q_{2}\Longleftrightarrow (N-1)pq\leq N+p, \\
q &\leq &q_{3}\Longleftrightarrow q(p-1)\leq 2,\qquad q\leq
q_{4}\Longleftrightarrow p(q-1)\leq 2, \\
q &\leq &q^{\ast }\Longleftrightarrow (N-1)pq\leq N+\frac{p+q}{2},
\end{eqnarray*}%
\begin{eqnarray*}
q_{1} &\leq &\min (q_{2},q_{3})\leq \max (q_{2},q_{3})\leq q_{4}, \\
q_{1} &\leq &\frac{N}{N-1}\Longleftrightarrow q_{2}\leq \frac{N}{N-1}%
\Longleftrightarrow p\geq \frac{N}{N-1}, \\
q_{2} &\leq &q_{3}\Longleftrightarrow p\leq N, \\
q_{1} &\leq &q^{\ast }\leq q_{2},\qquad q^{\ast }\leq
q_{3}\Longleftrightarrow p\leq 2(N-1).
\end{eqnarray*}%
When $q=p,$ we find again the two critical values of the scalar case: 
\begin{equation*}
q_{1}=q^{\ast }=q_{2}=\frac{N}{N-1},\qquad q_{3}=q_{4}=2.
\end{equation*}
\end{remark}

\section{First properties in the radial case\label{rad}}

In the radial case, system (\ref{one}) is reduced to 
\begin{equation}
\left\{ 
\begin{array}{c}
{-(u}_{1}^{\prime \prime }+\frac{N-1}{r}{u}_{1}^{\prime })=\left\vert
u_{2}^{\prime }\right\vert ^{p}, \\ 
{-(u}_{2}^{\prime \prime }+\frac{N-1}{r}{u}_{2}^{\prime })=\left\vert
u_{1}^{\prime }\right\vert ^{q},%
\end{array}%
\right.  \label{SA}
\end{equation}%
so it only involves the derivatives. Moreover, setting 
\begin{equation}
w_{1}=-r^{N-1}{u}_{1}^{\prime },\qquad w_{2}=-r^{N-1}{u}_{2}^{\prime },
\label{ww}
\end{equation}%
system (\ref{SA}) is equivalent to a first order system: 
\begin{equation}
\left\{ 
\begin{array}{c}
w_{1}^{\prime }=r^{(N-1)(1-p)}\left\vert w_{2}\right\vert ^{p}, \\ 
w_{2}^{\prime }=r^{(N-1)(1-q)}\left\vert w_{1}\right\vert ^{q}.%
\end{array}%
\right.  \label{Sw}
\end{equation}%
Consider any solution $({u}_{1},{u}_{2})\in C^{2}(0,r_{0}),$ $r_{0}\leq
\infty .$ We say that $u_{i}$ is \textit{regular} if $u_{i}^{\prime }$ has a
limit at $0$ and $\lim_{r\longrightarrow 0}u_{1}^{\prime
}=\lim_{r\longrightarrow 0}u_{2}^{\prime }=0$; if $u_{1}$ and $u_{2}$ are
regular, then $({u}_{1},{u}_{2})$ extends as a function in $%
C^{2}(B_{r_{0}}\times B_{r_{0}})$, from the equations. We say that $u_{i}$
is \textit{singular} if it is not regular. Among the singular solutions we
distinguish the \textbf{\ }$\infty $\textit{-singular} ones, such that $%
\lim_{r\longrightarrow 0}\left\vert u_{i}\right\vert =\infty ,$ and the 
\textit{cusp-solutions} $u_{i}$ ($\lim_{r\longrightarrow 0}u_{i}=c\in 
\mathbb{R}$ and $\lim_{r\longrightarrow 0}\left\vert u_{i}^{\prime
}\right\vert =\infty ),$ where the singularity is at the level of the
gradient.

\subsection{Local existence and uniqueness}

\begin{proposition}
\label{locun}Let $pq>1,p\geq q.$ (i) Let $r_{0}\geq 0$ and $%
c_{1},c_{2},b_{1},b_{2}\in \mathbb{R}$ such that $b_{2}\neq 0$ if $q<1.$
Then there exists a unique solution $(u_{1},u_{2})$ of system (\ref{SA}) in
the neighborhood of $r_{0},$ such that 
\begin{equation*}
u_{1}(r_{0})=c_{1},\qquad u_{2}(r_{0})=c_{2},\qquad u_{1}^{\prime
}(r_{0})=b_{1},\qquad u_{2}^{\prime }(r_{0})=b_{2}.
\end{equation*}%
(ii) For any $R>0,$ the only radial solutions $(u_{1},u_{2})\in
C^{1}(B_{R}\times B_{R})$ of system (\ref{one}) are the constants.
\end{proposition}

\begin{proof}
(i) First suppose $r_{0}>0.$ Since the system only depends on the
derivatives, it is equivalent to consider system (\ref{Sw}), with initial
data $(w_{01},w_{0,2})\in \mathbb{R}^{2}$ If $q>1,$ the classical
Cauchy-Lipschitz theorem applies. Next suppose $q<1;$ the theorem still
applies if $w_{01}\neq 0$. Consider the case $w_{01}=0,w_{0,2}\neq 0$. In a
neighborhood $\left[ r_{0}-\eta ,r_{0}+\eta \right] $ of $r_{0}.$ there
holds $0<C_{1}<\left\vert w_{2}\right\vert <C_{2},$ $C_{3}<w_{1}^{\prime
}<C_{4},$ $0<\left\vert w_{1}\right\vert <C_{5}.$ hence we can take $w_{1}$
as a new variable, and consider that $r$ is a function of $w_{1}.$ Then we
get the system 
\begin{equation*}
\left\{ 
\begin{array}{c}
\frac{dr}{dw_{1}}=f(w_{1},r,w_{2})=r^{(N-1)(1-p)}\left\vert w_{2}\right\vert
^{p}, \\ 
\frac{dw_{2}}{dw_{1}}=g(w_{1},r,w_{2})=\left\vert w_{1}\right\vert
^{q}r^{(N-1)(2-p-q)}\left\vert w_{2}\right\vert ^{p},%
\end{array}%
\right.
\end{equation*}%
for which we can apply the Cauchy-Lipschitz theorem in $\left[
0,C_{5}\right) \times \left( r_{0}-\eta ,r_{0}+\eta \right) \times
(C_{1},C_{2})$, since $f$ is continuous, and locally Lipschitz with respect
to $(r,w_{2}).\medskip $

(ii) Next suppose that $b_{1}=b_{2}=0,$ and $r_{0}\geq 0.$ Then 
\begin{equation*}
u_{1}^{\prime }(r)=r^{1-N}\dint_{r_{0}}^{r}s^{N-1}\left\vert u_{2}^{\prime
}\right\vert ^{p}ds,\qquad u_{2}^{\prime
}(r)=r^{1-N}\dint_{r_{0}}^{r}s^{N-1}\left\vert u_{1}^{\prime }\right\vert
^{q}ds,
\end{equation*}%
\begin{equation}
\left\vert u_{1}^{\prime }(r)\right\vert \leq r^{1-N}\left\vert
\dint_{r_{0}}^{r}s^{N-1}\left\vert u_{2}^{\prime }\right\vert
^{p}ds\right\vert \leq r^{1-N}\left\vert
\dint_{r_{0}}^{r}s^{N-1}(s^{1-N}\left\vert \dint_{r_{0}}^{s}\theta
^{N-1}\left\vert u_{1}^{\prime }\right\vert ^{q}d\theta \right\vert
)^{p})ds\right\vert ,  \label{cil}
\end{equation}%
$\bullet $ First assume $r_{0}>0.$ In a neigborhood $\left[ r_{0}-\eta
,r_{0}+\eta \right] $, there holds 
\begin{equation}
\left\vert u_{1}^{\prime }(r)\right\vert \leq C\left\vert
\dint_{r_{0}}^{r}\left\vert \dint_{r_{0}}^{s}\left\vert u_{1}^{\prime
}\right\vert ^{q})^{p}d\theta \right\vert ds\right\vert \leq C^{\prime
}\dint_{r_{0}}^{r}\dint_{r_{0}}^{s}\left\vert u_{1}^{\prime }\right\vert
^{pq}d\theta )ds\leq C^{\prime \prime }\left\vert
\dint_{r_{0}}^{r}\left\vert u_{1}^{\prime }\right\vert ^{pq}d\theta
\right\vert .  \label{col}
\end{equation}%
Let $F(r)=\dint_{r_{0}}^{r}\left\vert u_{1}^{\prime }\right\vert
^{pq}d\theta .$ Then $F^{\prime }(r)\leq C^{\prime \prime }F^{pq}$ on $%
(r_{0},r_{0}+\eta ).$ If $u_{1}^{\prime }$ is not identically $0$ on $%
(r_{0},r_{0}+\eta ),$ then $F>0$ on $(r_{0},r_{0}+\rho )$ for some $\rho $
small enough, and $F^{1-pq}+cr$ is nondecreasing, which is contradictory,
since $F(r_{0})=0$; thus $u_{1}^{\prime }\equiv 0$ and then $u_{2}^{\prime
}\equiv 0$ on $(r_{0},r_{0}+\rho ),$ and similarly on $(r_{0}-\rho ,r_{0})$
for $\rho $ small enough. Then $u_{1}\equiv c_{1},u_{2}\equiv c_{2}$ in $%
\left[ r_{0}-\eta ,r_{0}+\eta \right] .\medskip $

\noindent $\bullet $ Next assume $r_{0}=0.$ Then from (\ref{cil}), and since 
$p>1,$ 
\begin{eqnarray*}
\left\vert u_{1}^{\prime }(r)\right\vert  &\leq
&r^{1-N}\dint_{0}^{r}s^{(N-1)(1-p)}(\dint_{0}^{s}\theta ^{N-1}\left\vert
u_{1}^{\prime }\right\vert ^{q}d\theta )^{p}ds \\
&\leq &r^{1-N}\dint_{0}^{r}s^{(N-1)(1-p)+(N-1)p}(\dint_{0}^{s}\left\vert
u_{1}^{\prime }\right\vert ^{q}d\theta )^{p}ds\leq
\dint_{0}^{r}(\dint_{0}^{s}\left\vert u_{1}^{\prime }\right\vert ^{q}d\theta
)^{p}ds \\
&\leq &\dint_{0}^{r}s^{p-1}(\dint_{0}^{s}\left\vert u_{1}^{\prime
}\right\vert ^{pq}d\theta )ds\leq r^{p}\dint_{0}^{r}\left\vert u_{1}^{\prime
}\right\vert ^{pq}d\theta \leq C\dint_{0}^{r}\left\vert u_{1}^{\prime
}\right\vert ^{pq}d\theta .
\end{eqnarray*}%
Defining $F(r)=\dint_{0}^{r}\left\vert u_{1}^{\prime }\right\vert
^{pq}d\theta $ we conclude as above that $u_{1}^{\prime }\equiv
0,u_{2}^{\prime }\equiv 0$ on $\left( 0,\rho \right) $ for $\rho $ small
enough. Finally let $(u_{1},u_{2})\in C^{1}(B_{R}\times B_{R})$ be any
radial solution. Then the set $\left\{ (0,\rho ):u_{1}^{\prime }\equiv
u_{2}^{\prime }\equiv 0\text{ on }\left( 0,\rho \right) \right\} $ is closed
and open in $\left( 0,R\right) ,$ thus equal to $(0,R).$
\end{proof}

\begin{remark}
The result would be false for $pq<1.$ Indeed in that case there exist
solutions of system (\ref{abs}): 
\begin{equation*}
\widetilde{u_{1}}^{\ast }=A_{1}^{\ast }r^{\frac{2+p(1-q)}{1-pq}}=A_{1}^{\ast
}r^{2+\frac{p(q+1)}{1-pq}},\qquad \widetilde{u_{2}}^{\ast }=A_{2}^{\ast
}r^{2+\frac{q(p+1)}{1-pq}},
\end{equation*}%
belonging to $C^{2}(\mathbb{R}^{N})$ which proves the existence of entire
nontrivial solutions, and gives a counter-example to uniqueness.\medskip
\end{remark}

\begin{remark}
Note a consequence of Proposition \ref{locun}: for any $r_{0}>0,$ there
exist local radial solutions $(u_{1},u_{2})$ such that $u_{1}^{\prime
}(r_{0})=0,$ and $u_{2}^{\prime }(r_{0})=b\neq 0.$ This shows the great
richness of the solutions of the system: for system (\ref{sou}) (resp. (\ref%
{abs}) for example, it means that only one of the functions has a minimum
(resp. a maximum) at $r_{0}.\ \medskip $
\end{remark}

\begin{remark}
\label{tre}From Proposition \ref{locun}, one can divide the maximal
existence interval $(\rho ,R)$ of $(u_{1},u_{2})$ in intervals $(a,b)$ where 
$u_{1}^{\prime }$ and $u_{2}^{\prime }$ have a constant sign. Indeed there
exists at most one point $\rho <s<R$ where $w_{1}(s)=0;$ $w_{1}<0$ on $(\rho
<s)$ and $w_{1}>0$ on $(s<R);$ in the same way, there is at most one point $%
\rho <\tau <R$ where $w_{2}(\tau )=0;$ $w_{2}<0$ on $(\rho ,\tau )$ and $%
w_{1}>0$ on $(\tau ,R).$
\end{remark}

\subsection{Upperestimates on the radial supersolutions}

\begin{proposition}
\label{uprad}Let $pq>1,$ and $(u_{1},u_{2})$ be any radial supersolution of
system (\ref{one}), defined on an interval $(0,r_{0})$ (resp. $(r_{0},\infty
)$ . There exists a constant $C=C(N,p,q)>0$ such that for $r>0$ small enough
(resp. large enough), 
\begin{equation*}
\left\vert u_{1}^{\prime }(r)\right\vert \leq Cr^{-\frac{p+1}{pq-1}},\qquad
\left\vert u_{2}^{\prime }(r)\right\vert \leq Cr^{-\frac{q+1}{pq-1}}.
\end{equation*}
\end{proposition}

\begin{proof}
Consider any supersolution of system (\ref{one}). Equivalently, the
functions $w_{1},w_{2}$ defined by (\ref{ww}) satisfy 
\begin{equation*}
\left\{ 
\begin{array}{c}
w_{1}^{\prime }\geq r^{(N-1)(1-p)}\left\vert w_{2}\right\vert ^{p}, \\ 
w_{2}^{\prime }\geq r^{(N-1)(1-q)}\left\vert w_{1}\right\vert ^{q}.%
\end{array}%
\right.
\end{equation*}%
As in Remark \ref{tre}, $w_{1},w_{2}$ do not vanish for $r$ small enough
(resp. large enough), depending on $u_{1},u_{2}.$ Consider any interval $%
(R_{1},R_{2})$ where $w_{1},w_{2}$ do not vanish. Let $r_{0}\in
(R_{1},R_{2}) $ and $\varepsilon _{0}\in \left( 0,\frac{1}{4}\right] $ such
that $R_{1}<r_{0}(1-2\varepsilon _{0})<r_{0}(1+2\varepsilon _{0})<R_{2}.$
First assume that $w_{1}>0,w_{2}>0.$ Since $w_{i}$ is nondecreasing, for any 
$\varepsilon \in \left( 0,\varepsilon _{0}\right] ,$ we have $r\in
(r_{0}(1-\varepsilon )<r_{0}(1+\varepsilon ))$ there holds, since $w_{1}>0$
and $\varepsilon _{0}<1,$ with $c_{i}=c_{i}(N,p)$ 
\begin{eqnarray*}
w_{1}(r_{0}) &\geq &\int_{r_{0}(1-\varepsilon )}^{r_{0}}\theta
^{(N-1)(1-p)}w_{2}^{p}(\theta )d\theta \\
&\geq &w_{2}^{p}(r_{0}(1-\varepsilon ))\int_{r_{0}(1-\varepsilon
)}^{r_{0}}\theta ^{(N-1)(1-p)}d\theta =c_{1}\varepsilon
w_{2}^{p}(r_{0}(1-\varepsilon ))r_{0}^{N-(N-1)p},
\end{eqnarray*}%
and symmetrically, since $w_{1}>0,$ changing $r_{0}$ into $r_{0}-\varepsilon
,$ 
\begin{equation*}
w_{2}(r_{0}-\varepsilon )\geq c_{2}\varepsilon
w_{1}^{q}(r_{0}(1-2\varepsilon ))(r_{0}(1-\varepsilon ))^{N-(N-1)q}\geq
c_{3}\varepsilon w_{1}^{q}(r_{0}-2\varepsilon )r_{0}^{N-(N-1)q}.
\end{equation*}%
Then 
\begin{equation*}
w_{1}(r_{0})\geq c_{4}\varepsilon ^{p+1}w_{1}^{pq}(r_{0}(1-2\varepsilon
))r_{0}^{N+p-(N-1)pq},
\end{equation*}%
\begin{equation*}
w_{1}(r_{0}(1-2\varepsilon ))\leq c_{5}\varepsilon ^{-\frac{p+1}{pq}%
}(w_{1}(r_{0})^{\frac{1}{pq}}r_{0}^{-\frac{N+p-(N-1)pq}{pq}}.
\end{equation*}%
Setting $\overline{r_{0}}=r_{0}(1-2\varepsilon ),$ and using by the
bootstrap technique developed at \cite[Lemma 2.2]{BiGr}, \cite[Lemma 2.8]%
{BiGaYa} and \cite{Bi3}, we obtain, since $d=\frac{1}{pq}<1,$ and 
\begin{equation*}
w_{1}(r_{0})=r_{0}^{N-1}\left\vert u_{1}^{\prime }(r_{0})\right\vert \leq
c_{6}r_{0}^{\frac{N+p-(N-1)pq)}{1-d}}=c_{6}r_{0}^{-\frac{N+p-(N-1)pq}{pq-1}},
\end{equation*}%
equivalently%
\begin{equation*}
\left\vert u_{1}^{\prime }(r_{0})\right\vert \leq c_{7}r_{0}^{-\frac{p+1}{%
pq-1}}.
\end{equation*}%
If $w_{1}<0,w_{2}<0,$ the same conclusion holds by changing $1+\varepsilon $
into $1-\varepsilon .$ Finally if $w_{2}>0>w_{1},$ we get 
\begin{equation*}
w_{1}(r_{0})\geq c_{1}\varepsilon w_{2}^{p}(r_{0}(1-\varepsilon
))r_{0}^{N-(N-1)p},
\end{equation*}%
\begin{equation*}
w_{2}(r_{0}(1-\varepsilon )\geq c_{2}\varepsilon
w_{1}^{q}(r_{0})(r_{0}(1-\varepsilon ))^{N-(N-1)q}\geq c_{3}\varepsilon
w_{1}^{q}(r_{0})r_{0}^{N-(N-1)q},
\end{equation*}%
\begin{equation*}
w_{1}(r_{0})\geq c_{4}\varepsilon ^{p+1}w_{1}^{pq}(r_{0})r_{0}^{N+p-(N-1)pq},
\end{equation*}%
hence we obtain the same conclusion without using any bootstrap.
\end{proof}

\subsection{Reduction to a quadratic system of order 2 and formulation as a
Hardy-H\'{e}non equation}

In the sequel we show that system (\ref{SA}) can be reduced to a polynomial
system order 2, and equivalently to an equation of order 2 relative to \ $%
w_{1}:$

\begin{proposition}
\label{bigi}Let $(u_{1},u_{2})$ be any radial solution of system (\ref{one}%
). At any point $r$ where $u_{1}^{\prime }(r)\neq 0,$ $u_{2}^{\prime
}(r)\neq 0,$ we define 
\begin{equation}
S=r\frac{\left\vert u_{2}^{\prime }\right\vert ^{p}}{u_{1}^{\prime }},\qquad
Z=-r\frac{\left\vert u_{1}^{\prime }\right\vert ^{q}}{u_{2}^{\prime }}.
\label{dif}
\end{equation}%
Then system (\ref{SA}) is equivalent to a Lodka-Volterra type system: 
\begin{equation}
\left\{ 
\begin{array}{c}
S_{t}=S(N-(N-1)p+S+pZ), \\ 
Z_{t}=Z(N-(N-1)q-qS-Z).%
\end{array}%
\right. \qquad t=\ln r,  \label{SZ}
\end{equation}%
and we recover $u_{1}^{\prime }$ and $u_{2}^{\prime }$ in function of $r,S,Z$
by the formulas 
\begin{equation}
u_{1}^{\prime }=(r^{-(p+1)}\left\vert S\right\vert \left\vert Z\right\vert
^{p})^{\frac{1}{pq-1}}\text{sign}S,\qquad u_{2}^{\prime
}=-(r^{-(q+1)}\left\vert S\right\vert ^{q}\left\vert Z\right\vert )^{\frac{1%
}{pq-1}}\text{sign}Z.  \label{reco}
\end{equation}
\end{proposition}

\begin{proof}
Following the ideas of \cite{BiGi}, at each point where $u_{1}^{\prime }\neq
0$ and $u_{2}^{\prime }\neq 0$, we set 
\begin{equation*}
X=r\frac{u_{1}^{\prime \prime }}{u_{1}^{\prime }},\qquad Y=r\frac{%
u_{2}^{\prime \prime }}{u_{2}^{\prime }},\qquad S=r\frac{\left\vert
u_{2}^{\prime }\right\vert ^{p}}{u_{1}^{\prime }},\qquad Z=-r\frac{%
\left\vert u_{1}^{\prime }\right\vert ^{q}}{u_{2}^{\prime }},\qquad t=\ln r,
\end{equation*}%
Then 
\begin{equation*}
X+N-1=-S,\qquad Y+N-1=Z.
\end{equation*}%
Next we differentiate with respect to $t.$ Since $X=\frac{(u_{1}^{\prime
})_{t}}{u_{1}^{\prime }}.$ $Y=\frac{(u_{2}^{\prime })_{t}}{u_{2}^{\prime }},$
we obtain 
\begin{equation*}
\frac{S_{t}}{S}=1+p\frac{(u_{2}^{\prime })_{t}}{u_{2}^{\prime }}-\frac{%
(u_{1}^{\prime })_{t}}{u_{1}^{\prime }}=1+pY-X=N-(N-1)p+S+pZ,
\end{equation*}%
\begin{equation*}
\frac{Z_{t}}{Z}=1+q\frac{(u_{1}^{\prime })_{t}}{u_{1}^{\prime }}-\frac{%
(u_{2}^{\prime })_{t}}{u_{2}^{\prime }}=1+qX-Y=N-(N-1)q-qS-Z.
\end{equation*}%
So we obtain the quadratic system (\ref{SZ}) valid in any case of sign of
the unknown, and deduce (\ref{reco}).
\end{proof}

\begin{remark}
The fixed points of the system are 
\begin{equation}
\begin{array}{c}
M_{0}=(S_{0},Z_{0})=(\frac{(N-1)p(q_{2}-q)}{pq-1},-\frac{((N-1)p-1)(q_{1}-q)%
}{pq-1}), \\ 
(0,0),\quad \quad N_{0}=(0,N-(N-1)q),\quad \quad A_{0}=((N-1)p-N,0).%
\end{array}
\label{fixs}
\end{equation}%
We easily check that the particular solutions $(u_{1}^{\prime
},u_{2}^{\prime })$ of system (\ref{SA}) given at (\ref{pri}) correspond to
the fixed point $M_{0}$.
\end{remark}

Next we show the precise link with a Hardy-H\'{e}non equation:\medskip

When looking at the system (\ref{SZ}) for $pq\neq 1,$ and comparing to the
systems introduced in \cite[Section 3]{BiGi}, we observe that the system is
exactly linked to the positive solutions $w$ of a radial quasilinear
equation of Hardy-H\'{e}non type in dimension $\mathbf{N}$, of the form 
\begin{equation}
-\Delta _{\mathbf{p}}^{\mathbf{N}}w=-(\left\vert w^{\prime }\right\vert ^{%
\mathbf{p}-2}w^{\prime })^{\prime }-\frac{\mathbf{N}-1}{r}\left\vert
w^{\prime }\right\vert ^{\mathbf{p}-2}w^{\prime }=\varepsilon r^{\sigma }w^{%
\mathbf{q}},\qquad \varepsilon =\pm 1,  \label{eqo}
\end{equation}%
where $q>0,\sigma \in \mathbb{R},$ $\mathbf{p}>1,$ $\mathbf{q}\neq \mathbf{p}%
-1,$and $\mathbf{N}$ is not necessarily an integer. Indeed in \cite[Section 3%
]{BiGi}, this equation is reduced to a system of order 2, valid for $%
\varepsilon =\pm 1,$ by setting 
\begin{equation}
s(t)=-r\frac{w^{\prime }}{w},\qquad z(t)=-\varepsilon r^{1+\sigma }w^{%
\mathbf{q}}\left\vert w^{\prime }\right\vert ^{-\mathbf{p}}w^{\prime
},\qquad t=\ln r,  \label{siga}
\end{equation}%
and obtained the system%
\begin{equation}
\left\{ 
\begin{array}{c}
s_{t}=s(\frac{\mathbf{p}-\mathbf{N}}{\mathbf{p}-1}+s+\frac{z}{\mathbf{p}-1}),
\\ 
z_{t}=z(\mathbf{N}+\sigma -\mathbf{q}s-z),%
\end{array}%
\right.  \label{psz}
\end{equation}%
and one recovers $w$ by the formula%
\begin{equation}
w=r^{-\gamma }(\left\vert s^{\mathbf{p}-1}\right\vert \left\vert
z\right\vert )^{\frac{1}{\mathbf{q}+1-\mathbf{p}}},\qquad \gamma =\frac{%
\mathbf{p}+\sigma }{\mathbf{q}+1-\mathbf{p}}.  \label{recw}
\end{equation}%
It is precisely the case, as we show below.

\begin{proposition}
\label{nom}Let $(u_{1},u_{2})$ be any radial solution of system (\ref{one}),
and let $w_{1},w_{2}$ be defined by (\ref{ww}). Then the function $%
w=\left\vert w_{1}\right\vert $ satisfies the equation%
\begin{equation}
-\Delta _{\mathbf{p}}^{\mathbf{N}}w=\varepsilon r^{\sigma }w^{\mathbf{q}%
},\qquad \varepsilon =-\text{sign}(u_{1}^{\prime }u_{2}^{\prime })
\label{eqm}
\end{equation}%
where $\Delta _{\mathbf{p}}^{\mathbf{N}}$ is the $\mathbf{p}$-Laplace
operator in dimension $\mathbf{N},$ for specific values of $\mathbf{p}%
,\sigma $ and $\mathbf{N}:$%
\begin{equation}
\mathbf{p}=1+\frac{1}{p}>1,\qquad \mathbf{q}=q,\qquad \sigma =(N-1)\frac{1-pq%
}{p},\qquad \mathbf{N}=1+\frac{(N-1)(p-1)}{p}.  \label{mns}
\end{equation}%
Moreover at each point where $w\neq 0,$ 
\begin{equation}
S(t)=s(t)=-r\frac{w_{1}^{\prime }}{w_{1}},\qquad Z(t)=z(t)=r\frac{%
w_{2}^{\prime }}{w_{2}}=-\text{sign}(u_{2}^{\prime })\;r^{1+\sigma
}\left\vert w_{1}\right\vert ^{q}\left\vert w_{1}^{\prime }\right\vert ^{%
\frac{1}{p}}w_{1}^{\prime }  \label{res}
\end{equation}
\end{proposition}

\begin{proof}
Computing $w_{2}$ from the first equation of (\ref{Sw}) and reporting in the
second one, we get 
\begin{equation*}
-(r^{\frac{(N-1)(p-1)}{p}}(w_{1}^{\prime })^{\frac{1}{p}})^{\prime }=\text{%
sign}(u_{2}^{\prime })\;r^{(N-1)(1-q)}\left\vert w_{1}\right\vert ^{q},
\end{equation*}%
that is in a developed form: 
\begin{equation*}
-((w_{1}^{\prime })^{\frac{1}{p}})^{\prime }-\frac{(N-1)(p-1)}{pr}%
(w_{1}^{\prime })^{\frac{1}{p}}=\text{sign}(u_{2}^{\prime })\;r^{(N-1)\frac{%
1-pq}{p}}\left\vert w_{1}\right\vert ^{q}
\end{equation*}%
So we get an equation with the form (where we recall that $w_{1}^{\prime }>0$%
)%
\begin{equation*}
-(\left\vert w_{1}^{\prime }\right\vert ^{\mathbf{p}-2}w_{1}^{\prime
})^{\prime }-\frac{\mathbf{N}-1}{r}\left\vert w_{1}^{\prime }\right\vert ^{%
\mathbf{p}-2}w_{1}^{\prime })=-\Delta _{\mathbf{p}}^{\mathbf{N}}w_{1}=\text{%
sign}(u_{2}^{\prime })\;r^{\sigma }\left\vert w_{1}\right\vert ^{q}
\end{equation*}%
where $\mathbf{p},\sigma ,\mathbf{N}$ are defined at (\ref{mns}); and the
fact that $w_{1}=-$sign$(u_{1}^{\prime })\;w$ leads to equation (\ref{eqm}),
and we check easily that $s\equiv S$ and $z\equiv Z.$
\end{proof}

\begin{remark}
\label{comp}From the definitions (\ref{mns}) and (\ref{recw}) we get the
relations 
\begin{equation}
\frac{\mathbf{p}-\mathbf{N}}{\mathbf{p}-1}=N-(N-1)p,\qquad \mathbf{N}+\sigma
=N-(N-1)q,  \label{red}
\end{equation}%
\begin{equation}
\gamma =\frac{\mathbf{p}+\sigma }{q-\mathbf{p}+1}=\frac{p+N-(N-1)pq}{pq-1}=%
\frac{(N-1)p(q_{2}-q)}{pq-1}  \label{red2}
\end{equation}%
so 
\begin{equation}
\mathbf{N}>\mathbf{p}\Longleftrightarrow p>\frac{N}{N-1},\qquad \mathbf{N}%
+\sigma >0\Longleftrightarrow q<\frac{N}{N-1},\qquad \mathbf{p}+\sigma
>0\Longleftrightarrow q<q_{2}.  \label{reg}
\end{equation}
\end{remark}

\begin{remark}
For $pq>1,$ $p\geq q,$ there holds $1<\mathbf{p}<2,$ $\mathbf{q}>\mathbf{p}%
-1,$ and $\mathbf{N}>1.$ The map $(N,p)\in \left[ 1,\infty \right) \times
(1,\infty )\longmapsto (\mathbf{N},\mathbf{p})\in (1,\infty )\times (1,2)$
is injective, the reciprocal application is $(\mathbf{N},\mathbf{p}%
)\longmapsto (\frac{\mathbf{N}+1-\mathbf{p}}{2-\mathbf{p}},\frac{1}{\mathbf{p%
}-1}),$ and then $\sigma $ is \textbf{fixed} by the relation $\sigma =-\frac{%
(\mathbf{N}-1)(q+1-\mathbf{p})}{2-\mathbf{p}}.\ \medskip $
\end{remark}

\begin{remark}
\label{souab}Equation (\ref{eqm}) is of source type for $\varepsilon =1,$ of
absorption type for $\varepsilon =-1.$ One could think that problem (\ref%
{sou}), where $u_{1},u_{2}$ are positive superharmonic functions, with two 
\textbf{source }terms $\left\vert \nabla u_{2}\right\vert ^{p},\left\vert
\nabla u_{1}\right\vert ^{q}$ is linked to a Hardy-H\'{e}non equation with
source term $\varepsilon =1.$ In fact \textbf{it is not the case}: the
solutions of system (\ref{sou}) on an interval $(0,\rho )$ correspond to a
Hardy-H\'{e}non equation of \textbf{absorption} type. Indeed consider any
positive solution of system (\ref{sou}). As it is well known, any positive
solution $(u_{1},u_{2})$ satisfies $u_{1}^{\prime },u_{2}^{\prime }\leq 0$
in $(0,\rho ).$ Indeed $r^{N-1}u_{1}^{\prime }$ is decreasing; if there is $%
r_{0}$ such that $u_{1}^{\prime }(r_{0})>0,$ then $r^{N-1}u_{1}^{\prime
}\geq C_{0}>0$ on $(0,r_{0}),$ thus $u_{1}+\frac{C_{0}}{N-2}r^{2-N}$ is
increasing, which is impossible. Then $u_{1}^{\prime }\leq 0$ on $(0,\rho );$
moreover if there exists $r_{1}$ such that $u_{1}^{\prime }(r_{1})=0,$ then
it is unique and $r_{1}$ is a maximum point, so that $u_{1}^{\prime }\geq 0$
on $(0,r_{1})$ which is contradictory, unless $u_{1}$ is constant. Then the
nonconstant solutions satisfy\textbf{\ }$u_{1}^{\prime },u_{2}^{\prime }<0$%
\textbf{\ }in\textbf{\ }$(0,\rho ).$\medskip
\end{remark}

\begin{remark}
Another way to get an autonomous system is more common in the literature:
the change of unknown 
\begin{equation*}
u_{1}^{\prime }=-r^{-\frac{p+1}{pq-1}}x(t),\qquad u_{2}^{\prime }=-r^{-\frac{%
q+1}{pq-1}}y(t),\qquad t=\ln r.
\end{equation*}%
leads to the system%
\begin{equation*}
\left\{ 
\begin{array}{c}
x_{t}=b_{1}x-\left\vert y\right\vert ^{p}, \\ 
y_{t}=b_{2}y-\left\vert x\right\vert ^{q},%
\end{array}%
\right.
\end{equation*}%
where $b_{1}=\frac{(N-1)pq-(p+N)}{pq-1}$ and $b_{2}=\frac{(N-1)pq-(q+N)}{pq-1%
}.$ However this system gives less information on the solutions of system (%
\ref{SA}): it admits at most two fixed points, namely $(0,0),$ and $%
P_{0}=((\left\vert b_{1}\right\vert \left\vert b_{2}\right\vert ^{p})^{\frac{%
1}{pq-1}}$sign$(b_{1}),(\left\vert b_{2}\right\vert \left\vert
b_{1}\right\vert ^{q})^{\frac{1}{pq-1}}$sign$(b_{2}))$ which corresponds to
the particular solutions. Moreover it is singular at $(0,0)$ whenever $q<1.$
The quadratic system (\ref{SZ}) gives a great amount of information, because
it is obtained by differentiation of the equations of (\ref{SA}). It has
four fixed points, and each of them corresponds to a type of behaviour near $%
0$ or $\infty $.
\end{remark}

\section{Solutions of the Hardy-H\'{e}non equations \label{app}}

Here we consider the \textbf{positive} solutions of the radial equation%
\begin{equation}
-\Delta _{\mathbf{p}}^{\mathbf{N}}w=-\frac{d}{dr}(\left\vert \frac{dw}{dr}%
\right\vert ^{\mathbf{p}-2}\frac{dw}{dr})-\frac{\mathbf{N}-1}{r}\left\vert 
\frac{dw}{dr}\right\vert ^{\mathbf{p}-2}\frac{dw}{dr}=\varepsilon r^{\sigma
}w^{\mathbf{q}},\qquad \varepsilon =\pm 1,  \label{sca}
\end{equation}%
in dimension $\mathbf{N},$ where $\mathbf{q}>\mathbf{p}-1>0.$ In the sequel, 
$\mathbf{N}$ is not necessarily an integer.$\bigskip $

\subsection{General formulation by a quadratic system}

0n any interval where $w^{\prime }\neq 0,$ we define 
\begin{equation}
s(t)=-r\frac{w^{\prime }}{w},\qquad z(t)=-\varepsilon r^{1+\sigma }w^{%
\mathbf{q}}\left\vert w^{\prime }\right\vert ^{-\mathbf{p}}w^{\prime
},\qquad t=\ln r,  \label{fid}
\end{equation}%
and obtain the system, valid for the two equations, 
\begin{equation}
\left\{ 
\begin{array}{c}
s_{t}=s(\frac{\mathbf{p}-\mathbf{N}}{\mathbf{p}-1}+s+\frac{z}{\mathbf{p}-1}),
\\ 
z_{t}=z(\mathbf{N}+\sigma -\mathbf{q}s-z).%
\end{array}%
\right.  \label{pspz}
\end{equation}%
We recover $w$ by the formula%
\begin{equation}
w=r^{-\gamma }(\left\vert s\right\vert ^{\mathbf{p}-1}\left\vert
z\right\vert )^{\frac{1}{\mathbf{q}+1-\mathbf{p}}},\qquad w^{\prime
}=r^{-(\gamma +1)}(\left\vert z\right\vert \left\vert s\right\vert ^{\mathbf{%
q}})^{\frac{1}{\mathbf{q}+1-\mathbf{p}}}\text{sign}(-\varepsilon z),\quad
\gamma =\frac{\mathbf{p}+\sigma }{\mathbf{q}+1-\mathbf{p}}.  \label{wwp}
\end{equation}%
In the plane $(s,z)$ we define the quadrants 
\begin{equation*}
Q_{1}=\left\{ (s,z)\in \mathbb{R}^{2}:s>0,z>0\right\} ,\quad Q_{2}=\left\{
(s,z)\in \mathbb{R}^{2}:s<0,z>0\right\} ,\quad Q_{3}=-Q_{1},\quad
Q_{4}==-Q_{2.}
\end{equation*}

\begin{remark}
\label{deri}Observe that $sz$ has the sign of $\varepsilon .\ $If $%
\varepsilon >0$ (equation with source), then $(s,z)\in Q_{1}\cup Q_{3},$ If $%
\varepsilon <0$ (equation with absorption) then $(s,z)\in Q_{2}\cup Q_{4}.$
Moreover, if $\mathbf{p}<\mathbf{N},$ and if $w$ is defined on an interval $%
(0,\rho ),$ then it is always decreasing, hence $(s,z)\in Q_{1}.$ Indeed $%
(r^{\mathbf{N}-1}\left\vert w^{\prime }\right\vert ^{\mathbf{p}-2}w^{\prime
})$ is decreasing. If there exists $r_{0}$ such that $w^{\prime }(r_{0})>0,$
then for $r<r_{0}$,$w^{\prime \mathbf{p}-1}\geq C_{0}r^{1-\mathbf{N}},$ thus 
$w$ is bounded, which by integration contradicts the assumption $\mathbf{p}<%
\mathbf{N}$. In case $\mathbf{p}>\mathbf{N},$ it can happen, as we see in
the sequel, that $w^{\prime }>0$ near $0.$
\end{remark}

In this paragraph we exclude the limit cases $\mathbf{p=N},\sigma =-\mathbf{%
p,}\sigma =-\mathbf{N,}$ which will be studied at paragraph \ref{limit}%
.\bigskip

We define two possible critical values 
\begin{equation}
\mathbf{q}_{c}=\frac{(\mathbf{N}+\sigma )(\mathbf{p}-1)}{\mathbf{N}-\mathbf{p%
}},\quad \quad \mathbf{q}_{S}=\frac{\mathbf{N}(\mathbf{p}-1)+\mathbf{p}+%
\mathbf{p}\sigma }{\mathbf{N}-\mathbf{p}},  \label{qcqs}
\end{equation}%
which are the well known Serrin's exponent and Sobolev exponent
respectively, in case $\mathbf{N}>\mathbf{p}>-\sigma $. Note that $\mathbf{q}%
_{c}<0$ when $(\mathbf{N}+\sigma )(\mathbf{N}-\mathbf{p)<0.}$

\medskip We first observe that the equation admits particular solutions 
\begin{equation}
w^{\ast }=a^{\ast }r^{-\gamma },\qquad \gamma =\frac{\mathbf{p}+\sigma }{%
\mathbf{q}+1-\mathbf{p}},\qquad a^{\ast \mathbf{q}-\mathbf{p}+1}=\varepsilon
\left\vert \gamma \right\vert ^{\mathbf{p}-2}\gamma (\mathbf{N}-\mathbf{p}-(%
\mathbf{p}-1)\gamma ),  \label{weto}
\end{equation}%
well defined for $\varepsilon \gamma (\mathbf{N}-\mathbf{p}-(\mathbf{p}%
-1)\gamma )>0$ that is $\mathbf{q}-\mathbf{q}_{c}$ has the sign of $%
\varepsilon (\mathbf{p}+\sigma )(\mathbf{N}-\mathbf{p})$. Then $w^{\ast }$
is $\infty $-singular at $0$ if $\mathbf{p}+\sigma >0,$ and $C^{0}$-regular
if $\mathbf{p}+\sigma <0$. $\bigskip $

\begin{definition}
In the following we say that a positive solution $w$ of equation (\ref{sca})
on an interval $(0,\rho )$ is $C^{0}$-\textbf{regular} if $w\in C^{2}(0,\rho
)\cap C^{0}\left[ 0,\rho \right) $, that means $\lim_{r\rightarrow 0}w=c\geq
0.$ If it exists, such a solution satisfies $(r^{\mathbf{N}-1}\left\vert
w^{\prime }\right\vert ^{\mathbf{p}-1})^{\prime }\sim _{r\longrightarrow
0}r^{\mathbf{N}-1+\sigma }c^{p},$ then $r^{\mathbf{N}-1}\left\vert w^{\prime
}\right\vert ^{\mathbf{p}-1}\sim _{r\longrightarrow 0}Cr^{\mathbf{N}+\sigma
}c^{p},$ $\left\vert w^{\prime }\right\vert \sim _{r\longrightarrow 0}Cr^{%
\frac{\sigma +1}{\mathbf{p}-1}},\left\vert w^{\prime \prime }\right\vert
\sim _{r\longrightarrow 0}Cr^{\frac{\sigma +2-\mathbf{p}}{\mathbf{p}-1}}$.
Then $w\in C^{2}\left[ 0,\rho \right) $ if $\sigma \geq \mathbf{p}-2,$ $w\in
C^{1}\left[ 0,\rho \right) $ if $\sigma \geq -1,$ and $w$ presents a cusp at
0 if $\sigma <-1.$ We say that $w$ is a $C^{0}$-\textbf{ground state} if
moreover $\rho =\infty .$
\end{definition}

Next we divide the study in 6 regions relative to $(\mathbf{N},\sigma )\in $ 
$\mathbb{R}^{2}$, and we show that we can reduce the study to only 3
regions, for which one of them is well known. This is due to a
transformation proved in \cite[Remark 2.4]{BiGi}:

\begin{lemma}
\label{reduc} For fixed $\mathbf{p}>0,$ consider the sets defined by 
\begin{eqnarray}
\mathbf{A} &=&\left\{ (\mathbf{N},\sigma )\in \mathbb{R}^{2}:\mathbf{N>p}%
>-\sigma \right\} ,\qquad \mathbf{B}=\left\{ (\mathbf{N},\sigma )\in \mathbb{%
R}^{2}:\mathbf{p}>\mathbf{N}>-\sigma )\right\} ,  \notag \\
\mathbf{C} &=&\left\{ (\mathbf{N},\sigma )\in \mathbb{R}^{2}:\mathbf{%
N>-\sigma >p}\right\} ,\qquad \mathbf{D}=\left\{ (\mathbf{N},\sigma )\in 
\mathbb{R}^{2}:-\sigma >\mathbf{N>p}\right\} ,  \label{regi} \\
\mathbf{E} &=&\left\{ (\mathbf{N},\sigma )\in \mathbb{R}^{2}:\mathbf{p>}%
-\sigma >\mathbf{N}\right\} ,\qquad \mathbf{F}=\left\{ (\mathbf{N},\sigma
)\in \mathbb{R}^{2}:-\sigma >\mathbf{p}>\mathbf{N}\right\} .  \notag
\end{eqnarray}%
Let $(\mathbf{N},\sigma )\in \mathbb{R}^{2}$ and $(\widehat{\mathbf{N}},%
\widehat{\sigma })\in \mathbb{R}^{2}$ such that $(\widehat{\mathbf{N}}-%
\mathbf{p})\mathbf{(p}+\sigma )=(\mathbf{N}-\mathbf{p})(\mathbf{p}+\widehat{%
\sigma }).$ Set 
\begin{equation}
\lambda =\frac{\widehat{\mathbf{N}}-\mathbf{p}}{\mathbf{N}-\mathbf{p}}=\frac{%
\mathbf{p}+\widehat{\sigma }}{\mathbf{p}+\sigma }=\frac{\widehat{\mathbf{N}}+%
\widehat{\sigma }}{\mathbf{N}+\sigma },  \label{lam}
\end{equation}%
and consider the change of unknown 
\begin{equation}
\widehat{w}(\widehat{r})=Cw(r),\qquad r=\widehat{r}^{\lambda },C=\lambda ^{%
\frac{p}{q+1-p}}.  \label{chu}
\end{equation}%
Then $w$ satisfies the equation (\ref{sca}) if and only if $\widehat{w}$
satisfies the analogous equation with $r,\mathbf{N},\sigma $ replaced by $%
\widehat{r},\widehat{\mathbf{N}},\widehat{\sigma },$ and 
\begin{equation}
\widehat{r}^{\widehat{\gamma }}\widehat{w}(\widehat{r})=Cr^{\gamma
}w(r),\qquad \widehat{r}^{\frac{\widehat{\mathbf{N}}-\mathbf{p}}{\mathbf{p}-1%
}}\widehat{w}(\widehat{r})=Cr^{\frac{\mathbf{N}-\mathbf{p}}{\mathbf{p}-1}%
}w(r),\qquad \widehat{r}^{-\frac{\widehat{\sigma }+1}{\mathbf{p}-1}}\widehat{%
w}^{\prime }(\widehat{r})=\lambda Cr^{-\frac{\sigma +1}{\mathbf{p}-1}%
}w^{\prime }(r).  \label{cal}
\end{equation}%
Moreover for any $(\mathbf{N},\sigma )\in \mathbf{F}$ (resp. $\mathbf{D}$,
resp. $\mathbf{E})$ there exists $(\widehat{\mathbf{N}},\widehat{\sigma }%
)\in $ $\mathbf{A}$ (resp. $\mathbf{B}$, resp. $\mathbf{C})$, such that (\ref%
{lam}) holds with $\lambda =-1.$ As a consequence all the results valid for
regions $\mathbf{A},\mathbf{B},\mathbf{C}$ apply respectively to $\mathbf{F},%
\mathbf{D},\mathbf{E}$ by changing $r$ into $\frac{1}{r}.$
\end{lemma}

\begin{proof}
Setting $t=\lambda \widehat{t}$ and $(\widehat{s},\widehat{z})=\lambda (s,z),
$ we obtain a system analogous to (\ref{pspz}) where $\mathbf{N},\sigma $
are replaced by $\widehat{\mathbf{N}},\widehat{\sigma }.$ This corresponds
to the change of unknown (\ref{chu}). And the regions $\mathbf{A},\mathbf{B},%
\mathbf{C}$ are respectively exchanged to $\mathbf{F},\mathbf{E},\mathbf{D},$
after choosing suitable reals $\lambda <0$ such that $\mathbf{N>0}$
correponds to $\widehat{\mathbf{N}}>0.$ The relations (\ref{cal}) are
straightforward, implying the last conclusions.
\end{proof}

As a consequence we prove that some upperestimates, which are classical when 
$\mathbf{N}>\mathbf{p}>-\sigma ,$ and still valid for other ranges of the
parameters:

\begin{lemma}
\label{osse}Let $\mathbf{q}>\mathbf{p}-1>0,$ and $\mathbf{N},\sigma \in 
\mathbb{R}.$ There is a constant $C_{\mathbf{N,p,q}}>0$ such that any
positive solution of (\ref{sca}) solution with in $(0,r_{0})$ (resp in $%
(r_{0},\infty )$ satisfies 
\begin{equation*}
w(r)\leq C_{\mathbf{N,p,q}}r^{-\gamma }\quad \text{in }(0,\frac{r_{0}}{2}%
)\quad \text{(resp. in }(2r_{0},\infty )\text{).}
\end{equation*}
\end{lemma}

\begin{proof}
In case $\varepsilon =-1,$ this comes the Osserman's property, which is
valid in the nonradial case when $\mathbf{N}$ is an integer, see a proof in 
\cite[Proposition 5.2]{Ve2}, and for any subsolution. In the case $%
\varepsilon =1,$ it also extends to integral estimates in the nonradial
case, see \cite[Theorem 3.1]{BiPo}, valid for any $\sigma \in \mathbb{R}.$
Both results suppose $\mathbf{N>1,}$ and are given and $\mathbf{N>p}$, even
if this condition does not seem necessary. Next suppose $\mathbf{N<p.}$ We
use the transformation (\ref{lam}) (\ref{chu}) with $\lambda =-1$ : it
defines $\widehat{\mathbf{N}}=2\mathbf{p-N}$ and $\widehat{\sigma }=-2%
\mathbf{p-}\sigma $, and $w(r)=\widehat{w}(\widehat{r}),$ $\widehat{r}=\frac{%
1}{r}.$ From (\ref{cal}), $\widehat{r}^{\widehat{\gamma }}\widehat{w}(%
\widehat{r})=Cr^{\gamma }w(r),$ and the behaviours near $0$ and $\infty $
are exchanged, hence the conclusion is still valid.
\end{proof}

From now on in this section, the proofs of the lemmas and theorems are given
in the Appendix. We first analyze the nature of the fixed points of the
system:

\begin{lemma}
\label{nat}The fixed points of system (\ref{pspz}) are 
\begin{equation}
\begin{array}{c}
M_{0}=(s_{0},z_{0})=(\gamma ,\mathbf{N}-\mathbf{p}-(\mathbf{p}-1)\gamma )=(%
\frac{\mathbf{p}+\sigma }{\mathbf{q}+1-\mathbf{p}},\frac{(\mathbf{N}-\mathbf{%
p})(\mathbf{q}-\mathbf{q}_{c})}{\mathbf{q}+1-\mathbf{p}}), \\ 
N_{0}=(0,\mathbf{N}+\sigma ),\qquad A_{0}=(\frac{\mathbf{N}-\mathbf{p}}{%
\mathbf{p}-1},0),\qquad \text{and \quad }(0,0).%
\end{array}
\label{fixe}
\end{equation}

(i) The point $M_{0}$ corresponds to the particular solutions defined at (%
\ref{weto}).The eigenvalues $\lambda _{1},\lambda _{2}$ of the linearized
system associated to system (\ref{pspz}) at $M_{0}$ are the roots of the
equation 
\begin{equation}
T(\lambda )=\lambda ^{2}-(s_{0}-z_{0})\lambda +\frac{\mathbf{q}-\mathbf{p}+1%
}{\mathbf{p}-1}s_{0}z_{0}=\lambda ^{2}-(\mathbf{p}\gamma +\mathbf{p}-\mathbf{%
N})\lambda +\frac{(\mathbf{p}+\sigma )(\mathbf{N}-\mathbf{p})(\mathbf{q}-%
\mathbf{q}_{c})}{(\mathbf{p}-1)(\mathbf{q}-\mathbf{p}+1)}=0.  \label{valp}
\end{equation}%
$M_{0}$ is a saddle point when $s_{0}z_{0}<0.$

(ii) The eigenvalues at $N_{0}$ are $l_{1}=\frac{\mathbf{p}+\sigma }{\mathbf{%
p}-1}$ and $l_{2}=-(\mathbf{N}+\sigma );$ and $\left\{ s=0\right\} $
contains nonadmissible trajectories linked to $l_{2}$. $N_{0}$ is a saddle
point when $(\mathbf{p}+\sigma )(\mathbf{N}+\sigma )>0.$

(iii) The eigenvalues at $A_{0}$ are $\mu _{1}=\frac{\mathbf{N}-\mathbf{p}}{%
\mathbf{p}-1}$ and $\mu _{2}=\frac{\mathbf{N}-\mathbf{p}}{\mathbf{p}-1}(%
\mathbf{q}_{c}-\mathbf{q})=\mathbf{N}+\sigma +\mathbf{q}\frac{\mathbf{p}-%
\mathbf{N}}{\mathbf{p}-1};$ and $\left\{ z=0\right\} $ contain nonadmissible
trajectories linked to $\mu _{1}.$ $A_{0}$ is a saddle point when $\mathbf{%
q>q}_{c}.$

(iv) The eigenvalues at $(0,0)$ are $\rho _{1}=\frac{\mathbf{p}-\mathbf{N}}{%
\mathbf{p}-1}$ and $\rho _{2}=\mathbf{N}+\sigma ,$ and $\left\{ s=0\right\} $
and $\left\{ z=0\right\} $ contain nonadmissible trajectories, linked to $%
\rho _{1}$ and $\rho _{2}.$ The point $(0,0)$ is a saddle point when $(%
\mathbf{N}-\mathbf{p})(\mathbf{N}+\sigma )>0.$
\end{lemma}

\begin{remark}
Recall that any trajectory in the phase plane distinct from the two axis
corresponds to an infinity with one parameter of solutions of equation (\ref%
{sca}). So any saddle-point leads at most to two families of such solutions,
and any source or sink leads to an infinity of solutions with two parameters.
\end{remark}

Next we make the link between the convergence to the fixed points and the
behaviour of the solutions $w$ near $0$ or $\infty .$

\begin{lemma}
\label{link}Let $w$ be any positive solution of (\ref{sca}) defined near $0$
(resp near $\infty $). Then as $r\longrightarrow 0,t\longrightarrow -\infty $
(resp $r\longrightarrow \infty ,t\longrightarrow \infty )$%
\begin{equation*}
\begin{array}{cccc}
\text{(i)} & \lim (s,z)=M_{0} & \Longrightarrow & w\sim w^{\ast },\qquad
w^{\prime }\sim w^{\prime \ast }, \\ 
\text{(ii)} & \lim (s,z)=N_{0} & \Longrightarrow & \lim w=c>0,\qquad \lim
r^{-\frac{\sigma +1}{\mathbf{p}-1}}w^{\prime }=\left\vert \frac{c}{\mathbf{N}%
+\sigma }\right\vert ^{\frac{q}{p-1}}\text{sign}(-\varepsilon (\mathbf{N}%
+\sigma )), \\ 
\text{(iii)} & \lim (s,z)=(0,0) & \Longrightarrow & \lim w=c>0,\qquad
\lim_{r\rightarrow 0}r^{\frac{\mathbf{N}-1}{\mathbf{p}-1}}w^{\prime }=d\neq
0, \\ 
\text{(iv)} & \lim (s,z)=A_{0} & \Longrightarrow & \lim r^{\frac{\mathbf{N}-%
\mathbf{p}}{\mathbf{p}-1}}w=k>0,\qquad \lim r^{\frac{\mathbf{N}-1}{\mathbf{p}%
-1}}w^{\prime }=k\frac{\mathbf{p}-\mathbf{N}}{\mathbf{p}-1}\text{sign}%
(-\varepsilon z).%
\end{array}%
\end{equation*}
\end{lemma}

\begin{lemma}
\label{glo} Let $w$ be any local solution in $(0,r_{0})$ (resp. in $%
(r_{0},\infty ).$ Then the associated trajectory in the phase plane is
bounded as $t\longrightarrow -\infty $ (resp. $t\longrightarrow \infty ).$
It converges to one of the fixed points of the system, or has a limit cycle
around $M_{0}$ (when $\mathbf{q}=\mathbf{q}_{S}$).
\end{lemma}

\subsection{Study of regions \textbf{A} and \textbf{F }}

We first consider region $\mathbf{A,}$ where the behaviour of the system (%
\ref{pspz}) has been described in \cite{BiGi}, and we deduce the following:

\begin{theorem}
\label{orange}Let $\mathbf{N}>\mathbf{p}>-\sigma $ (region $\mathbf{A}$).

There exist local ($C^{0}$-regular) solutions: for $\varepsilon =\pm 1$ and
any $w_{0}>0$ there exists a unique solution such that 
\begin{equation*}
\lim_{r\longrightarrow 0}w=w_{0}>0,\qquad \lim_{r\longrightarrow 0}r^{-\frac{%
\sigma +1}{\mathbf{p}-1}}w^{\prime }=-\varepsilon c(w_{0}),\quad c(w_{0})>0.
\end{equation*}

\noindent (1) Let $\mathbf{p}-1<\mathbf{q}<\mathbf{q}_{c}.$

$\bullet $ For $\varepsilon =\pm 1$ and any $k>0,$ there exists an infinity
of local solutions near 0 such that 
\begin{equation*}
\lim_{r\rightarrow 0}r^{^{\frac{\mathbf{N-p}}{\mathbf{p}-1}}}w=k>0.\qquad
\lim_{r\longrightarrow 0}r^{\frac{\mathbf{N}-1}{\mathbf{p}-1}}w^{\prime
}=-c(k)<0.
\end{equation*}

$\bullet $ For $\varepsilon =1,$ there is no positive solution in $%
(r_{0},\infty ),$ $r_{0}>0$.

$\bullet $ For $\varepsilon =-1,$ there exists a global particular solution $%
w^{\ast }=a^{\ast }r^{-\frac{\mathbf{p}+\sigma }{\mathbf{q}+1-\mathbf{p}}}.$
Moreover there exist solutions such that 
\begin{equation*}
\lim_{r\rightarrow 0}r^{^{\frac{\mathbf{N-p}}{\mathbf{p}-1}}}w=k>0,\qquad
\lim_{r\rightarrow \infty }r^{\gamma }w=a^{\ast }.
\end{equation*}

\noindent (2) Let $\mathbf{q}>\mathbf{q}_{c}.$

$\bullet $ For $\varepsilon =\pm 1$ and $k>0,$ there exists a local solution
near $\infty $, unique up to a scaling, such that 
\begin{equation*}
\lim_{r\rightarrow \infty }r^{^{\frac{\mathbf{N}-\mathbf{p}}{\mathbf{p}-1}%
}}w=k>0,\qquad \lim_{r\longrightarrow \infty }r^{\frac{\mathbf{N}-1}{\mathbf{%
p}-1}}w^{\prime }=-c(k)<0.
\end{equation*}

$\bullet $ For $\varepsilon =-1,$ all the local solutions near 0 are $C^{0}$%
-regular (the singularity is called removable), and not global.

$\bullet $ For $\varepsilon =1,$ there is a particular solution $w^{\ast
}=a^{\ast }r^{-\frac{\mathbf{p}+\sigma }{\mathbf{q}+1-\mathbf{p}}}.$ Moreover

(i) either $\mathbf{q}<\mathbf{q}_{S}$, there exists no $C^{0}$- ground
state and there exist solutions such that%
\begin{equation*}
\lim_{r\rightarrow 0}r^{\gamma }w=a^{\ast },\qquad \lim_{r\rightarrow \infty
}r^{^{\frac{\mathbf{p}-\mathbf{N}}{\mathbf{p}-1}}}w=k>0.
\end{equation*}

(ii) or $\mathbf{q}>\mathbf{q}_{S}$ and there exist $C^{0}$-ground states,
such that%
\begin{equation*}
\lim_{r\rightarrow 0}w=w_{0}>0,\qquad \lim_{r\rightarrow \infty }r^{\gamma
}w=a^{\ast }.
\end{equation*}

(iii) or $\mathbf{q}=\mathbf{q}_{S}$ and there is a family of explicit (well
known) $C^{0}$-ground states\label{gsw}: 
\begin{equation}
w=c(d+r^{\frac{\mathbf{p}+\sigma }{\mathbf{p}-1}})^{-\frac{\mathbf{N}-%
\mathbf{p}}{\mathbf{p}+\sigma }},\qquad d=c^{\mathbf{q}-\mathbf{p}+1}(%
\mathbf{N}+\sigma )^{-1}(\frac{\mathbf{N}-\mathbf{p}}{\mathbf{p}-1})^{1-%
\mathbf{p}}.  \label{ground}
\end{equation}
\end{theorem}

In the case $\mathbf{N}>\mathbf{p}>-\sigma ,$ $\mathbf{q}=\mathbf{q}_{c}$, a
precise local behaviour of logarithmic type was in great part studied for $%
\sigma =0,\varepsilon =1$ in \cite{Avi}, \cite[Theorem 4.1]{GuVe}, and for $%
\varepsilon =1$ in \cite{Ve1}; the existence of such solutions was not
clear. We give below a complete description:

\begin{theorem}
\label{qccrit} (1) Let $\mathbf{N}>\mathbf{p}>-\sigma $ and $\mathbf{q}=%
\mathbf{q}_{c}.$ Then

$\bullet $ For $\varepsilon =\pm 1$ and any $w_{0}>0$ there exist a unique
solution such that 
\begin{equation*}
\lim_{r\longrightarrow 0}w=w_{0}>0,\qquad \lim_{r\longrightarrow 0}r^{-\frac{%
\sigma +1}{\mathbf{p}-1}}w^{\prime }=-\varepsilon c(w_{0}),\quad c(w_{0})>0.
\end{equation*}

$\bullet $ For $\varepsilon =\pm 1$ there exist an infinity of local
solutions near 0 such that 
\begin{equation*}
\lim_{r\longrightarrow 0}(r\left\vert \ln r\right\vert ^{\frac{1}{\mathbf{p}%
+\sigma }})^{^{\frac{\mathbf{N-p}}{\mathbf{p}-1}}}w=c(\mathbf{N},\mathbf{p}%
)>0,\qquad \lim_{r\longrightarrow 0}r^{\frac{\mathbf{N}-1}{\mathbf{p}-1}%
}\left\vert \ln r\right\vert ^{\frac{\mathbf{N-p}}{(\mathbf{p}+\sigma )(%
\mathbf{p}-1)}}w^{\prime }=-\frac{\mathbf{N}-\mathbf{p}}{\mathbf{p}-1}c(%
\mathbf{N},\mathbf{p}),
\end{equation*}%
where $c(\mathbf{N},\mathbf{p})=((\frac{\mathbf{N}-\mathbf{p}}{\mathbf{p}-1}%
)^{\mathbf{p}}\frac{\mathbf{p}-1}{\mathbf{p+\sigma }})^{\frac{\mathbf{N}-%
\mathbf{p}}{(\mathbf{p}-1)(\mathbf{p+\sigma )}}}.$

$\bullet $ For $\varepsilon =1,$ there is no positive solution in $%
(r_{0},\infty ),$ $r_{0}>0$.

$\bullet $ For $\varepsilon =-1,$ there exist an infinity of solutions such
that 
\begin{equation*}
\lim_{r\longrightarrow \infty }(r\left\vert \ln r\right\vert ^{\frac{1}{%
\mathbf{p}+\sigma }})^{^{\frac{\mathbf{N-p}}{\mathbf{p}-1}}}w=c(\mathbf{N},%
\mathbf{p}),\qquad \lim_{r\longrightarrow \infty }r^{\frac{\mathbf{N}-1}{%
\mathbf{p}-1}}\left\vert \ln r\right\vert ^{\frac{\mathbf{N-p}}{(\mathbf{p}%
+\sigma )(\mathbf{p}-1)}}\left\vert w^{\prime }\right\vert =\frac{\mathbf{N}-%
\mathbf{p}}{\mathbf{p}-1}c(\mathbf{N},\mathbf{p}).
\end{equation*}
\end{theorem}

\begin{remark}
\textbf{\ }As mentioned in \cite[Remark 3.2]{BiGi}, some energy functions of
Pohozaev-type linked to the equation are well known. In fact they are valid
for any values of the parameters $\mathbf{p,q,N,}\sigma ,$ and for the two
equations $(\varepsilon =\pm 1)$:%
\begin{equation*}
F_{\theta }(r)=r^{\mathbf{N}}(\frac{(\mathbf{p}-1)\left\vert w^{\prime
}\right\vert ^{\mathbf{p}}}{\mathbf{p}}+\varepsilon r^{\sigma }\frac{w^{%
\mathbf{q}+1}}{\mathbf{q}+1}+\theta \frac{w\left\vert w^{\prime }\right\vert
^{\mathbf{p}-2}w^{\prime }}{r})=r^{\mathbf{N}-\mathbf{p}}w^{\mathbf{p}%
}\left\vert s\right\vert ^{\mathbf{p}-2}s(\frac{(\mathbf{p}-1)s}{\mathbf{p}}+%
\frac{z}{\mathbf{q}+1}-\theta ),
\end{equation*}%
either with $\theta =\frac{\mathbf{N}-\mathbf{p}}{\mathbf{p}}$ or $\theta =%
\frac{\mathbf{N}+\sigma }{\mathbf{q}+1},$ satisfying respectively 
\begin{equation*}
F_{\frac{\mathbf{N}-\mathbf{p}}{\mathbf{p}}}^{\prime }(r)=r^{\mathbf{N}%
-1+\sigma }(\frac{\mathbf{N}+\sigma }{\mathbf{q}+1}-\frac{\mathbf{N}-\mathbf{%
p}}{\mathbf{p}})w^{\mathbf{q}+1},\qquad F_{\frac{\mathbf{N}+\sigma }{\mathbf{%
q}+1}}^{\prime }(r)=r^{\mathbf{N}-1}(\frac{\mathbf{N}+\sigma }{\mathbf{q}+1}-%
\frac{\mathbf{N}-\mathbf{p}}{\mathbf{p}})\left\vert w^{\prime }\right\vert ^{%
\mathbf{p}},
\end{equation*}%
Other type of functions can be computed as the ones obtained for $\sigma =0$
in \cite[Proposition 2.2]{Bi}, coinciding with the functions above when $%
\mathbf{q}=\mathbf{q}_{S}.$ In particular $\mathbf{q}=\mathbf{q}_{S}$ is the
case of constant energy, leading to the existence of the ground states
mentioned above at (\ref{ground}) when $\varepsilon =1$, and to explicit
local solutions on $\left[ 0,r_{0}\right) $ and on $\left[ r_{0},\infty
\right) $ when $\varepsilon =-1.$
\end{remark}

From Lemma \ref{reduc}, we deduce the complete behaviour in region $\mathbf{F%
}$. It offers a new striking result in case $\mathbf{q}=\mathbf{q}_{S}$ of
existence of explicit $C^{0}$-ground states, increasing and bounded at $%
\infty :$

\begin{theorem}
\label{hypF} Let $-\sigma >\mathbf{p}>\mathbf{N}$ (region $\mathbf{F}$).
Then all the conclusions of Theorems \ref{orange} (for $\mathbf{q}\neq 
\mathbf{q}_{c}$) and \ref{qccrit} (for $\mathbf{q}=\mathbf{q}_{c}$) apply
after changing $r$ into $\frac{1}{r}.$ In particular for $\varepsilon =\pm 1$
and any $c>0$ there exist a unique local solution near $\infty $ such that 
\begin{equation*}
\lim_{r\longrightarrow \infty }w=C>0,\qquad \lim_{r\longrightarrow \infty
}r^{-\frac{\sigma +1}{\mathbf{p}-1}}w^{\prime }=\varepsilon d(C),\quad
d(C)>0.
\end{equation*}%
For $\varepsilon =1,\mathbf{q}=\mathbf{q}_{S}$ there are explicit solutions
given by 
\begin{equation*}
w=c(d+r^{\frac{\mathbf{p}+\sigma }{\mathbf{p}-1}})^{\frac{\mathbf{p-N}}{%
\mathbf{p}+\sigma }},\qquad d=c^{\mathbf{q}-\mathbf{p}+1}(\mathbf{N}+\sigma
)^{-1}(\frac{\mathbf{p-N}}{\mathbf{p}-1})^{1-\mathbf{p}}
\end{equation*}%
satisfying $\lim_{r\rightarrow 0}w=0,$ with $w\sim _{r\longrightarrow 0}cr^{%
\frac{\mathbf{p-N}}{\mathbf{p-}1}}$ and $\lim_{r\rightarrow \infty }w=cd^{%
\frac{\mathbf{p-N}}{\mathbf{p}+\sigma }}.$
\end{theorem}

\subsection{Study of Regions \textbf{B} and \textbf{D}}

The case of region $\mathbf{B}$ is particularly interesting; indeed we prove
the following:

\begin{theorem}
\label{rose}Let $\mathbf{p}>\mathbf{N}>-\sigma $ (region $\mathbf{B}$)

$\bullet $ For $\varepsilon =\pm 1$ there exists local $C^{0}$\textbf{%
-regular }solutions of\textbf{\ three }types:%
\begin{equation}
\lim_{r\rightarrow 0}w=w_{0}>0,\quad \lim_{r\longrightarrow 0}r^{-\frac{%
\sigma +1}{\mathbf{p}-1}}w^{\prime }=-\varepsilon c(w_{0}),\qquad c(w_{0})>0,
\label{no}
\end{equation}%
\noindent \noindent\ 
\begin{equation}
\lim_{r\rightarrow 0}w=0,\quad \lim_{r\rightarrow 0}r^{\frac{\mathbf{N}-%
\mathbf{p}}{\mathbf{p}-1}}w=k>0,\qquad \lim_{r\rightarrow 0}r^{\frac{\mathbf{%
N}-1}{\mathbf{p}-1}}w^{\prime }=c(k)>0  \label{ao}
\end{equation}%
\begin{equation}
\lim_{r\rightarrow 0}w=w_{0}>0,\qquad \lim_{r\rightarrow 0}r^{\frac{\mathbf{N%
}-1}{\mathbf{p}-1}}w^{\prime }=-\varepsilon c(w_{0}),\quad c(w_{0})>0.
\label{oo}
\end{equation}

For $\varepsilon =1$ they are not global, and there is no solution in $%
(r_{0},\infty )$. For $\varepsilon =-1$ there exists two types of global
solutions in $(0,\infty ):$

$\bullet $ $w^{\ast }(r)=a^{\ast }r^{-\frac{\mathbf{p}+\sigma }{\mathbf{q}+1-%
\mathbf{p}}},$

$\bullet $ solutions such that 
\begin{equation*}
\lim_{r\rightarrow 0}w=w_{0}>0,\;\lim_{r\rightarrow 0}r^{\frac{\mathbf{N}-1}{%
\mathbf{p}-1}}w^{\prime }=-c(w_{0})<0,\qquad w\sim _{r\longrightarrow \infty
}w^{\ast }.
\end{equation*}
\end{theorem}

Note that in the case $\mathbf{N=1,p}=2,\sigma =0,$ Theorem \ref{rose} can
be checked easily, since the equation $-w^{\prime \prime }=\varepsilon w^{q}$
admits a first integral: $w^{\prime 2}+\frac{2\varepsilon }{q+1}w^{q+1}=C.$

As a direct consequence we obtain the behaviour in region $\mathbf{D}.$ Here
also we have an very interesting behaviour: we find infinitely many bounded
solutions in an exterior domain which do not converge to $0$ at $\infty ,$
and global ones when $\varepsilon =-1:$

\begin{theorem}
\label{white}(region $\mathbf{D}$) Let $\mathbf{p}<\mathbf{N}<-\sigma .$ Then

$\bullet $ For $\varepsilon =\pm 1$ there exists local solutions near $%
\infty $\textbf{\ } of\textbf{\ three }types:%
\begin{equation*}
\lim_{r\rightarrow \infty }w=C>0,\quad \lim_{r\longrightarrow \infty }r^{-%
\frac{\sigma +1}{\mathbf{p}-1}}w^{\prime }=\varepsilon c(C),\qquad c(C)>0,
\end{equation*}%
\noindent \noindent\ 
\begin{equation*}
\lim_{r\rightarrow \infty }w=0,\quad \lim_{r\rightarrow \infty }r^{\frac{%
\mathbf{N-p}}{\mathbf{p}-1}}w=k>0,\qquad \lim_{r\rightarrow \infty }r^{\frac{%
\mathbf{N}-1}{\mathbf{p}-1}}w^{\prime }=-c(k)<0,
\end{equation*}%
\begin{equation*}
\lim_{r\rightarrow \infty }w=C>0,\qquad \lim_{r\rightarrow \infty }r^{\frac{%
\mathbf{N}-1}{\mathbf{p}-1}}w^{\prime }=D\neq 0.
\end{equation*}

For $\varepsilon =1$ they are not global, and there is no solution in $%
(0,r_{0})$. For $\varepsilon =-1$ there exists two types of global solutions
in $(0,\infty ):$

$\bullet $ $w^{\ast }(r)=a^{\ast }r^{-\frac{\mathbf{p}+\sigma }{\mathbf{q}+1-%
\mathbf{p}}},$ which is $C^{0}$-regular,

$\bullet $ solutions such that 
\begin{equation*}
w\sim _{r\longrightarrow 0}w^{\ast },\qquad \lim_{r\rightarrow \infty
}w=C>0,\;\lim_{r\rightarrow \infty }r^{\frac{\mathbf{N}-1}{\mathbf{p}-1}%
}w^{\prime }=c(C)>0.
\end{equation*}
\end{theorem}

\subsection{Study of regions \textbf{C} and \textbf{E}}

In regions $\mathbf{C}$ and then in region $\mathbf{E}$ from Lemma \ref%
{reduc}, we obtain the following results:

\begin{theorem}
\label{yellow}Let $\mathbf{p}<-\sigma <\mathbf{N}$ (region $\mathbf{C}$)
Then there exists no $C^{0}$-regular solution. For $\varepsilon =\pm 1$
there exist local (nonglobal) solutions near $\infty $ of two types: 
\begin{equation}
\lim_{r\rightarrow \infty }w=C>0,\qquad \lim_{r\rightarrow \infty }r^{-\frac{%
\sigma +1}{\mathbf{p}-1}}w^{\prime }=-\varepsilon D,\quad D>0,  \label{no2}
\end{equation}%
\begin{equation}
\lim_{r\rightarrow \infty }w=0,\quad \lim_{r\rightarrow \infty }r^{\frac{%
\mathbf{N}-\mathbf{p}}{\mathbf{p}-1}}w=k>0,\qquad \lim_{r\rightarrow \infty
}r^{\frac{\mathbf{N}-1}{\mathbf{p}-1}}w^{\prime }=-c(k)<0.  \label{ao2}
\end{equation}

For $\varepsilon =-1,$ there exist two types of global solutions in $%
(0,\infty ):$

$\bullet $ $w^{\ast }(r)=a^{\ast }r^{-\frac{\mathbf{p}+\sigma }{\mathbf{q}+1-%
\mathbf{p}}},$which is a cusp-solution,

$\bullet $ solutions such that 
\begin{equation*}
w\sim _{r\longrightarrow 0}w^{\ast },\qquad \lim_{r\rightarrow \infty
}w=C>0,\quad \lim_{r\rightarrow \infty }r^{-\frac{\sigma +1}{\mathbf{p}-1}%
}w^{\prime }=D>0,
\end{equation*}%
and there exist also solutions on $(r_{0},\infty )$ such that $u(r_{0})=0$
and others such that $\lim_{r\longrightarrow r_{0}}w=\infty ,$ and $w\sim
_{r\longrightarrow \infty }w^{\ast }$, and nonglobal solutions on $(0,r_{0})$
such that $w\sim _{r\longrightarrow 0}w^{\ast }$ and lim$_{r\longrightarrow
r_{0}}=\infty $\textbf{.}

For $\varepsilon =1$ there is no local solution in $(0,r_{0})$.
\end{theorem}

\bigskip

\begin{theorem}
\label{hypE}(region $\mathbf{E}$) Let $\mathbf{p}>-\sigma >\mathbf{N.}$ For $%
\varepsilon =\pm 1$ there exist local (nonglobal) solutions near $0$ of two
types: 
\begin{equation*}
\lim_{r\rightarrow 0}w=w_{0}>0,\qquad \lim_{r\rightarrow 0}r^{-\frac{\sigma
+1}{\mathbf{p}-1}}w^{\prime }=\varepsilon D,\quad D>0,
\end{equation*}%
\begin{equation*}
\lim_{r\rightarrow 0}w=0,\quad \lim_{r\rightarrow 0}r^{\frac{\mathbf{N}-%
\mathbf{p}}{\mathbf{p}-1}}w=k>0,\qquad \lim_{r\rightarrow 0}r^{\frac{\mathbf{%
N}-1}{\mathbf{p}-1}}w^{\prime }=d(k)>0.
\end{equation*}%
For $\varepsilon =-1,$ there exist two types of global solutions in $%
(0,\infty ):$

$\bullet $ $w^{\ast }(r)=a^{\ast }r^{-\frac{\mathbf{p}+\sigma }{\mathbf{q}+1-%
\mathbf{p}}},$

$\bullet $ solutions such that 
\begin{equation*}
\lim_{r\rightarrow 0}w=w_{0}>0,\;\lim_{r\rightarrow 0}r^{-\frac{\sigma +1}{%
\mathbf{p}-1}}w^{\prime }=d<0,\qquad w\sim _{r\longrightarrow \infty
}w^{\ast },
\end{equation*}%
and there exist also solutions on $(0,r_{0})$ such that $w(r_{0})=0$ , and
others such that $\lim_{r\longrightarrow r_{0}}w=\infty ,$ and $w\sim
_{r\longrightarrow 0}w^{\ast }$; and there exist nonglobal solutions on $%
(r_{0},\infty )$ such that $w\sim _{r\longrightarrow \infty }w^{\ast }$ and
lim$_{r\longrightarrow r_{0}}w=\infty $\textbf{. }

For $\varepsilon =1$ there is no local solution in $(r_{0},\infty )$\textbf{.%
}
\end{theorem}

\subsection{Other case of explicit solutions\label{inte}}

We have recalled at (\ref{ground}) the well-known explicit grounds states
obtained for $\mathbf{q}=\mathbf{q}_{S}$ and $\varepsilon =1.$ Here we give
another case where we find global explicit solutions in $\mathbb{R}%
^{N}\backslash \left\{ 0\right\} .$ We remark that system (\ref{one}) admits
the solutions $(u_{1},u_{2})=(u,u)$ when $p=q,$ where $u$ is a solution of
the scalar equation (\ref{eq}), given explicitely by (\ref{eur}) and (\ref%
{logi}); and the corresponding solutions of system (\ref{SZ}) satisfy the
relation $S\equiv -Z,$ and $\mathbf{p}=\mathbf{N}=-\sigma .$ This suggests
that system (\ref{pspz}) with general $\mathbf{p},\mathbf{N},\sigma $ may
admit particular explicit solutions for some values of the parameters. We
show below that it is true, and this result appears to be new:

\begin{theorem}
\label{expli}Let $\mathbf{q}>\mathbf{p}-1>0.$ When $\sigma =-\mathbf{p}\frac{%
\mathbf{N}-1}{\mathbf{p}-1}$ there exist explicit radial solutions $w$ of
the Hardy-H\'{e}non equation (\ref{sca}) with $\varepsilon =-1$, of the form 
\begin{eqnarray}
w &=&(c\pm d_{\mathbf{p},\mathbf{q},\mathbf{N}}r^{\frac{\mathbf{p}-\mathbf{N}%
}{\mathbf{p}-1}})^{-\frac{\mathbf{p}}{\mathbf{q}-\mathbf{p}+1}},\text{
\qquad if }\mathbf{p}\neq \mathbf{N,}  \label{flo} \\
w &=&(c\pm d_{\mathbf{N}}\ln r)^{-\frac{\mathbf{N}}{\mathbf{q}-\mathbf{N}+1}%
},\qquad \text{if }\mathbf{p}=\mathbf{N,}  \label{flai}
\end{eqnarray}%
where $c>0$ and $d_{\mathbf{p},\mathbf{q},\mathbf{N}}=\frac{(\mathbf{p}-1)(%
\mathbf{q}-\mathbf{p}+1)}{\mathbf{p}(\mathbf{p}-\mathbf{N)}}(\frac{\mathbf{p}%
}{(\mathbf{p}-1)(\mathbf{q}+1)})^{\frac{1}{\mathbf{p}}}$, $d_{\mathbf{N}}=%
\frac{\mathbf{q}-\mathbf{N}+1}{\mathbf{N}}(\frac{\mathbf{N}}{(\mathbf{N}-1)(%
\mathbf{q}+1)})^{\frac{1}{\mathbf{N}}}$.
\end{theorem}

\subsection{Limit cases\label{limit}}

Here we give a complete study of all the critical cases $\sigma =-\mathbf{p}%
\neq -\mathbf{N,}$ $\sigma =-\mathbf{N}\neq -\mathbf{p,p}=\mathbf{N}\neq
-\sigma ,\mathbf{p}=\mathbf{N}=-\sigma .$

\subsubsection{Case $\protect\sigma =-\mathbf{p}\neq -\mathbf{N}$}

\begin{theorem}
\label{sigmap} (1) Assume $\mathbf{N>p=-}\sigma $

$\bullet $ For $\varepsilon =\pm 1,$ there exist local solutions near $%
\infty $ such that 
\begin{equation*}
\lim_{r\rightarrow \infty }r^{\frac{\mathbf{N}-\mathbf{p}}{\mathbf{p}-1}%
}w=k>0,\qquad \lim_{r\rightarrow \infty }r^{\frac{\mathbf{N}-1}{\mathbf{p}-1}%
}w^{\prime }=-c(k)<0.
\end{equation*}

$\bullet $ For $\varepsilon =1,$ there exists an infinity with 2 parameters
of local solutions near $\infty $ such that 
\begin{equation}
\lim_{r\longrightarrow \infty }(\ln r)^{\frac{\mathbf{p}-1}{\mathbf{q}-%
\mathbf{p}+1}}w(r)=\left( \frac{\mathbf{p}-1}{\mathbf{q}}\right) ^{\frac{%
\mathbf{p}-1}{\mathbf{q}-\mathbf{p}+1}}.  \label{port}
\end{equation}

$\bullet $ For $\varepsilon =-1,$ there exists at least a local solution
near $0$ such that 
\begin{equation}
\lim_{r\longrightarrow 0}\left\vert \ln r\right\vert ^{\frac{\mathbf{p}-1}{%
\mathbf{q}-\mathbf{p}+1}}w(r)=\left( \frac{\mathbf{p}-1}{\mathbf{q}}\right)
^{\frac{\mathbf{p}-1}{\mathbf{q}-\mathbf{p}+1}}.  \label{part}
\end{equation}

(2) Assume $\mathbf{p}=-\sigma >\mathbf{N.}$ Then the behaviour is deduced
from (1) by changing $r$ into $\frac{1}{r}.$\bigskip
\end{theorem}

\subsubsection{Case $\protect\sigma =-\mathbf{N}\neq -\mathbf{p}$}

\begin{theorem}
\label{sigman} (1) Assume $-\sigma =\mathbf{N}>\mathbf{p.}$

$\bullet $ For $\varepsilon =\pm 1,$ there exist local solutions near $%
\infty $ such that 
\begin{equation}
\lim_{r\rightarrow \infty }r^{\frac{\mathbf{N}-\mathbf{p}}{\mathbf{p}-1}%
}w=k>0,\qquad \lim_{r\rightarrow \infty }r^{\frac{\mathbf{N}-1}{\mathbf{p}-1}%
}w^{\prime }=-c(k)<0.  \label{hac}
\end{equation}%
$\bullet $ For $\varepsilon =\pm 1,$ there exists an infinity with 2
parameters of local solutions near $\infty $ such that%
\begin{equation}
\lim_{r\longrightarrow \infty }w=C>0,\qquad \lim_{r\longrightarrow \infty
}r^{\frac{\mathbf{N}-1}{p-1}}(\ln r)^{-\frac{1}{\mathbf{p}-1}}w^{\prime
}=-\varepsilon c(C).  \label{(hec}
\end{equation}

$\bullet $ For $\varepsilon =-1,$ there exists $C^{0}$-regular solutions $%
w^{\ast }=a^{\ast }r^{\frac{\mathbf{N}-\mathbf{p}}{\mathbf{q}+1-\mathbf{p}}%
}. $ There exists local solutions near $0$ such that $w\sim
_{r\longrightarrow 0}w^{\ast },$ and local ones near $\infty $ such that $%
w\sim _{r\longrightarrow \infty }w^{\ast }.$ There exists an infinity of
solutions such that 
\begin{equation}
w\sim _{r\longrightarrow 0}w^{\ast },\qquad \lim_{r\longrightarrow \infty
}w=C>0,\quad \lim_{r\longrightarrow \infty }r^{\frac{\mathbf{N}-1}{p-1}}(\ln
r)^{-\frac{1}{\mathbf{p}-1}}w^{\prime }=(\frac{\mathbf{q}+1\mathbf{-p}}{%
\mathbf{N}-\mathbf{p}})^{\frac{\mathbf{q}}{\mathbf{q}+1\mathbf{-p}}}.
\label{hoc}
\end{equation}

(2) Assume $\mathbf{p}>\mathbf{N=}-\sigma .$ Then the behaviour is deduced
from (1) by changing $r$ into $\frac{1}{r}.$
\end{theorem}

\subsubsection{Case $\mathbf{p}=\mathbf{N}\neq -\protect\sigma $}

\begin{theorem}
\label{pegaln}(1) Assume $\mathbf{p}=\mathbf{N}>-\sigma $

$\bullet $ For $\varepsilon =\pm 1$ there exists local ($C^{0}$-regular)
solutions: for any $w_{0}>0$ there exists a unique solution such that 
\begin{equation*}
\lim_{r\longrightarrow 0}w=w_{0}>0,\qquad \lim_{r\longrightarrow 0}r^{-\frac{%
\sigma +1}{\mathbf{p}-1}}w^{\prime }=-\varepsilon c(w_{0}),\quad c(w_{0})>0.
\end{equation*}

$\bullet $ For $\varepsilon =1,$ there exist an infinity of local solutions
near $0$ such that 
\begin{equation*}
\lim_{r\longrightarrow 0}\left\vert \ln r\right\vert ^{-1}w=C>0,\qquad
\lim_{r\longrightarrow 0}rw^{\prime }=-C.
\end{equation*}

$\bullet $ For $\varepsilon =-1,$ there exists a particular solution $%
w^{\ast }=a^{\ast }r^{-\frac{\mathbf{p+\sigma }}{\mathbf{q}+1-\mathbf{p}}}.$
There exist local solutions near 0 such that $w\sim _{r\longrightarrow
0}w^{\ast },$ and local ones near $\infty $ such that $w\sim
_{r\longrightarrow \infty }w^{\ast }.$ There exists an infinity of solutions
such that 
\begin{equation*}
\lim_{r\longrightarrow 0}\left\vert \ln r\right\vert ^{-1}w=C,\qquad
\lim_{r\longrightarrow 0}rw^{\prime }=-C,\qquad w\sim _{r\longrightarrow
\infty }w^{\ast }.
\end{equation*}

(2) Assume $\mathbf{p}=\mathbf{N}<-\sigma .$ Then the behaviour is deduced
from (1) by changing $r$ into $\frac{1}{r}.$
\end{theorem}

\subsubsection{Case $\mathbf{p}=\mathbf{N}=-\protect\sigma $}

\begin{theorem}
\label{pnsigma}Assume $\mathbf{p}=\mathbf{N}=-\sigma $. For $\varepsilon
=-1, $ there exist local explicit solutions near $0$ or near $\infty $ of
the form (\ref{flai}): 
\begin{equation}
w=(C\pm d_{\mathbf{N}}\ln r)^{-\frac{\mathbf{N}}{\mathbf{q}-\mathbf{N}+1}%
},\quad C\in \mathbb{R}.  \label{frio}
\end{equation}%
For $\varepsilon =1$ there is no local solution near $0$ nor $\infty .$
\end{theorem}

\section{Description of the radial solutions of system (\protect\ref{SA}) 
\label{exi}}

\subsection{The case $p=q$\label{pandq}}

In the case $p=q,$ the system (\ref{one}) admits solutions of the form $%
u_{1}\equiv u_{2}\equiv u,$ where $u$ is any solution of the scalar equation
(\ref{eq}). However at Proposition \ref{locun} we have constructed local
solutions such that $u_{1}\neq u_{2}$, for example such that $u_{1}^{\prime
}(r_{0})=0,$ $u_{2}^{\prime }(r_{0})\neq 0$ at some point $r_{0}>0.$ A
natural question is the existence global solutions in $\mathbb{R}%
^{N}\backslash \left\{ 0\right\} .$ Here we answer this question:

\begin{proposition}
\label{caspq}Assume that $p=q>1.$ Then all the radial solutions of system (%
\ref{one}) satisfy the first integral in any interval of definition 
\begin{equation}
\left\vert u_{1}^{\prime }\right\vert ^{q}u_{1}^{\prime }-\left\vert
u_{2}^{\prime }\right\vert ^{q}u_{2}^{\prime }\equiv Cr^{(1-N)(q+1)},\qquad
C\in \mathbb{R},  \label{first}
\end{equation}%
and can be computed by quadratures. All the radial solutions $(u_{1},u_{2})$
in $\mathbb{R}^{N}\backslash \left\{ 0\right\} $ satisfy $u_{1}\equiv
u_{2}\equiv u,$ where $u$ is any solution of the scalar Hamilton-Jacobi
equation.
\end{proposition}

\begin{proof}
Here system (\ref{Sw}) reduces to 
\begin{equation*}
\left\{ 
\begin{array}{c}
w_{1}^{\prime }=r^{(N-1)(1-q)}\left\vert w_{2}\right\vert ^{q}, \\ 
w_{2}^{\prime }=r^{(N-1)(1-q)}\left\vert w_{1}\right\vert ^{q}.%
\end{array}%
\right.
\end{equation*}%
As a consequence, 
\begin{equation*}
\left\vert w_{1}\right\vert ^{q}w_{1}^{\prime }-\left\vert w_{2}\right\vert
^{q}w_{2}^{\prime }=r^{(N-1)(1-q)}\left\vert w_{1}\right\vert ^{q}\left\vert
w_{2}\right\vert ^{q}-r^{(N-1)(1-q)}\left\vert w_{1}\right\vert
^{q}\left\vert w_{2}\right\vert ^{q}=0.
\end{equation*}%
So we get the relation 
\begin{equation*}
\left\vert w_{1}\right\vert ^{q}w_{1}-\left\vert w_{2}\right\vert
^{q}w_{2}\equiv C,
\end{equation*}%
equivalent to (\ref{first}). Suppose that $C\neq 0.$ By symmetry we can
suppose that $C=c^{q+1}>0.$ Then we obtain 
\begin{equation*}
\frac{w_{2}^{\prime }}{\left\vert C+\left\vert w_{2}\right\vert
^{q}w_{2}\right\vert ^{\frac{q}{q+1}}}=r^{(N-1)(1-q)}.
\end{equation*}%
We claim that the solution cannot be defined on\textbf{\ }$\mathbb{R}%
^{N}\backslash \left\{ 0\right\} $\textbf{. }Indeed let 
\begin{equation*}
F(\theta )=\dint_{0}^{\theta }\frac{d\theta }{\left\vert c^{q+1}+\left\vert
\theta \right\vert ^{q}\theta \right\vert ^{\frac{q}{q+1}}}=\dint_{0}^{-c}%
\frac{d\theta }{\left\vert c^{q+1}+\left\vert \theta \right\vert ^{q}\theta
\right\vert ^{\frac{q}{q+1}}}+\dint_{-c}^{\theta }\frac{d\theta }{\left\vert
c^{q+1}+\left\vert \theta \right\vert ^{q}\theta \right\vert ^{\frac{q}{q+1}}%
}.
\end{equation*}%
This function is well defined, since the integrals are convergent at the
bound $-c$ since $\frac{q}{q+1}<1.$ And the integrals are convergent at the
bounds $\pm \infty ,$ since $q>1,$ thus $F$ is bounded. If the solution is
global, that means $r$ describes $(0,\infty ),$ then $F(w_{2})=\frac{%
r^{N-(N-1)q}}{N-(N-1)q}+D,$ $D\in \mathbb{R},$ for $q\neq \frac{N}{N-1},$ $%
F(w_{2})=\ln r+D$ if $q=\frac{N}{N-1},$ which is impossible in any case. So
the solutions are not global. Hence all the global solutions on $(0,\infty )$
satisfy $w_{1}\equiv w_{2},$ $u_{1}=u_{2}+c,$ $c\in \mathbb{R},$ where $%
u_{2} $ is solution of the scalar equation. The nonglobal solutions can be
computed on any interval where $w_{1},$ \textbf{\ }$w_{2}$ have a constant
sign, by the formulas $F(w_{2})=\frac{r^{N-(N-1)q}}{N-(N-1)q}+D,$ for $q\neq 
\frac{N}{N-1},$ $F(w_{2})=\ln r+D$ if $q=\frac{N}{N-1}.$
\end{proof}

\begin{remark}
\label{tric}When $p=q\neq \frac{N}{N-1},$ the existence of local solutions
near $0$ or $\infty $ will be described as a particular case of Theorem \ref%
{locsol}. When $q=\frac{N}{N-1},$ such solutions do not exist, because $F$
is bounded.
\end{remark}

\subsection{Constant sign solutions $(u_{1}^{\prime },u_{2}^{\prime })$ of
system (\protect\ref{SA})}

In this paragraph, we study the existence of radial solutions $(u_{1},u_{2})$
of system (\ref{one}) in terms of the derivatives $u_{1}^{\prime
},u_{2}^{\prime },$ by applying all the results of Section \ref{app} to
system (\ref{SA}). It appears that the situation is extremely rich when $%
p\neq q.$ Here we focus our study on solutions defined in $%
B_{r_{0}}\backslash \left\{ 0\right\} ,$ or $\mathbb{R}^{N}\backslash 
\overline{B_{r_{0}}},$ and above all on global solutions in $\mathbb{R}%
^{N}\backslash \left\{ 0\right\} .\bigskip $ $.$

We distinguish four regions of study, corresponding respectively to the
regions $\mathbf{A},\mathbf{B},\mathbf{C},\mathbf{D}$ defined at (\ref{regi}%
) for system (\ref{pspz}), where $\mathbf{p},\mathbf{N},\sigma $ are defined
at (\ref{mns}), (\ref{red}).

\FRAME{ftbpF}{4.8372in}{4.0201in}{0in}{}{}{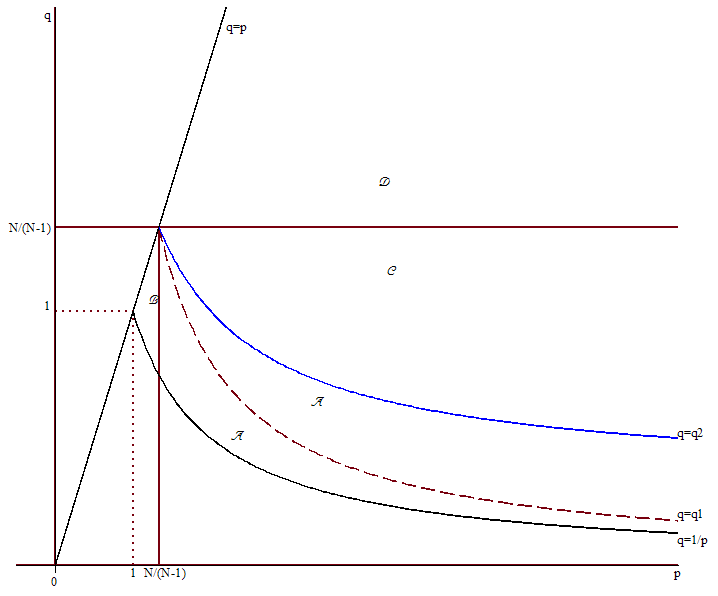}{\special%
{language "Scientific Word";type "GRAPHIC";maintain-aspect-ratio
TRUE;display "USEDEF";valid_file "F";width 4.8372in;height 4.0201in;depth
0in;original-width 4.5916in;original-height 3.811in;cropleft "0";croptop
"1";cropright "1";cropbottom "0";filename
'abcdlastfinal.png';file-properties "XNPEU";}}

\begin{definition}
\label{fir}Let $\mathcal{S=}$ $\left\{ (p,q)\in \mathbb{R}^{2}:pq>1,q\leq
p,q\neq q_{1},q_{2},p,q\neq \frac{N}{N-1}\right\} $ we consider the subsets
of $\mathcal{S}$ defined by 
\begin{eqnarray*}
\mathcal{A} &=&\left\{ p>\frac{N}{N-1},\;q<q_{2}\right\} ,\qquad \mathcal{B}%
=\left\{ p<\frac{N}{N-1}\right\} , \\
\mathcal{C} &=&\left\{ q_{2}<q<\frac{N}{N-1}\right\} ,\qquad \mathcal{D}%
=\left\{ q>\frac{N}{N-1}\right\} .
\end{eqnarray*}
\end{definition}

\begin{lemma}
\label{return}Let $(u_{1},u_{2})$ be any radial solution of system (\ref{one}%
) defined near $0$ (resp near $\infty $); let $w_{1},w_{2}$ be associated by
(\ref{ww}) and $w=\left\vert w_{1}\right\vert $ , solution of (\ref{eqm})
with (\ref{mns}). Then 
\begin{equation}
u_{1}^{\prime }\quad \text{has the sign of}-w^{\prime },\quad \text{and}%
\quad u_{2}^{\prime }\quad \text{has the sign of}\quad \varepsilon w^{\prime
}.  \label{sigi}
\end{equation}%
And as $r\longrightarrow 0,t\longrightarrow -\infty $ (resp $%
r\longrightarrow \infty ,t\longrightarrow \infty ),$%
\begin{equation*}
\begin{array}{cccc}
\text{(i)} & w\sim w^{\ast } & \Rightarrow & (u_{1}^{\prime },u_{2}^{\prime
})\sim (u_{1}^{\prime \ast },u_{2}^{\prime \ast }), \\ 
\text{(ii)} & \left\{ 
\begin{array}{c}
\lim w=c>0, \\ 
\lim r^{-\frac{\sigma +1}{\mathbf{p}-1}}w^{\prime }=d\neq 0,%
\end{array}%
\right. & \Rightarrow & \left\{ 
\begin{array}{c}
\lim r^{N-1}\left\vert u_{1}^{\prime }\right\vert =c_{1}>0, \\ 
\lim r^{(N-1)q-1}\left\vert u_{2}^{\prime }\right\vert =c_{2}>0,%
\end{array}%
\right. \\ 
\text{(iii)} & \left\{ 
\begin{array}{c}
\lim w=c>0, \\ 
\lim r^{\frac{\mathbf{N}-1}{\mathbf{p}-1}}w^{\prime }=d\neq 0,%
\end{array}%
\right. & \Rightarrow & \left\{ 
\begin{array}{c}
\lim r^{N-1}\left\vert u_{1}^{\prime }\right\vert =c_{1}>0, \\ 
\lim r^{N-1}\left\vert u_{2}^{\prime }\right\vert =c_{2}>0,%
\end{array}%
\right. \\ 
\text{(iv)} & \left\{ 
\begin{array}{c}
\lim r^{\frac{\mathbf{N}-\mathbf{p}}{\mathbf{p}-1}}w=k>0, \\ 
\lim r^{\frac{\mathbf{N}-1}{\mathbf{p}-1}}w^{\prime }=d(k)\neq 0,%
\end{array}%
\right. & \Rightarrow & \left\{ 
\begin{array}{c}
\lim r^{(N-1)p-1}\left\vert u_{1}^{\prime }\right\vert =k>0, \\ 
\lim r^{N-1}\left\vert u_{2}^{\prime }\right\vert =\left\vert
d(k)\right\vert >0.%
\end{array}%
\right.%
\end{array}%
\end{equation*}
\end{lemma}

\begin{proof}
From Proposition \ref{nom}, the function $w=\left\vert w_{1}\right\vert
=r^{N-1}\left\vert u_{1}^{\prime }\right\vert $ satisfies the Hardy-H\'{e}%
non equation with $\mathbf{p},\mathbf{q},\mathbf{N},\sigma $ given by (\ref%
{mns}),(\ref{red}), and we deduce $\left\vert w_{2}\right\vert $ from any of
the formulas 
\begin{equation}
\left\vert w_{2}\right\vert =(r^{(N-1)(p-1)}w_{1}^{\prime })^{\frac{1}{p}},%
\text{ \ or\ }w_{2}^{\prime }=r^{(N-1)(1-q)}w^{q},  \label{mac}
\end{equation}%
and then we obtain the conclusions by computation, from the relations $%
u_{i}^{\prime }=-r^{1-N}w_{i},$ $i=1,2.$ The signs of the derivatives are
obtained from (\ref{reco}) and (\ref{siga}), since $s\equiv S$.
\end{proof}

\begin{remark}
\label{prat}(i) From the phase plane analysis in Section \ref{app}, all the
trajectories relative to singular \textbf{global} solutions of the Hardy-H%
\'{e}non equation (\ref{eqm}) are located in the quadrant containing the
point $M_{0}=(S_{0},Z_{0})$ corresponding to the particular solutions.
Morover from (\ref{dif}), $u_{1}^{\prime }$ has the sign of $S$ and $%
u_{2}^{\prime }$ has the sign of $-Z.$ So $u_{1}^{\prime }$ has the sign of $%
S_{0}$ and $u_{2}^{\prime }$ has the sign of $-Z_{0}.$ From (\ref{fixs}),
for $q\neq q_{1},q_{2},$ 
\begin{equation}
u_{1}^{\prime }<0\Longleftrightarrow q>q_{2},\qquad u_{2}^{\prime
}<0\Longleftrightarrow q>q_{1}.  \label{sig}
\end{equation}%
Then, with (\ref{sigi}) and (\ref{sig}) we have a complete knowledge of
system (\ref{SA}) as soon as we know the absolute value of the
derivatives.\medskip

(ii) Some results of Section \ref{app} involve the Serrin exponent $\mathbf{q%
}_{c}$ and the Sobolev exponent $\mathbf{q}_{S},$ defined at (\ref{qcqs}).
We easily check that for our system, 
\begin{equation*}
\mathbf{q}\leq \mathbf{q}_{c}\Longleftrightarrow q\leq q_{1}=q_{1}(p),\qquad 
\mathbf{q}\leq \mathbf{q}_{S}\Longleftrightarrow q\leq q^{\ast }=q^{\ast
}(p).
\end{equation*}
\end{remark}

As a consequence of Theorems \ref{orange}, \ref{rose}, \ref{white}, \ref%
{yellow} , we get the following result on\textit{\ }\textbf{global}
solutions of system (\ref{one}), where we use Remark \ref{prat}:

\begin{theorem}
\label{regA}(1) Let $(p,q)\in \mathcal{A},$ with $q<q_{1}$. Then

$\bullet $ there exists a particular solution of system (\ref{SA}) such that 
$u_{1}^{\ast \prime }=a_{1}r^{-\frac{p+1}{pq-1}}>0$ and $u_{2}^{\ast \prime
}=a_{2}r^{-\frac{q+1}{pq-1}}>0,$

$\bullet $ there exist solutions on $(0,\infty )$ such that $u_{1}^{\prime
},u_{2}^{\prime }>0$ and 
\begin{equation*}
\lim_{r\rightarrow 0}r^{(N-1)p-1}u_{1}^{\prime }=c_{1}>0,\qquad
\lim_{r\rightarrow 0}r^{N-1}u_{2}^{\prime }=c(c_{1})>0,\qquad
(u_{1},u_{2})\sim _{r\rightarrow \infty }(u_{1}^{\ast },u_{2}^{\ast }).
\end{equation*}

\noindent (2) Let $(p,q)\in \mathcal{A},$ with $q>q_{1}.$Then

$\bullet $ there exists a particular solution such that $u_{1}^{\ast \prime
}=a_{1}r^{-\frac{p+1}{pq-1}}>0$ and $u_{2}^{\ast \prime }=-a_{2}r^{-\frac{q+1%
}{pq-1}}<0,$

$\bullet $ if $q<q^{\ast },$ there exist solutions such that $u_{1}^{\prime
}>0>u_{2}^{\prime },$ and 
\begin{equation*}
(u_{1}^{\prime },u_{2}^{\prime })\sim _{r\rightarrow 0}(u_{1}^{\ast \prime
},u_{2}^{\ast \prime }),\qquad \lim_{r\rightarrow \infty
}r^{(N-1)p-1}u_{1}^{\prime }=c_{1}>0,\qquad \lim_{r\rightarrow \infty
}r^{N-1}u_{2}^{\prime }=-c(c_{1})<0,
\end{equation*}

$\bullet $ if $q>q^{\ast },$ there exist solutions such that 
\begin{equation*}
\lim_{r\rightarrow 0}r^{N-1}u_{1}^{\prime }=c_{1}>0,\qquad
\lim_{r\rightarrow 0}r^{(N-1)q-1}u_{2}^{\prime }=-c_{3}<0,\qquad
(u_{1}^{\prime },u_{2}^{\prime })\sim _{r\rightarrow \infty }(u_{1}^{\ast
\prime },u_{2}^{\ast \prime }),
\end{equation*}

$\bullet $ if $q=q^{\ast },$ there exist \textbf{explicit} solutions: 
\begin{equation}
u_{1}^{\prime }=cr^{1-N}(d+r^{\frac{p-q^{\ast }}{2}})^{\frac{2(N-(N-1)p)}{%
p-q^{\ast }}},\qquad u_{2}^{\prime }=-br^{(1-(N-1)q^{\ast })}(d+r^{\frac{%
p-q^{\ast }}{2}})^{^{\frac{2((N-1)q^{\ast }-N)}{p-q^{\ast }}}}  \label{exact}
\end{equation}%
for any $c>0$, with $d=d(c)=(\frac{c^{pq^{\ast }-1}}{(N-1)p-N})^{\frac{1}{p}}%
\frac{1}{N-(N-1)q^{\ast }}.$ $b=b(c)=(c((N-1)p-N))^{\frac{1}{p}}$.\medskip

\noindent (3) Let $(p,q)\in \mathcal{B}.$ Then

$\bullet $ there exists a particular solution such that $u_{1}^{\ast \prime
}=a_{1}r^{-\frac{p+1}{pq-1}}>0$ and $u_{2}^{\ast \prime }=a_{2}r^{-\frac{q+1%
}{pq-1}}>0.$

$\bullet $ there exists solutions such that 
\begin{equation*}
\lim_{r\rightarrow 0}r^{N-1}u_{1}^{\prime }=c_{1},\qquad \lim_{r\rightarrow
0}r^{N-1}u_{2}^{\prime }=c_{2},\qquad (u_{1}^{\prime },u_{2}^{\prime })\sim
_{r\rightarrow \infty }(u_{1}^{\ast \prime },u_{2}^{\ast \prime }).
\end{equation*}%
\noindent (4) Let $(p,q)\in \mathcal{C\cup D}.$ Then there exists a
particular solution such that $u_{1}^{\ast \prime }=-a_{1}r^{-\frac{p+1}{pq-1%
}}<0$ and $u_{2}^{\ast \prime }=-a_{2}r^{-\frac{q+1}{pq-1}}<0.$

$\bullet $ If $(p,q)\in \mathcal{C},$ there also exist solutions such that 
\begin{equation*}
(u_{1}^{\prime },u_{2}^{\prime })\sim _{r\rightarrow 0}(u_{1}^{\ast \prime
},u_{2}^{\ast \prime }),\qquad \lim_{r\rightarrow \infty
}r^{N-1}u_{1}^{\prime }=-c_{1}<0,\qquad \lim_{r\rightarrow \infty
}r^{(N-1)q-1}u_{2}^{\prime }=-c(c_{1})<0.
\end{equation*}

$\bullet $ If $(p,q)\in \mathcal{D},$ there also exist solutions such that 
\begin{equation*}
(u_{1}^{\prime },u_{2}^{\prime })\sim _{r\rightarrow 0}(u_{1}^{\ast \prime
},u_{2}^{\ast \prime }),\qquad \lim_{r\rightarrow \infty
}r^{N-1}u_{1}^{\prime }=-c_{1}<0,\qquad \lim_{r\rightarrow \infty
}r^{N-1}u_{2}^{\prime }=-c_{2}<0.
\end{equation*}

\noindent And all the global solutions are described.
\end{theorem}

And we also get the existence of \textbf{local} but not global solutions
near $0$ or $\infty $ with a behaviour of linear type, by using Lemma \ref%
{return}:

\begin{theorem}
\label{locsol}

(1) Local existence of solutions of system (\ref{SA}) near $0$:

$\bullet $ in regions $\mathcal{A}$ or $\mathcal{B},$ there exist solutions
such that 
\begin{equation*}
\lim_{r\longrightarrow 0}r^{N-1}u_{1}^{\prime }=\pm c(c_{2}),\qquad
\lim_{r\longrightarrow 0}r^{(N-1)q-1}u_{2}^{\prime }=-c_{2}<0,
\end{equation*}

$\bullet $ in region $\mathcal{A}$ with $q<q_{1},$ there exist solutions
such that 
\begin{equation*}
\lim_{r\rightarrow 0}r^{(N-1)p-1}u_{1}^{\prime }=c(c_{2})>0,\qquad
\lim_{r\rightarrow 0}r^{N-1}u_{2}^{\prime }=\pm c_{2},
\end{equation*}

$\bullet $ in region $\mathcal{B},$ there exist solutions such that 
\begin{equation*}
\lim_{r\rightarrow 0}r^{(N-1)p-1}u_{1}^{\prime }=-c(c_{2})<0,\qquad
\lim_{r\rightarrow 0}r^{N-1}u_{2}^{\prime }=\pm c_{2},
\end{equation*}%
and solutions such that for any $c_{1},c_{2}\in \mathbb{R}\backslash \left\{
0\right\} ,$ 
\begin{equation*}
\lim_{r\longrightarrow 0}r^{N-1}u_{1}^{\prime }=c_{1},\qquad
\lim_{r\rightarrow 0}r^{N-1}u_{2}^{\prime }=c_{2},
\end{equation*}

\noindent and no such solution near $0$ in regions $\mathcal{C},\mathcal{D}$%
.\medskip

\noindent (2) Local existence of solutions near $\infty $:

$\bullet $ in regions $\mathcal{C}$ and $\mathcal{D},$ for any $%
c_{1},c_{2}\in \mathbb{R}\backslash \left\{ 0\right\} ,$ there exist
solutions such that 
\begin{equation*}
\lim_{r\longrightarrow \infty }r^{N-1}u_{1}^{\prime }=c_{1},\qquad
\lim_{r\rightarrow \infty }r^{N-1}u_{2}^{\prime }=c_{2},
\end{equation*}

$\bullet $ in region $\mathcal{A}$ with $q>q_{1},$ and in regions $\mathcal{C%
}$ and $\mathcal{D},$ there exist solutions such that%
\begin{equation*}
\lim_{r\rightarrow \infty }r^{(N-1)p-1}u_{1}^{\prime }=c_{1}>0,\qquad
\lim_{r\rightarrow \infty }r^{N-1}u_{2}^{\prime }=\pm c(c_{1}),
\end{equation*}

$\bullet $ in region $\mathcal{C},$ there exist solutions such that%
\begin{equation*}
\lim_{r\longrightarrow \infty }r^{N-1}u_{1}^{\prime }=\pm c(c_{2}),\qquad
\lim_{r\longrightarrow \infty }r^{(N-1)q-1}u_{2}^{\prime }=-c_{2}<0,
\end{equation*}

$\bullet $ in region $\mathcal{D},$ there exist solutions such that%
\begin{equation*}
\lim_{r\longrightarrow \infty }r^{N-1}u_{1}^{\prime }=\pm c(c_{2}),\qquad
\lim_{r\longrightarrow \infty }r^{(N-1)q-1}u_{2}^{\prime }=c_{2}>0.
\end{equation*}
\end{theorem}

Finally we consider the limit cases, where for simplicity we only mention
the solutions presenting a logarithmic behaviour:

\begin{theorem}
\label{limcas} (1) For $q=q_{1}<p$ , there exist local solutions of system (%
\ref{one}) near $0$ or $\infty $, with respectively 
\begin{equation*}
\lim_{r\longrightarrow 0}r^{(N-1)p-1}\left\vert \ln r\right\vert ^{\frac{p}{%
pq-1}}u_{1}^{\prime }=C_{1}>0,\qquad \lim_{r\longrightarrow
0}r^{N-1}\left\vert \ln r\right\vert ^{\frac{1}{pq-1}}u_{2}^{\prime }=\pm
C_{2},
\end{equation*}%
\begin{equation*}
\lim_{r\longrightarrow \infty }r^{(N-1)p-1}\left\vert \ln r\right\vert ^{%
\frac{p}{pq-1}}u_{1}^{\prime }=C_{1}>0,\qquad \lim_{r\longrightarrow \infty
}r^{N-1}\left\vert \ln r\right\vert ^{\frac{1}{pq-1}}u_{2}^{\prime }=C_{2}>0,
\end{equation*}

(2) For $q=q_{2}<p$ $\mathbf{,}$ there exist local solutions near $0$ or $%
\infty $, with respectively%
\begin{equation}
\lim_{r\longrightarrow 0}r^{N-1}\left\vert \ln r\right\vert ^{\frac{1}{pq-1}%
}u_{1}^{\prime }(r)=C_{1}>0,\quad \lim_{r\longrightarrow
0}r^{(N-1)q-1}\left\vert \ln r\right\vert ^{\frac{q}{pq-1}}u_{2}^{\prime
}=C_{2}>0,  \label{kil}
\end{equation}%
\begin{equation}
\lim_{r\longrightarrow \infty }r^{N-1}(\ln r)^{\frac{1}{pq-1}}u_{1}^{\prime
}(r)=C_{1}>0,\quad \lim_{r\longrightarrow \infty }r^{(N-1)q-1}(\ln r)^{\frac{%
q}{pq-1}}u_{2}^{\prime }=-C_{2}<0,  \label{kol}
\end{equation}

(3) For $p=\frac{N}{N-1}>q$ , there exist local solutions near $0,$ such
that 
\begin{equation*}
\lim_{r\longrightarrow 0}r^{N-1}\left\vert \ln r\right\vert
^{-1}u_{1}^{\prime }=C_{1}>0,\qquad \lim_{r\longrightarrow
0}r^{N-1}u_{2}^{\prime }=\pm C_{2},
\end{equation*}%
and global solutions such that 
\begin{equation*}
\lim_{r\longrightarrow 0}r^{N-1}\left\vert \ln r\right\vert
^{-1}u_{1}^{\prime }=C_{1}>0,\qquad \lim_{r\longrightarrow
0}r^{N-1}u_{2}^{\prime }=C_{2}>0,\qquad (u_{1}^{\prime },u_{2}^{\prime
})\sim _{r\longrightarrow \infty }(a_{1}r^{-\frac{p+1}{pq-1}},a_{2}r^{-\frac{%
q+1}{pq-1}}),
\end{equation*}

(4) For $q=\frac{N}{N-1}<p$ , there exist local solutions near $\infty $
such that 
\begin{equation*}
\lim_{r\longrightarrow \infty }r^{N-1}u_{1}^{\prime }=\pm C_{1},\qquad
\lim_{r\longrightarrow \infty }r^{N-1}(\ln r)^{-p}u_{2}^{\prime }=-C_{2}<0,
\end{equation*}%
and global ones such that, 
\begin{equation*}
(u_{1}^{\prime },u_{2}^{\prime })\sim _{r\longrightarrow 0}(-a_{1}r^{-\frac{%
p+1}{pq-1}},-a_{2}r^{-\frac{q+1}{pq-1}}),\quad \lim_{r\longrightarrow \infty
}r^{N-1}u_{1}^{\prime }=-C_{1}<0,\quad \lim_{r\longrightarrow \infty
}r^{N-1}(\ln r)^{-p}u_{2}^{\prime }=-C_{2}<0.
\end{equation*}%
(5) For $q=p=\frac{N}{N-1}$ ( where $\mathbf{p}=\mathbf{N}=-\sigma )$ we
find again the solutions given at (\ref{logi}) and no other local solution.
\end{theorem}

\begin{proof}
We deduce (1) from Theorem \ref{qccrit}, since $\mathbf{q}=\mathbf{q}_{c},$
then (2) from Theorem \ref{sigmap}, since $\mathbf{N>p}=-\sigma $, in turn
(3) from Theorem \ref{pegaln} since $\mathbf{p}=\mathbf{N}>-\sigma $, and
(4) from Theorem \ref{sigman} where $-\sigma =\mathbf{N>p}$, and finally (5) %
\ref{pnsigma} where $\mathbf{p}=\mathbf{N}=-\sigma .$ There exists no other
solution near $0$ or $\infty ,$ see Remark \ref{tric}.
\end{proof}

\subsection{Local radial existence results for system (\protect\ref{one})}

Next we deduce local existence results for system (\ref{one}), obtained from
Theorem \ref{locsol} by integration. We first give a result of existence of
local solutions near $0$, with a behaviour of linear type. For simplicity,
due to the great number of possibilities, we consider only the positive
solutions of the system. One can formulate analogous results for systems (%
\ref{abs}) and (\ref{mixt}).

\begin{theorem}
\label{Dirac} Existence of solutions of system (\ref{sou}) in $%
B_{r_{0}}\backslash \left\{ 0\right\} $, $r_{0}>0:$

$\bullet $ in regions $\mathcal{A}$ or $\mathcal{B},$ there exist solutions
such that 
\begin{equation*}
\lim_{r\longrightarrow 0}r^{N-2}u_{1}=c_{1}>0,\qquad \lim_{r\longrightarrow
0}r^{N-2}u_{2}=0,\quad \left\{ 
\begin{array}{c}
\lim_{r\rightarrow 0}r^{(N-1)q-2}u_{2}=c(c_{1}),\text{ if }q>\frac{2}{N-1},
\\ 
\lim_{r\rightarrow 0}u_{2}=c_{2},\text{ \quad \quad if }q<\frac{2}{N-1}, \\ 
\lim_{r\rightarrow 0}(\left\vert \ln r\right\vert ^{-1}u_{2})=c_{2},\text{
if }q=\frac{2}{N-1};%
\end{array}%
\right.
\end{equation*}%
$\bullet $ in region $\mathcal{B},$ there exist solutions such that 
\begin{equation*}
\lim_{r\longrightarrow 0}r^{N-2}u_{1}=0,\quad \lim_{r\longrightarrow
0}r^{(N-1)p-2}u_{1}=c(c_{2}),\qquad \lim_{r\longrightarrow
0}r^{N-2}u_{2}=c_{2}>0,
\end{equation*}%
and solutions such that 
\begin{equation*}
\lim_{r\longrightarrow 0}r^{N-2}u_{1}=c_{1}>0,\qquad \lim_{r\longrightarrow
0}r^{N-2}u_{2}=c_{2}>0.
\end{equation*}%
In any case the solutions satisfy the equations in the sense of
distributions in $B_{r_{0}},$ where $\delta _{0}$ is the Dirac mass at $0$: 
\begin{equation}
\left\{ 
\begin{array}{c}
{-\Delta u}_{1}=\left\vert u_{2}^{\prime }\right\vert ^{p}+C_{1}\delta _{0},
\\ 
{-\Delta u}_{2}=\left\vert u_{1}^{\prime }\right\vert ^{q}+C_{2}\delta _{0},%
\end{array}%
\right.  \label{distri}
\end{equation}%
with respectively $(C_{1},C_{2})=(c_{N}c_{1},0)$, $%
(C_{1},C_{2})=(0,c_{N}c_{2}),$ $(C_{1},C_{2})=(c_{N}c_{1},c_{N}c_{2}).%
\medskip \ $

$\bullet $ There is no radial supersolutions of system (\ref{distri}) such
that $C_{1}>0$ or $C_{2}>0$ for $p\geq q\geq q_{2},$ and no supersolutions
such that $C_{2}>0$ for $p\geq \frac{N}{N-1}.$
\end{theorem}

\begin{proof}
The existence of the solutions is a direct consequence of Theorem \ref%
{locsol}. and the behaviour follows from Brezis-Lions Lemma applied to $%
u_{i},$ $i=1,2.$ Next consider any radial supersolution $(u_{1},u_{2})$ and
assume $q\geq q_{2}.$ If $C_{1}>0,$ then ${u}_{1}\geq cr^{2-N}$ and $%
\left\vert u_{1}^{\prime }\right\vert \geq cr^{1-N}$ near $0,$ and $%
\left\vert u_{1}^{\prime }\right\vert ^{q}\in L_{loc}^{1}(B_{r_{0}}),$ thus $%
q<\frac{N}{N-1},$ and ${-(r}^{N-1}u_{2}^{\prime })^{\prime }\geq
cr^{(N-1)(1-q)};$ then by integration we deduce that $\lim_{r\longrightarrow
0}{r}^{N-1}\left\vert u_{2}^{\prime }\right\vert =l\geq 0,$ and $\left\vert
u_{2}^{\prime }\right\vert \geq \frac{r^{1-(N-1)q}}{N-(N-1)q}$ if $l=0.$ In
any case, $\left\vert u_{2}^{\prime }\right\vert \geq cr^{1-(N-1)q}.$ Then ${%
-(r}^{N-1}u_{1}^{\prime })^{\prime }\geq cr^{(1-(N-1)q)p},$ thus $%
r^{(1-(N-1)q)p}\in L_{loc}^{1}(B_{r_{0}}),$ so that $N+p-(N-1)pq>0,$ which
is $q<q_{2}.$ Similarly, if $C_{2}>0,$ then $q<q_{1};$ moreover $\left\vert
u_{2}^{\prime }\right\vert \geq cr^{1-N}$ near $0$, thus $\left\vert 
\overline{{u}_{2}}^{\prime }\right\vert ^{p}\geq Cr^{(1-N)p},$ and $%
\left\vert u_{1}^{\prime }\right\vert ^{p}\in L_{loc}^{1}(B_{r_{0}}),$ thus $%
N-(N-1)p>0.$\medskip 
\end{proof}

In the same way we obtain existence of solutions in an exterior domain with
a linear type:

\begin{theorem}
\label{exterior}Existence of solutions of system (\ref{sou}) in $\mathbb{R}%
^{N}\backslash \overline{B_{r_{0}}}$:

$\bullet $ in region $\mathcal{C},$ there exist solutions such that 
\begin{equation*}
\lim_{r\longrightarrow \infty }r^{N-2}u_{1}=c_{1}>0,\qquad \left\{ 
\begin{array}{c}
\lim_{r\rightarrow \infty }r^{(N-1)q-2}u_{2}=\pm c(c_{1}),\text{ if }q<\frac{%
2}{N-1}, \\ 
\lim_{r\rightarrow \infty }u_{2}=c_{2}>0,\quad \quad \text{ if }q>\frac{2}{%
N-1}, \\ 
\lim_{r\rightarrow \infty }(\left\vert \ln r\right\vert ^{-1}u_{2})=c(c_{1}),%
\text{ if }q=\frac{2}{N-1},%
\end{array}%
\right.
\end{equation*}

$\bullet $ in regions $\mathcal{C}$ and $\mathcal{D},$ there exist solutions
such that 
\begin{equation*}
\lim_{r\longrightarrow \infty }r^{N-2}u_{1}=c_{1}>0,\qquad
\lim_{r\longrightarrow \infty }r^{N-2}u_{2}=c_{2}>0,
\end{equation*}

$\bullet $ in regions $\mathcal{A}$ with $q>q_{1}$, and $\mathcal{D},$ there
exist solutions such that 
\begin{equation*}
\lim_{r\longrightarrow \infty }u_{1}=c_{1}>0,\qquad \lim_{r\longrightarrow
\infty }r^{N-2}u_{2}=c_{2}>0.
\end{equation*}
\end{theorem}

\subsection{Global radial existence and behaviour}

Next we study all the \textit{global} constant sign solutions $(u_{1},u_{2})$
of system (\ref{one}), of any sign. Since they are obtained by integration
of the derivates, we are lead to divide some of the regions $\mathcal{A},%
\mathcal{B},\mathcal{C},\mathcal{D}$ into subregions, according to the
position of $q$ with respect to $q_{3},q_{4}.$\medskip

\begin{definition}
\label{ijk}We set $\mathcal{A}=\mathcal{A}_{1}\cup \mathcal{A}_{2}\cup 
\mathcal{A}_{3},\mathcal{C=C}_{1}\cup \mathcal{C}_{2}\cup \mathcal{C}_{3},%
\mathcal{D}=\mathcal{D}_{1}\cup \mathcal{D}_{2}\cup \mathcal{D}_{3},$ where%
\begin{equation*}
\begin{array}{c}
\mathcal{A}_{1}=\left\{ p>\frac{N}{N-1},q<q_{1}\right\} ,\quad \mathcal{A}%
_{2}=\left\{ q_{1}<q<\min (q_{2},q_{3})\right\} ,\quad \mathcal{A}%
_{3}=\left\{ q_{3}<q<q_{2}\right\} , \\ 
\mathcal{C}_{1}=\left\{ q_{2}<q<\min (q_{3},\frac{N}{N-1})\right\} ,\;%
\mathcal{C}_{2}=\left\{ \max (q_{2},q_{3})<q<\min (q_{4},\frac{N}{N-1}%
)\right\} ,\;\mathcal{C}_{3}=\left\{ q_{4}<q<\frac{N}{N-1}\right\} , \\ 
\mathcal{D}_{1}=\left\{ \frac{N}{N-1}<q<q_{3}\right\} ,\quad \mathcal{D}%
_{2}=\left\{ \max (q_{3},\frac{N}{N-1})<q<q_{4}\right\} ,\quad \mathcal{D}%
_{3}=\left\{ \max (q_{4},\frac{N}{N-1})<q\right\} .%
\end{array}%
\end{equation*}%
\FRAME{ftbpF}{6.2167in}{4.5418in}{0in}{}{}{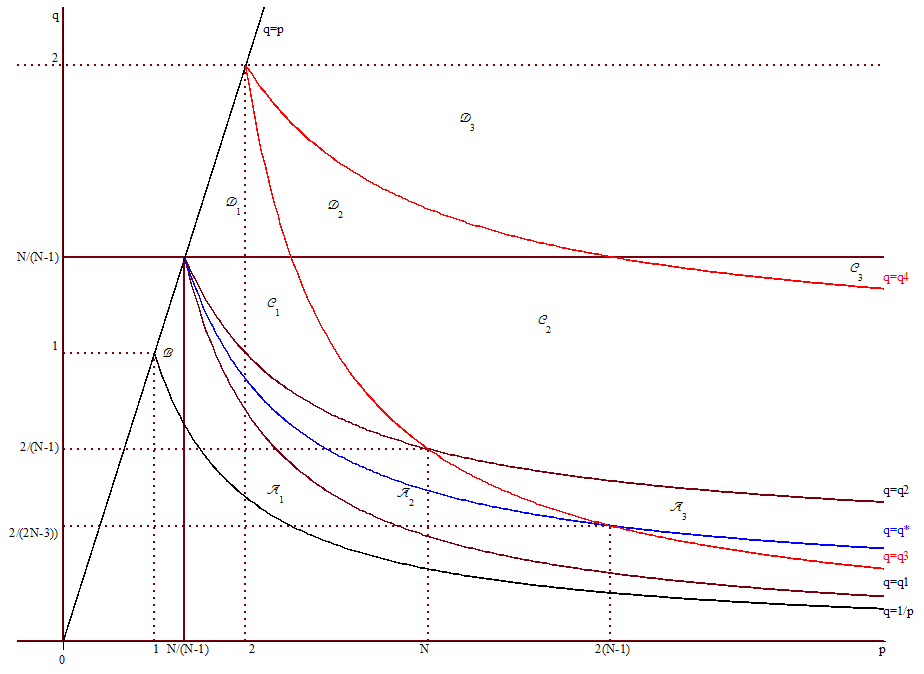}{\special%
{language "Scientific Word";type "GRAPHIC";maintain-aspect-ratio
TRUE;display "USEDEF";valid_file "F";width 6.2167in;height 4.5418in;depth
0in;original-width 5.909in;original-height 4.3096in;cropleft "0";croptop
"1";cropright "1";cropbottom "0";filename
'completefinal.png';file-properties "XNPEU";}}
\end{definition}

We first study the problem with source terms (\ref{sou}) in case $q<q_{4}:$

\begin{theorem}
\label{thsou}For $(p,q)\in \mathcal{C}_{1}\cup \mathcal{D}_{1},$ up to
positive constants, system (\ref{sou}) admits two types of global
nonconstant solutions:

$\bullet $ $(u_{1}^{\ast },u_{2}^{\ast })$, with both components $\infty $%
-singular,

$\bullet $ if $(p,q)\in \mathcal{C}_{1},$ solutions such that 
\begin{equation*}
(u_{1},u_{2})\sim _{r\rightarrow 0}(u_{1}^{\ast },u_{2}^{\ast }),\qquad
\lim_{r\rightarrow \infty }r^{N-2}u_{1}=c_{1}>0,\qquad \lim_{r\rightarrow
\infty }r^{(N-1)q-2}u_{2}=c(c_{1})>0,
\end{equation*}

$\bullet $ if $(p,q)\in \mathcal{D}_{1},$ solutions such that 
\begin{equation*}
(u_{1},u_{2})\sim _{r\rightarrow 0}(u_{1}^{\ast },u_{2}^{\ast }),\qquad
\lim_{r\rightarrow \infty }r^{N-2}u_{1}=c_{1}>0,\qquad \lim_{r\rightarrow
\infty }r^{N-2}u_{2}=c_{2}>0.
\end{equation*}

\noindent There is no global nonconstant solution in the other regions such
that $q<q_{4}.$
\end{theorem}

\begin{proof}
Region $\mathcal{C}_{1}\cup \mathcal{D}_{1}$ is contained in $\mathcal{C\cup
D},$ so we can apply Theorem \ref{regA},(4). In any case $\lim
{}r^{N-1}u_{1}^{\prime }=-C_{1}<0,$ so $u_{1}^{\prime }$ is integrable at $%
\infty $ since $N>2,$ and the same holds for $u_{2}^{\prime }$ for $q>\frac{N%
}{N-1}.$ If $q<\frac{N}{N-1},\lim r^{(N-1)q-1}u_{2}^{\prime }=-C_{2}<0$ also
implies that $u_{2}^{\prime }$ is integrable; Then the functions 
\begin{equation*}
u_{1}(r)=-\int_{r}^{\infty }u_{1}^{\prime }(\tau )d\tau ,\qquad
u_{2}(r)=-\int_{r}^{\infty }u_{1}^{\prime }(\tau )d\tau ,
\end{equation*}%
satisfy the conclusions. Note that $q>\frac{2}{N-1}$ in $\mathcal{C}_{1}\cup 
\mathcal{D}_{1}.$\medskip

Next we note that there is no global solution of (\ref{sou}) for $q<q_{2}$ :
indeed if such solution exists then $u_{1}^{\prime }$ and $u_{2}^{\prime }$
are negative; but all the global solutions are given at Theorems \ref{regA}
and \ref{limcas}, and they do not fulfil these conditions.Then we are lead
to consider the regions $\mathcal{C}_{2},\mathcal{D}_{2}$. From Theorems \ref%
{white} and \ref{yellow}, there is no global solution $w$ on $(0,\infty )$
for $\varepsilon =1,$ and the global solutions relative to $\varepsilon =-1$
do not bring solutions of system (\ref{sou})$.$ Indeed near $0$, the
function $w$ behaves as $w^{\ast },$ and the corresponding solutions $%
(u_{1}^{\ast },u_{2}^{\ast })$ are such that $u_{2}^{\ast }$ is not
positive.\medskip
\end{proof}

Next we consider the problem with absorption, for which the situation is
very rich. Using Theorem \ref{regA} and similar arguments of integrability,
we obtain the following:

\begin{theorem}
\label{thabs}Consider the system (\ref{abs}) For $(p,q)\in \mathcal{B}\cup 
\mathcal{A}_{1}\cup \mathcal{A}_{3}$ it admits a particular solution $(%
\widetilde{u_{1}}^{\ast },\widetilde{u_{2}}^{\ast }).$

(1) If $(p,q)\in \mathcal{A}_{1},$ $\widetilde{u_{1}}^{\ast },\widetilde{%
u_{2}}^{\ast }$ are $\infty $-singular, and there exist solutions such that 
\begin{equation*}
\lim_{r\rightarrow 0}r^{(N-1)p-2}\widetilde{u_{1}}=c(c_{2})>0,\qquad
\lim_{r\rightarrow 0}r^{N-2}\widetilde{u_{2}}^{\ast }=c_{2}>0,\qquad
(u_{1},u_{2})\sim _{r\rightarrow \infty }(u_{1}^{\ast },u_{2}^{\ast }).
\end{equation*}

(2) If $(p,q)\in \mathcal{B},$ $\widetilde{u_{1}}^{\ast },\widetilde{u_{2}}%
^{\ast }$ are still $\infty $-singular, and there exist solutions such that 
\begin{equation*}
\lim_{r\rightarrow 0}r^{N-2}\widetilde{u_{1}}=c_{1}>0,\qquad
\lim_{r\rightarrow 0}r^{N-2}\widetilde{u_{2}}=c_{2}>0,\qquad
(u_{1},u_{2})\sim _{r\rightarrow \infty }(u_{1}^{\ast },u_{2}^{\ast }).
\end{equation*}

(3) If $(p,q)\in \mathcal{A}_{3},$ $\widetilde{u_{1}}^{\ast }$ is $\infty $%
-singular, and $\widetilde{u_{2}}^{\ast }$ is a cusp-solution; moreover

$\bullet $ if $q<q^{\ast }$ there exist solutions such that 
\begin{equation*}
(\widetilde{u_{1}},\widetilde{u_{2}})\sim _{r\rightarrow 0}(\widetilde{u_{1}}%
,\widetilde{u_{2}}^{\ast }),\qquad \lim_{r\rightarrow \infty }r^{(N-1)p-2}%
\widetilde{u_{1}}=c(c_{2})>0,\qquad \lim_{r\rightarrow \infty }\widetilde{%
u_{2}}=c_{2}>0,\lim_{r\rightarrow \infty }r^{N-1}\widetilde{u_{2}}^{\prime
}=k>0,
\end{equation*}%
so the function $\widetilde{u_{2}}$ varies from $0$ to $C_{2}.$

$\bullet $ if $q>q^{\ast }$ there exist solutions such that 
\begin{equation*}
\lim_{r\rightarrow 0}r^{N-2}\widetilde{u_{1}}=c_{1}>0,\qquad
\lim_{r\rightarrow 0}r^{(N-1)q-2}\widetilde{u_{2}}=c(c_{1})>0,\qquad (%
\widetilde{u_{1}},\widetilde{u_{2}})\sim _{r\rightarrow \infty }(\widetilde{%
u_{1}},\widetilde{u_{2}}^{\ast }),
\end{equation*}%
and there are explicit solutions if $q=q^{\ast }.$\medskip

\noindent There is no global nonconstant solution in the other regions such
that $q<q_{4}.$
\end{theorem}

Now we study the system (\ref{mixt}), also of a great richness, in
particular involving the Sobolev exponent $q^{\ast }$

\begin{theorem}
\label{thmix}Consider the system (\ref{mixt}). For $(p,q)\in \mathcal{A}%
_{2}\cup \mathcal{C}_{2}\cup \mathcal{D}_{2}$ it admits a particular
solution $(\widehat{u_{1}}^{\ast },\widehat{u_{2}}^{\ast }).$

(1) Assume $(p,q)\in \mathcal{A}_{2}.$ Then $\widehat{u_{1}}^{\ast },%
\widehat{u_{2}}^{\ast }$ are $\infty $-singular; moreover

$\bullet $ if $q<q^{\ast },$ there exist solutions such that 
\begin{equation*}
(\widehat{u_{1}},\widehat{u_{2}})\sim _{r\rightarrow 0}(\widehat{u_{2}}%
^{\ast },\widehat{u_{2}}^{\ast }),\qquad \lim_{r\rightarrow \infty
}r^{(N-1)p-2}\widehat{u_{1}}=c(c_{2})>0,\qquad \lim_{r\rightarrow \infty
}r^{N-2}\widehat{u_{2}}=c_{2}>0;
\end{equation*}

$\bullet $ if $q>q^{\ast },$there exist solutions such that%
\begin{eqnarray*}
\lim_{r\rightarrow 0}r^{N-2}\widehat{u_{1}} &=&c_{1}>0,\qquad \left\{ 
\begin{array}{c}
\lim_{r\rightarrow 0}\widehat{u_{2}}=c_{2}>0,\;\lim_{r\rightarrow
0}r^{(N-1)q-1}\widehat{u_{2}}^{\prime }=-c(c_{1})<0,\text{ if }q<\frac{2}{N-1%
}, \\ 
\lim_{r\rightarrow 0}r^{(N-1)q-2}\widehat{u_{2}}=c(c_{1})>0,\quad \quad 
\text{ if }q>\frac{2}{N-1},%
\end{array}%
\right. \\
(\widehat{u_{1}},\widehat{u_{2}}) &\sim &_{r\rightarrow \infty }(\widehat{%
u_{2}}^{\ast },\widehat{u_{2}}^{\ast }),
\end{eqnarray*}

$\bullet $ if $q=q^{\ast },$ there exist explicit solutions, given by
integration of (\ref{exact}).\medskip

(2) Assume $(p,q)\in \mathcal{C}_{2},$ $q\neq \frac{2}{N-1}.$ Then $\widehat{%
u_{1}}^{\ast }$ is $\infty $-singular and $\widehat{u_{2}}^{\ast }$ is a
cusp-solution.\ There exist solutions such that%
\begin{eqnarray*}
(\widehat{u_{1}},\widehat{u_{2}}) &\sim &_{r\rightarrow 0}(\widehat{u_{2}}%
^{\ast },\widehat{u_{2}}^{\ast }), \\
\lim_{r\rightarrow \infty }r^{N-2}\widehat{u_{1}} &=&c_{1}>0,\qquad \left\{ 
\begin{array}{c}
\lim_{r\rightarrow \infty }r^{(N-1)q-2}\widehat{u_{2}}=c(c_{1})>0,\quad
\quad \text{ if }q<\frac{2}{N-1}, \\ 
\lim_{r\rightarrow \infty }\widehat{u_{2}}=c_{2}>0,\lim_{r\rightarrow \infty
}r^{(N-1)q-1}\widehat{u_{2}}^{\prime }=-c(c_{1})<0,\text{ if }q>\frac{2}{N-1}%
.%
\end{array}%
\right.
\end{eqnarray*}

(3) Assume $(p,q)\in \mathcal{D}_{2}$. Then again $\widehat{u_{1}}^{\ast }$
is $\infty $-singular and $\widehat{u_{2}}^{\ast }$ is a cusp-solution.There
exist solutions such that 
\begin{equation*}
(\widehat{u_{1}},\widehat{u_{2}})\sim _{r\rightarrow 0}(\widehat{u_{2}}%
^{\ast },\widehat{u_{2}}^{\ast }),\qquad \lim_{r\rightarrow \infty }r^{N-2}%
\widehat{u_{1}}=c_{1}>0,\qquad \lim_{r\rightarrow \infty }\widehat{u_{2}}%
=c(c_{2})>0,\quad \lim_{r\rightarrow \infty }r^{N-1}\widehat{u_{2}}^{\prime
}=-c_{2}<0.
\end{equation*}%
\noindent There is no global nonconstant solution in the other regions such
that $q<q_{4}$.
\end{theorem}

The region $\left\{ q>q_{4}\right\} ,$ containing $\mathcal{C}_{3}\cup 
\mathcal{D}_{3}$, plays a particular role, as in the scalar case for $q>2,$
because we can construct solutions such that $u_{1}$ and $u_{2}$ are \textbf{%
bounded} on $(0,\infty ),$ then for any of the three systems (\ref{sou}), (%
\ref{abs}) and (\ref{mixt}) by adding suitable constants.\bigskip

\begin{theorem}
\label{thall}Let $(p,q)\in \mathcal{C}_{3}\cup \mathcal{D}_{3}=\left\{
q>q_{4}\right\} .$ (i)Then system (\ref{abs}) still admits a particular
solution $(\widetilde{u_{1}}^{\ast },\widetilde{u_{2}}^{\ast }),$ where both
components are cusp-solutions. Moreover there exist \textbf{bounded} global
solutions on $(0,\infty ):$

$\bullet $ If $(p,q)\in \mathcal{C}_{3},$ there exist solutions such that 
\begin{eqnarray*}
(\widetilde{u_{1}},\widetilde{u_{2}}) &\sim &_{r\rightarrow 0}(\widetilde{%
u_{1}}^{\ast },\widetilde{u_{2}}^{\ast }), \\
\lim_{r\rightarrow \infty }\widetilde{u_{1}} &=&c_{1}>0,\quad
\lim_{r\rightarrow \infty }\widetilde{u_{2}}=c_{2}>0,\quad
\lim_{r\rightarrow \infty }r^{N-1}\widetilde{u_{1}}^{\prime }=k_{1}>0,\quad
\lim_{r\rightarrow \infty }r^{(N-1)q-2}\widetilde{u_{2}}^{\prime
}=c(k_{1})>0.
\end{eqnarray*}

$\bullet $ If $(p,q)\in \mathcal{D}_{3},$ there exist solutions such that 
\begin{equation*}
(\widetilde{u_{1}},\widetilde{u_{2}})\sim _{r\rightarrow 0}(\widetilde{u_{1}}%
,\widetilde{u_{2}}^{\ast }),\qquad \lim_{r\rightarrow \infty }\widetilde{%
u_{i}}=c_{i}>0,\quad \lim_{r\rightarrow \infty }r^{N-1}\widetilde{u_{1}}%
^{\prime }=k_{1}>0,\quad \lim_{r\rightarrow \infty }r^{N-1}\widetilde{u_{2}}%
^{\prime }=k_{2}>0.
\end{equation*}%
(ii) System (\ref{sou}) admits global bounded solutions in $\mathcal{C}_{3}$
or $\mathcal{D}_{3}$, of the form $(C_{1}-\widetilde{u_{1}},C_{2}-\widetilde{%
u_{2}})$ for $C_{1}>c_{1},C_{2}>c_{2}.\medskip $

\noindent (iii) System (\ref{mixt}) admits global bounded solutions of the
same form for $C_{1}>c_{1},C_{2}<0.\medskip $
\end{theorem}

\begin{proof}[Proof of Theorem \protect\ref{posi}]
The proofs of (1) (2) are direct consequences of Theorem \ref{Dirac},\ref%
{thsou} and \ref{thall} respectively, from the definition \ref{ijk} of the
regions.
\end{proof}

\begin{remark}
The limit cases $p=\frac{N}{N-1}$ and $q=q_{1},q_{2},\frac{N}{N-1}$ can be
deduced from Remark \ref{limcas} by integration. We leave the computations
to the reader. In the limit cases $q=q_{3}\neq q_{2},q_{4}$, the particular
solutions have a logarithmic form: If $p>q=q_{3}\neq q_{2}$, and $q$ we
obtain solutions such that $u_{1}^{\ast \prime }=-a_{1}r^{1-p}$ and $%
u_{2}^{\ast \prime }=-a_{2}r^{-1},$ so we get non constant sign solutions of
the form 
\begin{equation}
u_{1}^{\ast }=\pm a_{1}r^{2-p}+c_{1},\qquad u_{2}^{\ast }=-a_{2}\ln r+c_{2}.
\end{equation}%
If $q=q_{4}<p$ (hence $p>2$) we get $u_{1}^{\ast \prime }=-a_{1}r^{-1}$ and $%
u_{2}^{\ast \prime }=-a_{2}r^{-\frac{2}{p}},$ then we find nonconstant sign
solutions of the form 
\begin{equation*}
u_{1}^{\ast }=a_{1}\ln r+c_{2},\qquad u_{2}^{\ast }=-\frac{p}{p-2}a_{2}r^{%
\frac{p-2}{p}}+c_{1}.
\end{equation*}%
And for $p=q=2,$ we recall the solutions $u_{1}^{\ast }=u_{2}^{\ast
}=(2-N)\ln r+c$. in these limit cases, there is no global nonconstant
solution.\medskip

Finally when $q=q_{2}=q_{3},$ equivalently $(p,q)=(N,\frac{2}{N-1}),$ the
behaviours of type (\ref{kil}) (\ref{kol}) near $0$ or $\infty $ resumes to $%
\lim r\left\vert \ln r\right\vert ^{\frac{2}{N+1}}u_{2}^{\prime }=-c_{2}<0,$
which gives by integration solutions such that 
\begin{equation*}
\lim \ln (\left\vert \ln r\right\vert ^{\frac{2}{N+1}})u_{2}=-C<0.
\end{equation*}
\end{remark}

\section{Nonradial case: upper estimates and local behaviour \label{nonrad}}

\subsection{Nonexistence results, and upperestimates of mean values}

Here we give upperestimates for the supersolutions of system (\ref{one}),
implying in particular nonexistence results of entire solutions of the
system. the nonexistence results obtained by integral methods were initiated
by \cite{MiPo} for the positive supersolutions of the Lane-Emden system (\ref%
{LE}) where $a=b=0$, and then extended in various directions, to more
general operators and second members in \cite{Fi}, involving quasilinear
operators and gradient terms, of type 
\begin{equation}
\left\{ 
\begin{array}{c}
{-\Delta }_{P}^{N}{u}_{1}\geq u_{1}^{a}{u}_{2}^{b}\left\vert \nabla
u_{2}\right\vert ^{p}, \\ 
{-\Delta }_{Q}^{N}{u}_{2}\geq {u}_{1}^{c}u_{2}^{d}\left\vert \nabla
u_{1}\right\vert ^{q},%
\end{array}%
\right.  \label{fili}
\end{equation}%
where $P,Q>1,$ the solutions are positive, and $b,c>0$. In \cite{BiPo} we
also obtained integral estimates of the positive solutions, for example on
problems of type (\ref{fili}) with $p=q=0.$ In the situation of system (\ref%
{one}) we adapt the methods of \cite{BiPo}, and obtain integral
upperestimates of the gradient. A noticeable fact is that the solutions are
not supposed to be positive, and not even of constant sign.

\begin{definition}
We say that a couple $(u_{1},u_{2})$ of $C^{2}$ function in a domain $\Omega
\subset \mathbb{R}^{N}$ is a supersolution (resp. a subsolution) of system (%
\ref{one}) if 
\begin{equation}
\left\{ 
\begin{array}{c}
{-\Delta u}_{1}\geq \left\vert \nabla u_{2}\right\vert ^{p}, \\ 
{-\Delta u}_{2}\geq \left\vert \nabla u_{1}\right\vert ^{q},%
\end{array}%
\quad \text{in }\Omega \right. ,\qquad \text{resp.}\left\{ 
\begin{array}{c}
{-\Delta u}_{1}\leq \left\vert \nabla u_{2}\right\vert ^{p}, \\ 
{-\Delta u}_{2}\leq \left\vert \nabla u_{1}\right\vert ^{q},%
\end{array}%
\quad \text{in }\Omega .\right.  \label{ineg}
\end{equation}
\end{definition}

\begin{theorem}
\label{esti}Let $pq>1,$ $p\geq q\geq 1,$and $(u_{1},u_{2})$ be any
supersolution of system (\ref{one}) (with no condition of sign) in a domain $%
\Omega \subset \mathbb{R}^{N}$.

(i) If $\Omega =B_{r_{0}}\backslash \left\{ 0\right\} $, resp $\Omega =%
\mathbb{R}^{N}\backslash \overline{B_{r_{0}}}$, then there exist $%
C=C(N,p,q)>0$ such that for any $R<\frac{r_{0}}{2}$ (resp. $R>2r_{0}$) 
\begin{equation}
\dint\limits_{\frac{R}{2}\leq \left\vert x\right\vert \leq \frac{3R}{2}%
}\left\vert \nabla u_{2}\right\vert ^{p}dx\leq CR^{N-\frac{p(q+1)}{pq-1}%
},\qquad \dint\limits_{\frac{R}{2}\leq \left\vert x\right\vert \leq \frac{3R%
}{2}}\left\vert \nabla u_{1}\right\vert ^{q}dx\leq CR^{N-\frac{q(p+1)}{pq-1}%
}.  \label{mil}
\end{equation}%
(ii) If $\Omega =\mathbb{R}^{N}$ and 
\begin{equation}
(N-1)pq<\max (N+p,N+q)=N+p.  \label{scr}
\end{equation}%
that means $q<q_{2},$\textbf{\ }then all the \textbf{solutions} of system (%
\ref{one}) are constant.
\end{theorem}

\begin{proof}
(i) Let $R>0$ and $x_{0}\in \Omega $ such that $\overline{B(x_{0},2R)}%
\subset \Omega .$ $\ $Let $\zeta \in C_{0}^{\infty }(\mathbb{R}^{N}),$ with
values in $\left[ 0,1\right] ,$ such that $\zeta =1$ on $B_{\frac{R}{2}},$ $%
\zeta =0$ on $B_{R}^{C}$ and $\left\vert \nabla \zeta \right\vert \leq \frac{%
2}{R}.$ We take as test function $\varphi (x)=\zeta ^{\lambda }(x-x_{0})$, $%
\lambda >0,$ in the first inequality and get by integration, for any \textbf{%
\ }$p>1$ and $\alpha >0,$ 
\begin{align*}
\int_{B(x_{0},R)}\left\vert \nabla u_{1}\right\vert ^{q}\zeta ^{\lambda }dx&
\leq \lambda \left\vert \int_{B(x_{0},R)}<\nabla u_{2},\nabla \zeta >\zeta
^{\lambda -1}dx\right\vert =\lambda \left\vert \int_{B(x_{0},R)}\left\vert
\nabla u_{2}\right\vert \zeta ^{\alpha }\left\vert \nabla \zeta \right\vert
\zeta ^{\lambda -1-\alpha }dx\right\vert \\
& \leq \lambda (\int_{B(x_{0},R)}\left\vert \nabla u_{2}\right\vert
^{p}\zeta ^{\alpha p}dx)^{\frac{1}{p}}(\int_{B(x_{0},R)}\left\vert \nabla
\zeta \right\vert ^{p^{\prime }}\zeta ^{(\lambda -1-\alpha )p^{\prime }}dx)^{%
\frac{1}{p^{\prime }}};
\end{align*}%
then for $\lambda \geq 1+\alpha ,$%
\begin{equation*}
\int_{B(x_{0},R)}\left\vert \nabla u_{1}\right\vert ^{q}\zeta ^{\lambda
}dx\leq \lambda (\int_{B(x_{0},R)}\left\vert \nabla u_{2}\right\vert
^{p}\zeta ^{\alpha p}dx)^{\frac{1}{p}}(\int_{B(x_{0},R)}\left\vert \nabla
\zeta \right\vert ^{p^{\prime }}dx)^{\frac{1}{p^{\prime }}}.
\end{equation*}%
Similarly for given $\beta >0,\mu >0$ such that $\mu \geq 1+\beta ,$ 
\begin{equation*}
\int_{B(x_{0},R)}\left\vert \nabla u_{2}\right\vert ^{p}\zeta ^{\mu }dx\leq
\mu (\int_{B(x_{0},R)}\left\vert \nabla u_{1}\right\vert ^{q}\xi ^{\beta
q}dx)^{\frac{1}{p}}(\int_{B(x_{0},R)}\left\vert \nabla \zeta \right\vert
^{q^{\prime }}dx)^{\frac{1}{q^{\prime }}}.
\end{equation*}%
Taking $\lambda =\frac{q(p+1)}{pq-1},$ $\alpha =\frac{q+1}{pq-1}$ $\mu =%
\frac{p(q+1)}{pq-1}$ and $\beta =\frac{p+1}{pq-1},$ we find 
\begin{equation*}
\int_{B(x_{0},R)}\left\vert \nabla u_{1}\right\vert ^{q}\zeta ^{\frac{q(p+1)%
}{pq-1}}dx\leq \lambda (\int_{B(x_{0},R)}\left\vert \nabla u_{2}\right\vert
^{p}\zeta ^{\frac{p(q+1)}{pq-1}}dx)^{\frac{1}{p}}(\int_{B(x_{0},R)}\left%
\vert \nabla \zeta \right\vert ^{p^{\prime }}dx)^{\frac{1}{p^{\prime }}},
\end{equation*}%
and by symmetry, when $q>1,$ 
\begin{equation*}
\int_{B(x_{0},R)}\left\vert \nabla u_{2}\right\vert ^{p}\xi ^{\frac{p(q+1)}{%
pq-1}}dx\leq \mu (\int_{B(x_{0},R)}\left\vert \nabla u_{1}\right\vert
^{q}\xi ^{\frac{q(p+1)}{pq-1}}dx)^{\frac{1}{q}}(\int_{B(x_{0},R)}\left\vert
\nabla \zeta \right\vert ^{q^{\prime }}dx)^{\frac{1}{q^{\prime }}}.
\end{equation*}%
As a consequence, 
\begin{equation*}
\int_{B(x_{0},R)}\left\vert \nabla u_{2}\right\vert ^{p}\xi ^{\frac{p(q+1)}{%
pq-1}}dx\leq \mu \lambda ^{\frac{1}{q}}(\int_{B(x_{0},R)}\left\vert \nabla
u_{2}\right\vert ^{p}\xi ^{\frac{p(q+1)}{pq-1}}dx)^{\frac{1}{pq}%
}(\int_{B(x_{0},R)}\left\vert \nabla \zeta \right\vert ^{q^{\prime }}dx)^{%
\frac{1}{q^{\prime }}}(\int_{B(x_{0},R)}\left\vert \nabla \zeta \right\vert
^{p^{\prime }}dx)^{\frac{1}{qp^{\prime }}},
\end{equation*}%
\begin{eqnarray*}
(\int_{B(x_{0},\frac{R}{2})}\left\vert \nabla u_{2}\right\vert ^{p}dx)^{%
\frac{pq-1}{pq}} &\leq &(\int_{B(x_{0},R)}\left\vert \nabla u_{2}\right\vert
^{p}\xi ^{\frac{p(q+1)}{pq-1}}dx)^{\frac{pq-1}{pq}} \\
&\leq &\mu \lambda ^{\frac{1}{q}}(\int_{B(x_{0},R)}\left\vert \nabla \zeta
\right\vert ^{q^{\prime }}dx)^{\frac{1}{q^{\prime }}}(\int_{B(x_{0},R)}\left%
\vert \nabla \zeta \right\vert ^{p^{\prime }}dx)^{\frac{1}{qp^{\prime }}%
}\leq CR^{\frac{N-q^{\prime }}{q^{\prime }}+\frac{N-p^{\prime }}{qp^{\prime }%
}}=CR^{N\frac{pq-1}{pq}-\frac{q+1}{q}},
\end{eqnarray*}%
for some constant $C=C(N,p,q)>0.$ When $q=1,$ we get directly with $\mu =%
\frac{2p}{p-1}$ 
\begin{eqnarray*}
\int_{B(x_{0},R)}\left\vert \nabla u_{2}\right\vert ^{p}\zeta ^{\frac{2p}{p-1%
}}dx &\leq &\frac{2p}{p-1}\left\vert \int_{B(x_{0},R)}<\nabla u_{1},\nabla
\zeta >\zeta ^{\frac{p+1}{p-1}}dx\right\vert =\frac{2p}{p-1}\left\vert
\int_{B(x_{0},R)}\left\vert \nabla u_{1}\right\vert \zeta ^{\frac{p+1}{p-1}%
}\left\vert \nabla \zeta \right\vert dx\right\vert \\
&\leq &\frac{C}{R}\int_{B(x_{0},R)}\left\vert \nabla u_{1}\right\vert \zeta
^{\frac{p+1}{p-1}},
\end{eqnarray*}%
\begin{equation*}
\int_{B(x_{0},R)}\left\vert \nabla u_{1}\right\vert \zeta ^{\frac{p+1}{p-1}%
}dx\leq \lambda (\int_{B(x_{0},R)}\left\vert \nabla u_{2}\right\vert
^{p}\zeta ^{\frac{2p}{p-1}}dx)^{\frac{1}{p}}(\int_{B(x_{0},R)}\left\vert
\nabla \zeta \right\vert ^{p^{\prime }}dx)^{\frac{1}{p^{\prime }}},
\end{equation*}%
then we obtain 
\begin{equation*}
(\int_{B(x_{0},\frac{R}{2})}\left\vert \nabla u_{2}\right\vert ^{p}dx)^{%
\frac{1}{p^{\prime }}}\leq CR^{\frac{N}{p^{\prime }}-2}.
\end{equation*}%
In any case the estimates (\ref{mil}) follow by considering a finite
recovering of $\left\{ \frac{R}{2}\leq \left\vert x\right\vert \leq \frac{3R%
}{2}\right\} $ by balls $B(x_{0},R).\medskip $

(ii) If $\Omega =\mathbb{R}^{N}$ and $N(pq-1)<p(q+1),$ equivalently (\ref%
{scr}) holds, we consider any ball $B(0,R)$ and make $R\rightarrow \infty ;$
we deduce that $\nabla u_{2}=0,$ hence $u_{2}=C_{2};$ then ${-\Delta u}%
_{2}=0\geq \left\vert \nabla u_{1}\right\vert ^{q},$ thus $%
u_{1}=C_{1}.\medskip $
\end{proof}

\begin{remark}
In case $p=q,$ and $u_{1}\geq 0,u_{2}\geq 0,$ the nonexistence can be
obtained by reducing to the scalar case: 
\begin{equation*}
-\Delta (u_{1}+u_{2})=\left\vert \nabla u_{2}\right\vert ^{q}+\left\vert
\nabla u_{1}\right\vert ^{q}\geq c_{q}^{1}\left( \left\vert \nabla
u_{2}\right\vert +\left\vert \nabla u_{1}\right\vert \right) ^{q}\geq
c_{q}^{2}\left\vert \nabla (u_{1}+u_{2})\right\vert ^{q}.
\end{equation*}%
If $1<q<\frac{N}{N-1},$ then the only nonnegative solutions in whole $%
\mathbb{R}^{N}$ satisfy $u_{1}+u_{2}=C,$ for example from \cite[Proposition
2.1]{BiGaVe2}. Then $-\Delta (u_{1}+u_{2})=0\geq \left\vert \nabla
u_{2}\right\vert ^{q}+\left\vert \nabla u_{1}\right\vert ^{q},$ hence $u_{1}$
and $u_{2}$ are constant. At Theorem \ref{esti}we have shown that the
positivity is not required. And moreover we obtain integral upperestimates
of the gradients when $q\geq \frac{N}{N-1}$.
\end{remark}

In the sequel, the mean value of any function $u$ in $\left\{ R\leq
\left\vert x\right\vert \leq R^{\prime }\right\} $ on the sphere $S^{N-1}$
is denoted by $\overline{u}$.

\begin{lemma}
\label{mean}Let $p\geq q\geq 1$ and $(u_{1},u_{2})$ be a supersolution of
system (\ref{one}) in $\left\{ R\leq \left\vert x\right\vert \leq R^{\prime
}\right\} .$ Then $(\overline{u}_{1},\overline{u}_{2})$ is also a
supersolution. In particular if $(\widetilde{u_{1}},\widetilde{u_{2}})$ is a
positive subsolution of system (\ref{sou}), then $(\overline{\widetilde{u_{1}%
}},\overline{\widetilde{u_{2}}})$ is also a subsolution.
\end{lemma}

\begin{proof}
For any function $u\in C^{1}(B(0,R)),$ and any $p\geq 1,$ from the Jensen
inequality, 
\begin{equation}
\overline{\left\vert \nabla u\right\vert ^{p}}\geq \left\vert \overline{u}%
^{\prime }\right\vert ^{p}.  \label{jen}
\end{equation}%
Indeed 
\begin{eqnarray*}
\overline{\left\vert \nabla u\right\vert ^{p}} &=&\frac{1}{\left\vert
S^{N-1}\right\vert }\int_{S^{N-1}}(u_{2,r}^{2}+r^{-2}\left\vert \nabla
^{\prime }u_{2}\right\vert ^{2})^{\frac{p}{2}}d\sigma \geq \frac{1}{%
\left\vert S^{N-1}\right\vert }\int_{S^{N-1}}\left\vert u_{r}\right\vert
^{p}d\sigma \\
&\geq &(\frac{1}{\left\vert S^{N-1}\right\vert }\int_{S^{N-1}}\left\vert
u_{r}\right\vert d\sigma )^{p}\geq \left\vert \frac{1}{\left\vert
S^{N-1}\right\vert }\int_{S^{N-1}}u_{r}d\sigma \right\vert ^{p}=\left\vert 
\overline{u}^{\prime }\right\vert ^{p}.
\end{eqnarray*}%
Then, since $q\geq 1,$%
\begin{equation*}
-\Delta \overline{u_{1}}\geq \overline{\left\vert \nabla u_{2}\right\vert
^{p}}\geq \left\vert \overline{u_{2}}^{\prime }\right\vert ^{p},\qquad
-\Delta \overline{u_{2}}\geq \overline{\left\vert \nabla u_{1}\right\vert
^{q}}\geq \left\vert \overline{u_{1}}^{\prime }\right\vert ^{q}.
\end{equation*}
\end{proof}

\subsection{Local behaviour near 0 or $\infty $}

As a consequence of Theorem \ref{esti} and Proposition \ref{uprad}, we get

\begin{proposition}
\label{esto}Let $pq>1,$ $p\geq q\geq 1,$ and $({u}_{1},{u}_{2})$ be any
supersolution of system (\ref{sou}) in $B_{r_{0}}\backslash \left\{
0\right\} .$ Then there exist $C=C(N,p,q)>0$, and $\rho \in (0,r_{0})$
depending on ${u}_{1},{u}_{2}$, such that for $r\in \left( 0,\rho \right) ,$%
s 
\begin{equation*}
\left\vert \overline{u_{1}}(r)\right\vert \leq \left\{ 
\begin{array}{c}
Cr^{-\frac{2-p(q-1)}{pq-1}}\text{ if }q<q_{4}, \\ 
C\left\vert \ln r\right\vert \quad \text{if }q=q_{4}, \\ 
C\text{ \quad if }q>q_{4},%
\end{array}%
\right. \qquad \left\vert \overline{u_{2}}(r)\right\vert \leq \left\{ 
\begin{array}{c}
Cr^{-\frac{2-q(p-1)}{pq-1}}\text{if }q<q_{3}, \\ 
C\left\vert \ln r\right\vert \quad \text{if }q=q_{3}, \\ 
C\text{ if }q>q_{3}.%
\end{array}%
\right.
\end{equation*}%
(ii) $({u}_{1},{u}_{2})$ be any supersolution of system (\ref{sou}) in $%
\mathbb{R}^{N}\backslash \overline{B_{r_{0}}}.$ Then there exists $\eta
>r_{0}$ such that for $r>\eta ,$ 
\begin{equation*}
\left\vert \overline{u_{1}}(r)\right\vert \leq \left\{ 
\begin{array}{c}
C\left\vert x\right\vert ^{\frac{p(q-1)-2}{pq-1}}\text{if }q>q_{4}, \\ 
C\ln r\quad \text{if }q=q_{4}, \\ 
C\text{ if }q<q_{4},%
\end{array}%
\right. \qquad \left\vert \overline{u_{2}}(r)\right\vert \leq \left\{ 
\begin{array}{c}
Cr^{\frac{q(p-1)-2}{pq-1}}\text{if }q>q_{3}, \\ 
C\ln r\quad \text{if }q=q_{3}, \\ 
C\text{ if }q<q_{3}.%
\end{array}%
\right.
\end{equation*}
\end{proposition}

\subsubsection{Case of system (\protect\ref{abs})}

In this case, we can improve the preceeding results. We give Osserman's type
estimates for the local solutions near $0$ or $\infty $

\begin{theorem}
\label{Osserman}Let $pq>1,$ $p\geq q\geq 1.$(i) Let $(\widetilde{u_{1}},%
\widetilde{u_{2}})$ be any positive subsolution of system (\ref{abs}) in $%
B_{r_{0}}\backslash \left\{ 0\right\} .$ Then there exists and $C=C(N,p,q)>0$%
, and $\rho \in (0,r_{0})$ depending on $\widetilde{u_{1}},\widetilde{u_{2}} 
$, such that for $0<\left\vert x\right\vert <\rho ,$ 
\begin{equation*}
\widetilde{u_{1}}(x)\leq \left\{ 
\begin{array}{c}
C\left\vert x\right\vert ^{-\frac{2-p(q-1)}{pq-1}}\text{ if }q<q_{4}, \\ 
C\left\vert \ln \left\vert x\right\vert \right\vert \quad \text{if }q=q_{4},
\\ 
C\text{ if }q>q_{4},%
\end{array}%
\right. \qquad \widetilde{u_{2}}(x)\leq \left\{ 
\begin{array}{c}
C\left\vert x\right\vert ^{-\frac{2-q(p-1)}{pq-1}}\text{if }q<q_{3}, \\ 
C\left\vert \ln \left\vert x\right\vert \right\vert \quad \text{if }q=q_{3},
\\ 
C\text{ if }q>q_{3}.%
\end{array}%
\right.
\end{equation*}%
(ii) Let $(\widetilde{u_{1}},\widetilde{u_{2}})$ be any positive subsolution
of system (\ref{abs}) in $\mathbb{R}^{N}\backslash \overline{B_{r_{0}}}.$
Then there exists $\eta >r_{0}$ such that for $\left\vert x\right\vert >\eta
,$ 
\begin{equation*}
\widetilde{u_{1}}(x)\leq \left\{ 
\begin{array}{c}
C\left\vert x\right\vert ^{\frac{p(q-1)-2}{pq-1}}\text{if }q>q_{4}, \\ 
C\left\vert \ln \left\vert x\right\vert \right\vert \quad \text{if }q=q_{4},
\\ 
C\text{ if }q<q_{4},%
\end{array}%
\right. \qquad \widetilde{u_{2}}(x)\leq \left\{ 
\begin{array}{c}
C\left\vert x\right\vert ^{\frac{q(p-1)-2}{pq-1}}\text{if }q>q_{3}, \\ 
C\left\vert \ln \left\vert x\right\vert \right\vert \quad \text{if }q=q_{3},
\\ 
C\text{ if }q<q_{3}.%
\end{array}%
\right.
\end{equation*}
\end{theorem}

\begin{proof}
(i) Since the functions are subharmonic, if $u_{i}$ is nonconstant near $0,$
then $\overline{\widetilde{u_{i}}}$ is strictly monotone for small $r,$ and
either $r^{N-2}\overline{\widetilde{u_{i}}}$ has a positive limit as $%
r\longrightarrow 0,$ or $\overline{\widetilde{u_{i}}}$ is bounded. And for $%
\left\vert x\right\vert <\rho $ small enough, depending on the functions,
from Proposition \ref{uprad} applied to the functions $\overline{\widetilde{%
u_{i}}}$ 
\begin{equation*}
\left\vert \overline{\widetilde{u_{1}}}^{\prime }(r)\right\vert \leq
C(N,p,q)r^{-\frac{p+1}{pq-1}},\qquad \left\vert \overline{\widetilde{u_{2}}}%
^{\prime }(r)\right\vert \leq Cr^{-\frac{q+1}{pq-1}}..
\end{equation*}%
Then by integration, for $\rho $ small enough, there holds $\overline{%
\widetilde{u_{1}}}(r)\leq Cr^{-\frac{2-p(q-1)}{pq-1}}$ if $q<q_{4}$ and $%
\overline{\widetilde{u_{1}}}(r)\leq C$ if $q>q_{4},$ similarly for $%
\overline{\widetilde{u_{2}}}$ by exchanging $q_{3}$ and $q_{4}.$ Since the
functions are subharmonic, 
\begin{equation*}
u_{i}(x)\leq C(N)\max (\overline{\widetilde{u_{i}}}(\frac{\left\vert
x\right\vert }{2}),\overline{\widetilde{u_{i}}}(2\left\vert x\right\vert ),
\end{equation*}%
see \cite[Lemma 2.1]{BiGr}, hence the conclusions hold. Moreover if $q>q_{4},
$ then $\left\vert \overline{\widetilde{u_{1}}}^{\prime }\right\vert $ is
integrable, hence $\overline{\widetilde{u_{1}}}$ has a finite limit $l$ as $%
r\longrightarrow 0$.

(ii) The proof is similar.
\end{proof}

\subsubsection{Case of system (\protect\ref{sou})}

We first consider the exterior problem. We observe that in the scalar case,
when $1<q<\frac{N}{N-1},$ there is no positive radial solution of the
equation such that $\lim_{r\rightarrow \infty }u=0.$ But there exist
solutions such that $u$ is increasing to some $l>0$ as $r\rightarrow \infty
, $ so there exist solutions of system (\ref{sou}) for $p=q.$ In the general
case we prove the following when $q$ is subcritical, namely $q<q_{2}$:

\begin{theorem}
\label{nonext}Let $pq>1,$ $p\geq q\geq 1.$ If $q<q_{2},$ there is no
positive supersolution $(u_{1},u_{2})$ of system (\ref{sou}) such that $%
\lim_{r\rightarrow \infty }\overline{u_{1}}=0$ in an exterior set $\mathbb{R}%
^{N}\backslash \overline{B_{r_{0}}}$.
\end{theorem}

\begin{proof}
We have the upper estimates of the derivatives 
\begin{equation}
\left\vert \overline{u_{1}}^{\prime }(r)\right\vert \leq Cr^{-\frac{p+1}{pq-1%
}},\qquad \left\vert \overline{u_{2}}^{\prime }(r)\right\vert \leq Cr^{-%
\frac{q+1}{pq-1}},  \label{esa}
\end{equation}%
and 
\begin{equation*}
\left\{ 
\begin{array}{c}
{-\Delta }\overline{u}_{1}\geqq \left\vert \overline{u_{2}}^{\prime
}\right\vert ^{p}, \\ 
{-\Delta }\overline{u}_{2}\geqq \left\vert \overline{u_{1}}^{\prime
}\right\vert ^{q}.%
\end{array}%
\right.
\end{equation*}%
The function $\overline{u_{1}}$ is superharmonic, and 
\begin{equation*}
\overline{u_{1}}(r)\leq \overline{u_{1}}(r_{0})+\tint\limits_{r_{0}}^{r}%
\left\vert \overline{u_{1}}^{\prime }(r)\right\vert dr\leq \overline{u_{1}}%
(r_{0})+Cr^{-\frac{(2-p(q-1))}{pq-1}},
\end{equation*}%
so as soon as $q<q_{4}$, $\overline{u_{1}}$ is bounded. And more precidely $%
\lim_{r\longrightarrow \infty }\overline{u_{1}}=l_{1}$ (and $%
\lim_{r\longrightarrow \infty }\overline{u_{2}}=l_{2}$). Here we suppose
that $l_{1}=0.$ From the superharmonicity, $r^{N-1}\overline{u_{1}}^{\prime
} $ is decreasing so either $\lim_{r\longrightarrow \infty }r^{N-1}\overline{%
u_{1}}^{\prime }=l\in \left( 0,\infty \right) $ or $\lim_{r\longrightarrow
\infty }r^{N-1}\overline{u_{1}}^{\prime }=-\infty ;$ and $l\in \left(
0,\infty \right) $ is impossible because then $\overline{u_{1}}^{\prime }>0$
for large $r$ so $\overline{u_{1}}$ cannot tend to $0.$ Thus $\overline{u_{1}%
}^{\prime }<0$. Then $r^{N-1}\overline{u_{1}}^{\prime }<-K$ for large $r,$
so $\overline{u_{1}}-\frac{K}{N-2}r^{2-N}$ is decreasing to $0$, hence $%
\overline{u_{1}}\geq \frac{K}{N-2}r^{2-N}.$ But for large $r,$ there holds $%
\overline{u_{1}}^{\prime }(r)\geq -Cr^{-\frac{p+1}{pq-1}}=(\frac{C(pq-1)}{%
p(q_{4}-q)}r^{\frac{p(q-q_{4})}{pq-1}})^{\prime },$ then $\overline{u_{1}}-%
\frac{C(pq-1)}{p(q_{4}-q)}r^{\frac{p(q-q_{4})}{pq-1}},$ is increasing to $0,$
thus $\overline{u_{1}}\leq cr^{\frac{p(q-q_{4})}{pq-1}}=cr^{\frac{p(q-1)-2}{%
pq-1}}.$ Therefore $r^{2-N}\leq r^{\frac{p(q-1)-2}{pq-1}},$ which implies $%
q\geq q_{2},$ leading to a contradiction.
\end{proof}

\begin{remark}
In the case $q<q_{2},$ from Theorem (\ref{exterior}) there exist solutions
such that $\lim_{r\longrightarrow \infty }u_{1}=c_{1}>0.$ In regions $%
\mathcal{C}$ and $\mathcal{D}$, where $q>q_{2}$, we have proved the
existence of solutions in $\mathbb{R}^{N}\backslash B_{r_{0}}$ at the same
theorem. For $q>q_{4}$ from Theorem \ref{thall} there exist solutions on $%
(0,\infty )$. So our result is optimal.
\end{remark}

Finally we consider the local behaviour of the \textit{solutions} of system (%
\ref{sou}) near a singularity, and moreover the existence of a system with
measure data. We give a nonradial existence result of solutions of a
Dirichlet problem when $q<q_{2},$ extending the radial ones of Theorem \ref%
{Dirac} $:$

\begin{theorem}
\label{orig}Let $pq>1,$ $p\geq q\geq 1.\ $Let $\Omega $ be a $C^{2}$ bounded
domain containing $0.\medskip $

(i) Let $({u}_{1},{u}_{2})\in W_{0}^{1,q}(\Omega \backslash \left\{
0\right\} )\times W_{0}^{1,p}(\Omega \backslash \left\{ 0\right\} )$, such
that ${\Delta u}_{1},{\Delta u}_{2}\in L_{loc}^{1}(\Omega \backslash \left\{
0\right\} )$ be a solution of system (\ref{sou}) in $\Omega \backslash
\left\{ 0\right\} ,$ with ${u}_{1}={u}_{2}=0$ on $\partial \Omega $. Then $({%
u}_{1},{u}_{2})\in W_{0}^{1,q}(\Omega )\times W_{0}^{1,p}(\Omega )$ and
there exist $C_{1}\geq 0,C_{2}\geq 0,$ such that 
\begin{equation}
\left\{ 
\begin{array}{c}
{-\Delta u}_{1}=\left\vert \nabla u_{2}\right\vert ^{p}+C_{1}\delta _{0}, \\ 
{-\Delta u}_{2}=\left\vert \nabla u_{1}\right\vert ^{q}+C_{2}\delta _{0},%
\end{array}%
\right. \qquad \text{in }\mathcal{D}^{\prime }(\Omega ).  \label{dpri}
\end{equation}%
Moreover 
\begin{equation}
\text{if }C_{1}>0\text{ then }q<q_{2},\text{ and if }C_{2}>0,\text{ then }p<%
\frac{N}{N-1}.  \label{cond}
\end{equation}%
(ii) Reciprocally, let $C_{1},C_{2}$ satisfying (\ref{cond}) with $C_{1}>0$
or $C_{2}>0,$ and $C_{1},C_{2}$ small enough.Then there exists a solution $({%
u}_{1},{u}_{2})$ of system (\ref{dpri}), such that $({u}_{1},{u}_{2})\in
W_{0}^{1,q}(\Omega )\times W_{0}^{1,p}(\Omega ).$ More generally, for any
bounded Radon measures $\mu ,\nu $ in $\Omega ,$ under the same conditions
on $C_{1},C_{2},$ there exists a solution of the system 
\begin{equation*}
\left\{ 
\begin{array}{c}
{-\Delta u}_{1}=\left\vert \nabla u_{2}\right\vert ^{p}+C_{1}\mu , \\ 
{-\Delta u}_{2}=\left\vert \nabla u_{1}\right\vert ^{q}+C_{2}\nu ,%
\end{array}%
\right. \qquad \text{in }\mathcal{D}^{\prime }(\Omega ).
\end{equation*}
Moreover ${u}_{1},{u}_{2}\in W_{0}^{1,s}(\Omega )$ for any $s\in \left[ 1,%
\frac{N}{N-1}\right) ,$ and if $C_{2}=0,$ then ${u}_{2}\in
W_{0}^{1,s_{2}}(\Omega )$ for $s_{2}<\frac{N}{(N-1)q-1}.$
\end{theorem}

\begin{proof}
(i) The functions ${u}_{1},{u}_{2}$ are superharmonic and nonnegative. Then
there exist $C_{1}\geq 0,C_{2}\geq 0,$ such that (\ref{dpri}) holds, and $%
\left\vert \nabla u_{2}\right\vert ^{p}+\left\vert \nabla u_{1}\right\vert
^{q}\in L_{loc}^{1}(\Omega ),$ and moreover ${u}_{1},{u}_{2}\in
W_{0}^{1,s}(\Omega )$ for any $s\in \left[ 1,\frac{N}{N-1}\right) ,$ see 
\cite{BrLi}. $\ $From Lemma \ref{mean} $(\overline{{u}_{1}},$ $\overline{{u}%
_{2}})$ is a supersolution, then the conditions on $C_{1},C_{2}$ follow from
Theorem \ref{Dirac}.\medskip

(ii) We recall a result of \cite[Theorem 3.1]{AtBeLa}: for any nonnegative $%
f\in L^{m}(\Omega )$ and $g\in L^{k}(\Omega ),$ $m,k\in \left( 1,N\right) $
such that $qk<\frac{mN}{N-m}$ and $pm<\frac{kN}{N-k},$ there exists $\Lambda
=\Lambda (p,q,m,k)>0$ such that the problem 
\begin{equation*}
\left\{ 
\begin{array}{c}
{-\Delta u}_{1}=\left\vert \nabla u_{2}\right\vert ^{p}+C_{1}f, \\ 
{-\Delta u}_{2}=\left\vert \nabla u_{1}\right\vert ^{q}+C_{2}g,%
\end{array}%
\right. \qquad \text{in }\Omega ,
\end{equation*}%
admits a solution such that $({u}_{1},{u}_{2})\in W_{0}^{1,q}(\Omega )\times
W_{0}^{1,p}(\Omega )$ under the condition 
\begin{equation*}
C_{1}^{q}\left\Vert f\right\Vert _{L^{m}(\Omega )}^{q}+C_{2}\left\Vert
g\right\Vert _{L^{k}(\Omega )}\leq \Lambda .
\end{equation*}%
And $({u}_{1},{u}_{2})\in W_{0}^{1,s_{1}}(\Omega )\times
W_{0}^{1,s_{2}}(\Omega )$ for any $s_{1}\in \left[ 1,\frac{mN}{N-m}\right) $
and $s_{2}\in \left[ 1,\frac{kN}{N-k}\right) ,$ with the estimate 
\begin{equation*}
\left\Vert {u}_{1}\right\Vert _{W_{0}^{1,s_{1}}(\Omega )}+\left\Vert {u}%
_{2}\right\Vert _{W_{0}^{1,s_{2}}(\Omega )}\leq C(\Lambda ,s_{1},s_{2},m,k).
\end{equation*}%
Let $\mu ,\nu $ be two positive Radon measures in $\Omega .$ Let $(\mu
_{n}), $ $(\nu _{n})$ be sequences in $C^{\infty ,+}\left( \overline{\Omega }%
\right) $ converging respectively to $\mu ,\nu $ in the sense of measures,
and $\mu _{n}(\Omega )\leq 2\mu (\Omega )$, $\nu _{n}(\Omega )\leq 2\nu
(\Omega ).$ Then there exists a solution of the approximate problem 
\begin{equation*}
\left\{ 
\begin{array}{c}
{-\Delta u}_{1,n}=\left\vert \nabla u_{2,n}\right\vert ^{p}+C_{1}\mu _{n},
\\ 
{-\Delta u}_{2,n}=\left\vert \nabla u_{1,n}\right\vert ^{q}+C_{2}\nu _{n},%
\end{array}%
\right. \qquad \text{in }\Omega ,
\end{equation*}%
First suppose $C_{2}>0,$ then $1\leq q\leq p<\frac{N}{N-1}.$ We take $m=k=1,$
and $f=$ $\mu _{n},$ $g=\nu _{n}$. For $2^{q}C_{1}^{q}\mu ^{q}(\Omega
)+2C_{2}\nu (\Omega )<\Lambda $ we get the existence of sequences $({u}%
_{1,n}),({u}_{2,n})$ bounded in $W_{0}^{1,s}(\Omega )$ for any $s\in \left[
1,\frac{N}{N-1}\right) .$ Next suppose $C_{2}=0<C_{1}.$ We take $f=\mu _{n}$
and $g=0$. We need to find some $m$ and $k$ such that $qk<\frac{mN}{N-m}$
and $pm<\frac{kN}{N-k}.$ We fix $m=1$, so we require that that $\frac{Np}{N+p%
}<k<\frac{N}{(N-1)q}$ (which implies $q<\frac{N}{N-1}).$ This is possible
because $q<q_{2}=\frac{N+p}{(N-1)p}\leq \frac{N}{N-1}$. Then $({u}_{1,n})$
is bounded in $W_{0}^{1,s}(\Omega )$ for any $s\in \left[ 1,\frac{N}{N-1}%
\right) $ ${u}_{2,n}$ bounded in $W_{0}^{1,s_{2}}(\Omega )$ for any $%
s_{2}\in \left[ 1,\frac{N}{(N-1)q-1}\right) ,$ hence in particular for $%
s_{2}\in \left[ 1,\frac{N}{N-1}\right) .$ Therefore we get the existence for 
$2^{q}C_{1}^{q}\mu ^{q}(\Omega )\leq \Lambda .$ Next we can pass to the
limit in $\mathcal{D}^{\prime }(\Omega )$. We fix $s_{1}$ with $q<s_{1}<%
\frac{N}{N-1}$ and $s_{2}$ with $p<s_{2}<\frac{N}{N-1}$ if $C_{2}>0$ and $%
p<s_{2}<\frac{N}{(N-1)q-1}$ if $C_{2}=0;$ up to subsequences, $({u}_{1,n})$
converges weakly and $a.e.$ in $W_{0}^{1,s_{1}}(\Omega )$ and $({u}_{2,n})$
converges weakly and $a.e.$ in $W_{0}^{1,s_{2}}(\Omega )$ to some $%
u_{1},u_{2}$; in both cases, $(\left\vert \nabla u_{2,n}\right\vert ^{p})$
converges strongly in $L^{1}(\Omega ),$ then ${u}$ satisfies the equation in 
$\mathcal{D}^{\prime }(\Omega )$ 
\begin{equation*}
{-\Delta u}_{1}=\left\vert \nabla u_{2}\right\vert ^{p}+C_{1}\mu ,
\end{equation*}%
and $(\left\vert \nabla u_{1,n}\right\vert ^{q})$ converges strongly in $%
L^{1}(\Omega ),$ so we also pass to the limit and get in any case%
\begin{equation*}
{-\Delta u}_{2}=\left\vert \nabla u_{1}\right\vert ^{q}+C_{2}\nu .
\end{equation*}
\end{proof}

\section{Extensions\label{exten}}

(1) First note also that the study of Hardy-H\'{e}non type equations by
using of system (\ref{pspz}) can be adapted to other ranges of the
parameters, for example to a "sublinear" case, that is $\mathbf{q<p-1},$
corresponding to the case where $pq<1$ for the system, or also to case $%
\mathbf{q<0,}$ corresponding to $q<0.$ This allows to extend some results of 
\cite{GiQu} given for the Hardy-H\'{e}non equation $-\Delta w=\left\vert
x\right\vert ^{\sigma }u^{q}$ with the Laplacian to the $m$-Laplacian, in
the radial case, in particular their study in dimension 1.\medskip

(2) Moreover consider an equation of Hardy-H\'{e}non-type in dimension $\nu
\in \mathbb{N},\nu \geq 1,$ with a weight $\left\vert x\right\vert ^{a}$ ($%
a\in \mathbb{R}$) inside the operator: 
\begin{equation*}
-\func{div}(\left\vert x\right\vert ^{a}\left\vert \nabla w\right\vert ^{%
\mathbf{p}-2}\nabla w)=\varepsilon \left\vert x\right\vert ^{b}w^{\mathbf{q}%
},
\end{equation*}%
In the radial case, it reduces to 
\begin{equation*}
-\Delta _{\mathbf{p}}^{\mathbf{N}}w=-\frac{d}{dr}(\left\vert \frac{dw}{dr}%
\right\vert ^{\mathbf{p}-2}\frac{dw}{dr})-\frac{\mathbf{\nu +}a-1}{r}%
\left\vert \frac{dw}{dr}\right\vert ^{\mathbf{p}-2}\frac{dw}{dr}=\varepsilon
r^{b-a}w^{\mathbf{q}},
\end{equation*}%
so it directly falls in the scope of our study, with $N=\nu +a$ and $\sigma
=b-a.\ $It allows for example to find again rapidly some recents results of 
\cite{Vi}, given for the Laplacian with some conditions on $a,b,$ and extend
them to the $p$-Laplacian, in all the ranges of the parameters.\medskip

(3) All our radial study of system (\ref{one}) can be easily directly
extended to a system of $k$-Laplacians, 
\begin{equation}
\left\{ 
\begin{array}{c}
{-\Delta }_{k}{u}_{1}=\left\vert \nabla u_{2}\right\vert ^{p}, \\ 
{-\Delta }_{k}{u}_{2}=\left\vert \nabla u_{1}\right\vert ^{q},%
\end{array}%
\right.  \label{sysk}
\end{equation}%
where $k>1.$ Indeed the radial system reduces to%
\begin{equation*}
\left\{ 
\begin{array}{c}
-(\left\vert {u}_{1}^{\prime }\right\vert ^{k_{1}-2}u_{1}^{\prime })^{\prime
}-\frac{N-1}{r}\left\vert {u}_{1}^{\prime }\right\vert
^{k_{1}-2}u_{1}^{\prime }=\left\vert u_{2}^{\prime }\right\vert ^{p}, \\ 
-(\left\vert {u}_{2}^{\prime }\right\vert ^{k_{2}-2}u_{2}^{\prime })^{\prime
}-\frac{N-1}{r}\left\vert {u}_{2}^{\prime }\right\vert
^{k_{1}-2}u_{2}^{\prime }=\left\vert u_{1}^{\prime }\right\vert ^{q},%
\end{array}%
\right.
\end{equation*}%
\begin{equation*}
\left\{ 
\begin{array}{c}
w_{1}^{\prime }=r^{(N-1)(1-\frac{p}{k-1}})\left\vert w_{2}\right\vert ^{%
\frac{p}{k-1}}, \\ 
w_{2}^{\prime }=r^{(N-1)(1-\frac{q}{k-1}})\left\vert w_{1}\right\vert ^{%
\frac{q}{k-1}},%
\end{array}%
\right.
\end{equation*}%
where 
\begin{equation*}
-r^{N-1}\left\vert {u}_{1}^{\prime }\right\vert ^{k-2}u_{1}^{\prime
}=w_{1},\qquad -r^{N-1}\left\vert {u}_{2}^{\prime }\right\vert
^{k-2}u_{2}^{\prime }=w_{2}.
\end{equation*}%
So we are reduced to study functions $w_{1},w_{2},$ \textbf{with }$p,q$%
\textbf{\ replaced by }$\frac{p}{k-1}$\textbf{\ and }$\frac{q}{k-1};$ and
the results apply for $pq>(k-1)^{2}.$ We obtain analogous results as in the
case $k=2,$ with new parameters $q_{i}$ defined by%
\begin{equation*}
q_{1}=\frac{N}{(N-1)p-k+1},\quad q_{2}=\frac{N+(k-1)p}{(N-1)p},\quad q_{3}=%
\frac{2(k-1)^{2}}{p-k+1},\quad q_{4,}=\frac{(k-1)p+2(k-1)^{2}}{p}.
\end{equation*}%
Moreover one can make an easy adaptation of the results of the nonradial
section, relative to the supersolutions of system (\ref{one}),where the
estimates of the mean values are replaced by integral estimates, as in \cite%
{BiPo} in case of source, aborption or mixed terms. The computations are let
it to the reader. \bigskip

(4) More generally our study allows to treat systems of the type 
\begin{equation*}
\left\{ 
\begin{array}{c}
{-\func{div}(}\left\vert x\right\vert ^{a_{1}}\left\vert \nabla {u}%
_{1}\right\vert ^{k_{1}-2}\nabla {u}_{1})=\left\vert x\right\vert
^{b_{1}}\left\vert \nabla u_{2}\right\vert ^{p}, \\ 
{-\func{div}(}\left\vert x\right\vert ^{a_{2}}\left\vert \nabla {u}%
_{1}\right\vert ^{k_{2}-2}\nabla {u}_{2})=\left\vert x\right\vert
^{b_{2}}\left\vert \nabla u_{1}\right\vert ^{q},%
\end{array}%
\right.
\end{equation*}%
which in the radial case reduce to 
\begin{equation*}
\left\{ 
\begin{array}{c}
-(\left\vert {u}_{1}^{\prime }\right\vert ^{k_{1}-2}u_{1}^{\prime })^{\prime
}-\frac{N+a_{1}-1}{r}\left\vert {u}_{1}^{\prime }\right\vert
^{k_{1}-2}u_{1}^{\prime }=r^{b_{1}-a_{1}}\left\vert u_{2}^{\prime
}\right\vert ^{p}, \\ 
-(\left\vert {u}_{2}^{\prime }\right\vert ^{k_{2}-2}u_{2}^{\prime })^{\prime
}-\frac{N+a_{2}-1}{r}\left\vert {u}_{2}^{\prime }\right\vert
^{k_{1}-2}u_{2}^{\prime }=r^{b_{2}-a_{2}}\left\vert u_{1}^{\prime
}\right\vert ^{q},%
\end{array}%
\right.
\end{equation*}
and then to a Hardy-H\'{e}non equation (\ref{hql}) with new parameters $%
\mathbf{p,q,N}$ defined by 
\begin{equation*}
\frac{p}{k_{2}-1}=\frac{1}{\mathbf{p}-1},\;\;\frac{q}{k_{1}-1}=\mathbf{q}%
,\;\;\frac{\mathbf{p}-\mathbf{N}}{\mathbf{p}-1}=N+b_{1}-(N+a_{2}-1)\frac{p}{%
k_{2}-1},\;\;\mathbf{N}+\sigma =N+b_{2}-(N+a_{1}-1)\mathbf{q};
\end{equation*}
and $\mathbf{q>p}-1$ is equivalent to $pq>(k_{1}-1)(k_{2}-1).\medskip $

(3) Some of our nonradial results can be extended to more general systems,
like (\ref{sysk}), following the methods of \cite{BiPo}, for example Theorem %
\ref{esti}. It would be interesting to extend Theorem \ref{orig} to such
type systems.

\section{Appendix\label{append}}

In this section we give the proofs of the main results of Section \ref{app}

\begin{proof}[Proof of Lemma \protect\ref{nat}]
(i) Setting $s=s_{0}+\overline{s},$ $z=z_{0}+\overline{z},$ the linearized
problem at $M_{0}$ is 
\begin{equation*}
\left\{ 
\begin{array}{c}
\overline{s}_{t}=s_{0}(\overline{s}+\frac{\overline{z}}{\mathbf{p}-1)}), \\ 
\overline{z}_{t}=z_{0}(-\mathbf{q}\overline{s}-\overline{z}).%
\end{array}%
\right.
\end{equation*}%
It admits the eigenvalues $\lambda _{1},\lambda _{2},$ roots of the equation
(\ref{valp}). When the roots are real, that means $s_{0}z_{0}<0$ then $%
\lambda _{1}<0<\lambda _{2},$ and $\lambda _{1}<s_{0}<\lambda _{2}$ since $%
T(s_{0})<0,$ and the corresponding eigenvectors $\overrightarrow{v_{1}}=(%
\frac{s_{0}}{\mathbf{p}-1},\lambda _{1}-s_{0})$ and $\overrightarrow{v_{2}}=(%
\frac{s_{0}}{\mathbf{p}-1},\lambda _{2}-s_{0})$ have the slopes $\rho _{i}=(%
\mathbf{p}-1)(\frac{\lambda _{i}}{s_{0}}-1)=\frac{\mathbf{q}z_{0}}{%
z_{0}+\lambda _{i}},$ $i=1,2.$

(ii) Setting $z=\mathbf{N}+\sigma +\overline{z},$ the linearized problem at $%
N_{0}$ is 
\begin{equation*}
\left\{ 
\begin{array}{c}
s_{t}=\frac{\mathbf{p}+\sigma }{\mathbf{p}-1}s, \\ 
\overline{z}_{t}=(\mathbf{N}+\sigma )(-\mathbf{q}s-\overline{z}).%
\end{array}%
\right.
\end{equation*}%
It admits the eigenvalues $l_{1}=\frac{\mathbf{p}+\sigma }{\mathbf{p}-1}$
and $l_{2}=-(\mathbf{N}+\sigma ).$ When $l_{1}\neq l_{2},$ the corresponding
eigenvectors are $\overrightarrow{v_{1}}=(l_{1}-l_{2},\mathbf{q}l_{2})$ and $%
\overrightarrow{v_{2}}=(0,1).$

(iii) Setting $s=\frac{\mathbf{N}-\mathbf{p}}{\mathbf{p}-1}+\overline{s},$
the linearized system at $A_{0}$ is 
\begin{equation*}
\left\{ 
\begin{array}{c}
\overline{s}_{t}=\frac{\mathbf{N}-\mathbf{p}}{\mathbf{p}-1}(\overline{s}+%
\frac{z}{\mathbf{p}-1}), \\ 
z_{t}=z(\mathbf{N}+\sigma -\mathbf{q}\frac{\mathbf{N}-\mathbf{p}}{\mathbf{p}%
-1}).%
\end{array}%
\right.
\end{equation*}%
It admits the eigenvalues $\mu _{1}=\frac{\mathbf{N}-\mathbf{p}}{\mathbf{p}-1%
}$ and $\mu _{2}=\mathbf{N}+\sigma -\mathbf{q}\frac{\mathbf{N}-\mathbf{p}}{%
\mathbf{p}-1},$ And $z=0$ contain particular trajectories linked to $\mu
_{1} $. When $\mu _{1}\neq \mu _{2}\neq 0,$ corresponding eigenvectors are $%
\overrightarrow{v_{1}}=(1,0)$ and $\overrightarrow{v_{2}}=(\frac{\mu _{1}}{%
\mathbf{p}-1},\mu _{2}-\mu _{1}).$

(iv) The linearized system at point $(0,0),$ 
\begin{equation*}
\left\{ 
\begin{array}{c}
s_{t}=\frac{\mathbf{p}-\mathbf{N}}{\mathbf{p}-1}z, \\ 
z_{t}=(\mathbf{N}+\sigma )z.%
\end{array}%
\right.
\end{equation*}%
gives the eigenvalues $\rho _{1}=\frac{\mathbf{p}-\mathbf{N}}{\mathbf{p}-1}%
,\rho _{2}=\mathbf{N}+\sigma ,$ which are distinct for $\sigma \neq -\frac{%
(N-1)p}{p-1}$, and corresponding eigenvectors $\overrightarrow{v_{1}}=(1,0)$
and $\overrightarrow{v_{2}}=(0,1).$
\end{proof}

$\medskip $

\begin{proof}[Proof of Lemma \protect\ref{link}]
(i) Direct consequence of (\ref{wwp}).

(ii) Suppose $\lim (s,z)=N_{0}$. From (\ref{pspz}) we find $s_{t}=s(l_{1}+s+%
\overline{z})=s(l_{1}+o(1)),$ then $s=o(e^{-\frac{\left\vert
l_{1}t\right\vert }{2}}),$ and $\overline{z}_{t}=z-(\mathbf{N}+\sigma
)=(-l_{2}+\overline{z})(-\mathbf{q}s-\overline{z})=\overline{z}(l_{2}-%
\mathbf{q}s-\overline{z})+l_{2}\mathbf{q}s=(l_{2}+o(1))\overline{z}%
+o(e^{-2\left\vert l_{1}t\right\vert }).$ Then $z=o(e^{-k\left\vert
t\right\vert })$ for some $k>0.$ In turn $s_{t}=s(l_{1}+o(e^{-k\left\vert
t\right\vert })$ for some $k>0,$ thus $\lim e^{-l_{1}t}s=C\neq 0.$ The
conclusion follows from (\ref{wwp}) with $\lim e^{-l_{1}t}s=C\neq 0$ and $%
\lim z=N+\sigma .\medskip $

(iii) Suppose $\lim (s,z)=(0,0)$. Then $s_{t}=s(\rho _{1}+s+\overline{z}%
)=s(\rho _{1}+o(1)),$ and $z_{t}=z(\rho _{2}+o(1)).$ Then $%
s,z=o(e^{-k\left\vert t\right\vert })$ for some $k>0.$ Therefore we obtain $%
s_{t}=s(\rho _{1}+o(e^{kt})),z_{t}=z(\rho _{2}+o(e^{kt})).$ By integration
we deduce $\lim e^{-\rho _{1}t}s=C\neq 0$ and $\lim e^{-\rho _{2}t}z=D\neq 0$%
, and we still apply(\ref{wwp}).\medskip

(iv) Suppose $\lim (s,z)=A_{0}.$ Then $z_{t}=z(\mu _{2}-\mathbf{q}\overline{s%
}-z)$ with $\overline{s}=s-\frac{\mathbf{N}-\mathbf{p}}{\mathbf{p}-1}=o(1),$
so that $z_{t}=z(\mu _{2}+o(1)),$ then $z=o(e^{-\frac{\left\vert \mathbf{\mu 
}_{2}t\right\vert }{2}}).$ It follows that $\overline{s}_{t}=(\mu _{1}+%
\overline{s})(\overline{s}+\frac{z}{p-1})=\overline{s}(\mu _{1}+\frac{z}{p-1}%
+\overline{s})+\frac{\mathbf{N}-\mathbf{p}}{(\mathbf{p}-1)}z=\mu _{1}%
\overline{s}+o((e^{-\frac{\left\vert \mathbf{\mu }_{2}t\right\vert }{2}}).$
Then $s=o(e^{-k\left\vert t\right\vert })$ for some $k>0;$ in turn $%
z_{t}=z(\mu _{2}+o(e^{-k\left\vert t\right\vert });$ thus $\lim e^{-\mu
_{2}t}z=D\neq 0,$ and we apply (\ref{wwp}) where $\lim s=\mu _{1},\lim
e^{-\mu _{2}t}z=D.$
\end{proof}

$\medskip $

\begin{proof}[Proof of Lemma \protect\ref{glo}]
Let for example $w$ be any solution in $(0,r_{0}).$ From Lemma (\ref{osse}),
and (\ref{wwp}), $\Phi =\left\vert s\right\vert ^{\mathbf{p}-1}\left\vert
z\right\vert $ is bounded. And 
\begin{equation*}
\Phi _{t}=\Phi (\mathbf{p}+\sigma -(\mathbf{q}-\mathbf{p}+1)s).
\end{equation*}%
Suppose that $s$ is unbounded near $-\infty $. Either $s$ is monotone near $%
-\infty ,$ then $\lim_{t\longrightarrow -\infty }\left\vert s\right\vert
=\infty ,$ hence $\frac{\Phi _{t}}{\Phi }\geq -1$ near $-\infty ,$ $\ln
\left\vert \Phi \right\vert \geq \frac{\left\vert t\right\vert }{2},$ which
is impossible. Or there exists a sequence $t_{n}\longrightarrow -\infty $
such that $\left\vert s(t_{n})\right\vert \longrightarrow \infty $ of points
of maximum of $\left\vert s\right\vert .$ At these points, and $\frac{%
\mathbf{p}-\mathbf{N}}{\mathbf{p}-1}+s(t_{n})+\frac{z(t_{n})}{\mathbf{p}-1}%
=0 $, hence $s(t_{n})z(t_{n})<0,$ and 
\begin{equation*}
s_{tt}(t_{n})=\frac{s(t_{n})z_{t}(t_{n})}{\mathbf{p}-1}=\frac{%
s(t_{n})z(t_{n})}{\mathbf{p}-1}(\mathbf{p}+\sigma -(\mathbf{q}-\mathbf{p}%
+1)s(t_{n})),
\end{equation*}%
thus $s_{tt}(t_{n})s(t_{n})>0,$ which is contradictory. Then $s$ is bounded.
Suppose that $z$ is unbounded. Either $z$ is monotone near $-\infty ,$ then $%
\lim_{t\longrightarrow -\infty }\left\vert z\right\vert =\infty ,$ thus $%
\lim_{t\longrightarrow -\infty }\left\vert s_{t}\right\vert =\infty ,$ which
is impossible. Or there exists a sequence $t_{n}\longrightarrow -\infty $
such that $\left\vert z(t_{n})\right\vert \longrightarrow \infty $ of points
of maximum of $\left\vert z\right\vert .$ At these points there holds $%
\mathbf{N}+\sigma -\mathbf{q}s(t_{n})-z(t_{n})=0,$ which is impossible since 
$s$ is bounded. Then $z$ is bounded. It converges to one of the fixed points
of the system, or has a limit cycle around them. Since $N_{0},(0,0)$ or $%
A_{0}$ have real eigenvalues, such cycle can happen only at $M_{0},$ when $%
\mathbf{q}=\mathbf{q}_{S}.$
\end{proof}

\medskip

\begin{proof}[Proof of Theorem \protect\ref{orange}]
Here we consider region $\mathbf{A,}$ where $\mathbf{N}>\mathbf{p}>-\sigma .$
In this case we refer to \cite{BiGi} for the description of the system. Here
the behaviour depend on position of $\mathbf{q}$ with respect to $\mathbf{q}%
_{c},$ and also of $\mathbf{q}_{S}$ when $\varepsilon =1,$ see
figures1,2,3,4. The point $M_{0}$ is in $Q_{4}$ for $\mathbf{q}<\mathbf{q}%
_{c},$ corresponding to solutions of the equation with absorption ($%
\varepsilon =-1$), and $M_{0}\in Q_{1}$ for $\mathbf{q}>\mathbf{q}_{c}.$ The
point $N_{0}$ is a saddle point, corresponding to $C^{0}$-regular solutions
near $0$. The point $A_{0}$ is a source for $\mathbf{q}<\mathbf{q}_{c}$ and
a saddle point for $\mathbf{q}>\mathbf{q}_{c}.$ The point $(0,0)$ is a
saddle point, and the associated trajectories are not admissible. The
existence of ground states of the equation with source ($\varepsilon =1$) if
and only if $q\geq \mathbf{q}_{S}$ is well known. The phase plane in the
critical Sobolev case $q=\mathbf{q}_{S}$ is remarkable: there exist
particular particular trajectories , located on a straight line 
\begin{equation*}
(\mathbf{p}-1)s+\frac{\mathbf{p}z}{\mathbf{q}+1}-\mathbf{N}+\mathbf{p}=0.
\end{equation*}%
These trajectories correspond to the ground states given at (\ref{ground})
when $\varepsilon =1,$ and to explicit local solutions near $0$ or $\infty $
when $\varepsilon =-1.$
\end{proof}

\begin{equation*}
\begin{array}{ccc}
\FRAME{itbpF}{2.7068in}{2.7068in}{0in}{}{}{regionaqlessq1.png}{\special%
{language "Scientific Word";type "GRAPHIC";maintain-aspect-ratio
TRUE;display "USEDEF";valid_file "F";width 2.7068in;height 2.7068in;depth
0in;original-width 2.5584in;original-height 2.5584in;cropleft "0";croptop
"1";cropright "1";cropbottom "0";filename
'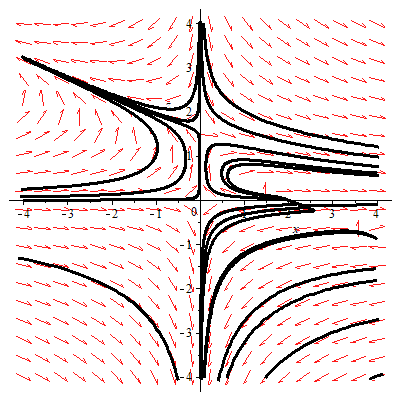';file-properties "XNPEU";}} &  & \FRAME{itbpF}{2.7068in}{%
2.7068in}{0in}{}{}{regionaqbetq1andqstar.png}{\special{language "Scientific
Word";type "GRAPHIC";maintain-aspect-ratio TRUE;display "USEDEF";valid_file
"F";width 2.7068in;height 2.7068in;depth 0in;original-width
2.5584in;original-height 2.5584in;cropleft "0";croptop "1";cropright
"1";cropbottom "0";filename '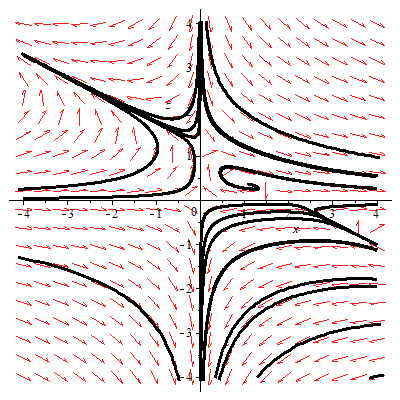';file-properties
"XNPEU";}} \\ 
\text{Figure 1,Theorem \ref{orange}: } &  & \text{Figure 2,Theorem \ref%
{orange}: } \\ 
\mathbf{N}=2.2>\mathbf{p}=1.4>-\sigma =0.6,\text{ }\mathbf{q}=0.7\mathbf{<q}%
_{c}=0.75 &  & \mathbf{q}_{c}<\mathbf{q}=0.85\mathbf{<q}_{S}=85/90%
\end{array}%
\end{equation*}

\begin{equation*}
\begin{array}{ccc}
\FRAME{itbpF}{2.7068in}{2.7068in}{0in}{}{}{regionaqgreaterqstar.png}{\special%
{language "Scientific Word";type "GRAPHIC";maintain-aspect-ratio
TRUE;display "USEDEF";valid_file "F";width 2.7068in;height 2.7068in;depth
0in;original-width 2.5584in;original-height 2.5584in;cropleft "0";croptop
"1";cropright "1";cropbottom "0";filename
'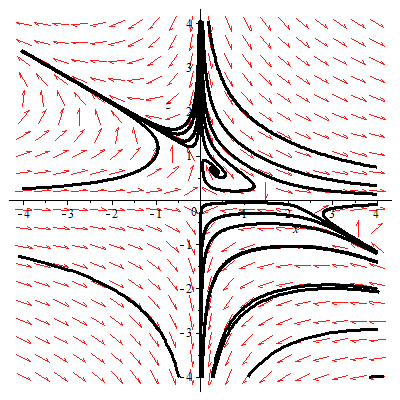';file-properties "XNPEU";}} &  & \FRAME{itbpF}{%
2.7068in}{2.7068in}{0in}{}{}{regionaqegalqstar.png}{\special{language
"Scientific Word";type "GRAPHIC";maintain-aspect-ratio TRUE;display
"USEDEF";valid_file "F";width 2.7068in;height 2.7068in;depth
0in;original-width 2.5584in;original-height 2.5584in;cropleft "0";croptop
"1";cropright "1";cropbottom "0";filename
'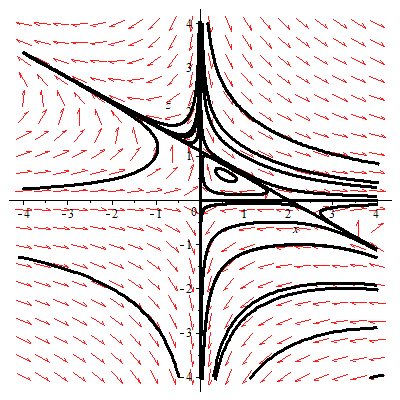';file-properties "XNPEU";}} \\ 
\text{Figure 3,Theorem \ref{orange}: } &  & \text{Figure 4,Theorem \ref%
{orange}: } \\ 
\mathbf{q}=1>\mathbf{q}_{S}=\frac{85}{90} &  & \mathbf{q=q}_{S}=\frac{85}{90}%
\end{array}%
\end{equation*}

\begin{proof}[Proof of Theorem \protect\ref{qccrit}]
$\bullet $ The point $N_{0}$ is still as saddle point as in Theorem \ref%
{orange}, corresponding to $C^{0}$-regular solutions near $0$. \medskip

$\bullet $ Here the point $M_{0}=(\frac{\mathbf{N}-\mathbf{p}}{\mathbf{p}-1}%
,0)\ $coincides with $A_{0},$ see figure 5. The eigenvalues are $\mu _{1}=%
\frac{\mathbf{N}-\mathbf{p}}{\mathbf{p}-1}$ and $\mu _{2}=0,$ with
eigenvectors are $\overrightarrow{v_{1}}=(0,1)$ and $\overrightarrow{v_{2}}%
=(1,-(\mathbf{p}-1))$. Corresponding to the eigenvalue $0,$ there exists a
central manifold of dimension 1, tangent to $\overrightarrow{v_{2}},$
invariant by the flow. It contains trajectories converging to $M_{0}$ as $%
t\longrightarrow \infty $ or $t\longrightarrow -\infty $. The eigenspace
associated to $\mu _{1}>0$ has dimension $1,$ so it is precisely the axis $%
z=0,$ and the trajectories are not admissible.\medskip

First assume $\varepsilon =1.$ In fact there exists an infinity of
trajectories tangent to $\overrightarrow{v_{2}},$ that means an infinity of
central manifolds. The idea is the following: we know that there exists a
trajectory $\mathcal{T}_{0}$ starting from $N_{0}$ and corresponding to $%
C^{0}$-regular solutions; they are not global, $w$ vanishes, which means
that $s\longrightarrow \infty $ and $z\longrightarrow 0$ in finite time. The
region $\mathcal{R}$ in $Q_{1}$ delimitated by $\mathcal{T}_{0}$ is
invariant. The region $\mathcal{R\cap }\left\{ s\leq \frac{\mathbf{N}-%
\mathbf{p}}{\mathbf{p}-1}\right\} $ is negatively invariant, since the field
at $(\frac{\mathbf{N}-\mathbf{p}}{\mathbf{p}-1},z)$ satisfies $s_{t}=\frac{sz%
}{p-1}>0.$ Then any trajectory passing by a point of this region stays in
it, and cannot converge to $(0,0),$ so it converges to $M_{0}$ as $%
t\longrightarrow -\infty $.

Setting $s=\frac{\mathbf{N}-\mathbf{p}}{\mathbf{p}-1}+\overline{s},$ we
obtain the system 
\begin{equation*}
\left\{ 
\begin{array}{c}
\overline{s}_{t}=(\overline{s}+\frac{\mathbf{N}-\mathbf{p}}{\mathbf{p}-1})(%
\overline{s}+\frac{z}{\mathbf{p}-1}), \\ 
z_{t}=z(-\mathbf{q}\overline{s}-z).%
\end{array}%
\right. 
\end{equation*}%
Since the trajectory is tangent to $\overrightarrow{v_{2}},$ there holds $%
z\sim -(\mathbf{p}-1)\overline{s},$ then $z_{t}\sim \frac{\mathbf{q}+1-%
\mathbf{p}}{\mathbf{p-1}}z^{2}$ and by integration $z\sim -\frac{\mathbf{p}-1%
}{(\mathbf{q}+1-\mathbf{p})t}=-\frac{\mathbf{N}-\mathbf{p}}{(\mathbf{p}%
+\sigma )t}.$ Since $s\sim \frac{\mathbf{N}-\mathbf{p}}{\mathbf{p}-1},$ we
deduce the exact behaviour of $w$ from (\ref{wwp}).\medskip 

Then assume $\varepsilon =-1.$ There exist trajectories converging to $M_{0}$
as $t\longrightarrow \infty .$ Indeed up to a scaling there exist solutions $%
w$ with a logarithmic behaviour as $r\longrightarrow \infty $, from \cite%
{Ve1} for $\sigma =0,$ obtained by minimisation. and the construction
extends to $\sigma \neq 0$. The exact behaviour follows as before.
\end{proof}

\begin{equation*}
\begin{array}{c}
\FRAME{itbpF}{2.7068in}{2.7068in}{0in}{}{}{regionaqegalq1.png}{\special%
{language "Scientific Word";type "GRAPHIC";maintain-aspect-ratio
TRUE;display "USEDEF";valid_file "F";width 2.7068in;height 2.7068in;depth
0in;original-width 2.5584in;original-height 2.5584in;cropleft "0";croptop
"1";cropright "1";cropbottom "0";filename
'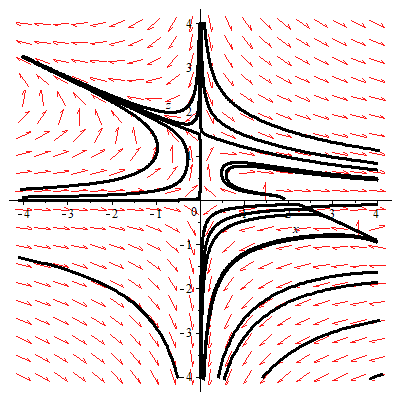';file-properties "XNPEU";}} \\ 
\text{Figure 5,Theorem \ref{qccrit}: } \\ 
\mathbf{N}=2.2>\mathbf{p}=1.4>\sigma =-0.6,\mathbf{q}=\mathbf{q}_{c}=0.75%
\end{array}%
\end{equation*}

\begin{proof}[Proof of Theorem \protect\ref{rose}]
Here we consider region $\mathbf{B,}$ where $\mathbf{p}>\mathbf{N}>\sigma ,$
see figure 6.\medskip

$\bullet $ The point $M_{0}$ is in $Q_{4},$ since $s_{0}=\gamma =\frac{%
\mathbf{p}+\sigma }{\mathbf{q}+1-\mathbf{p}}>0,$ and $z_{0}=\mathbf{N}-%
\mathbf{p}-(\mathbf{p}-1)\gamma <\mathbf{N}-\mathbf{p}<0$, and $M_{0}$ is a
saddle point from Lemma \ref{nat}. There exist four trajectories $\mathcal{T}%
_{1}^{M},\mathcal{T}_{2}^{M},\mathcal{T}_{3}^{M},\mathcal{T}_{4}^{M},$ such
that $\mathcal{T}_{1}^{M},\mathcal{T}_{2}^{M}$ (resp. $\mathcal{T}_{3}^{M},%
\mathcal{T}_{4}^{M})$ converge to $M_{0}$ as $t\rightarrow \infty $ (resp.$%
-\infty $). associated to $\lambda _{1}<0$ (resp. $\lambda _{2}>0$) with a
negative slope $\rho _{1}=-(\mathbf{p}-1)(1+\frac{\left\vert \lambda
_{1}\right\vert }{s_{0}})<-(\mathbf{p}-1)$ (resp. a positive slope). Then $%
\mathcal{T}_{1}^{M}$ lies in the bounded region $\mathcal{H}$ of $Q_{4}$
where $s_{t}<0,z_{t}<0$ for large $t.$ Since this region is negatively
invariant, the trajectory $\mathcal{T}_{1}^{M}$ stays in $\mathcal{H}$, so
it is bounded, hence defined on $(\infty ,\infty ),$ and converges to $(0,0)$
as $t\longrightarrow -\infty ,$ since it is the only possible fixed point in 
$\overline{\mathcal{H}}$.\medskip 

$\bullet $ The point $N_{0}$ is located at the boundary of $Q_{1}$ and $%
Q_{2}.$ The eigenvalues satisfy $l_{2}<0<l_{1}$, then $N_{0}$ is a saddle
point. There are two trajectories $\mathcal{T}_{1}^{\prime },\mathcal{T}%
_{2}^{\prime },$ starting from $N_{0},$ the first one in $Q_{1}$ and one in $%
Q_{2},$ and the corresponding solutions satisfy (\ref{no}) the solutions are
not global, from Lemma \ref{glo}, since in $\overline{Q_{1}}$ and $\overline{%
Q_{2}}$ there is not fixed point attracting at $\infty .$.\medskip 

$\bullet $ The point $A_{0}$ is located at the boundary of $Q_{2}$ and $%
Q_{3}.$ The eigenvalues satify $\mu _{1}<0<\mu _{2},$ then it is a saddle
point. Two trajectories corresponding to $\mu _{1}$ are ending at $A_{0}$ as 
$t\rightarrow \infty $ are located on the line $z=0$ not admissible. Two
trajectories correponding to $\mu _{2}$ are starting from $A_{0}$ at $%
-\infty ,$ a trajectory $\mathcal{T}_{2}^{A}$ in $Q_{2}$ and a trajectory $%
\mathcal{T}_{3}^{A}$ in $Q_{3},$ with a negative slope $-(\mathbf{p}-1)\frac{%
\mu _{2}+\left\vert \mu _{1}\right\vert }{\left\vert \mu _{1}\right\vert }.$
The corresponding solutions satisfy (\ref{ao}). They are not global, from
Lemma \ref{glo}, since in $\overline{Q_{2}}$ and $\overline{Q_{3}}$ there is
no attracting point at $\infty .$\medskip

$\bullet $ The point $(0,0)$ is a source, since $\rho _{1}>0$ and $\rho
_{2}>0$. So if $\rho _{1}\neq \rho _{2},$ that means $\frac{\mathbf{p}-%
\mathbf{N}}{\mathbf{p}-1}\neq \mathbf{N}+\sigma ,$ there exists an infinity
of trajectories starting from $(0,0)$. From Lemma \ref{link} the
corresponding solutions satisfy (\ref{oo}). If $\rho _{1}=\rho _{2}$, we
prove that there exist particular solutions at Proposition \ref{expli}%
.\medskip

$\bullet $ In $\overline{Q_{1}}$ and $\overline{Q_{3}}$ there is no fixed
point attracting at $\infty ,$ so for $\varepsilon =1$ there is no solution
in $(r_{0},\infty ).$
\end{proof}

\begin{equation*}
\begin{array}{ccc}
\FRAME{itbpF}{2.7068in}{2.7068in}{0in}{}{}{regionb.png}{\special{language
"Scientific Word";type "GRAPHIC";maintain-aspect-ratio TRUE;display
"USEDEF";valid_file "F";width 2.7068in;height 2.7068in;depth
0in;original-width 2.5584in;original-height 2.5584in;cropleft "0";croptop
"1";cropright "1";cropbottom "0";filename '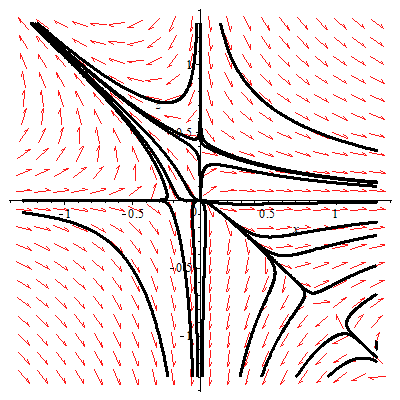';file-properties
"XNPEU";}} &  & \FRAME{itbpF}{2.7068in}{2.7068in}{0in}{}{}{regionc.png}{%
\special{language "Scientific Word";type "GRAPHIC";maintain-aspect-ratio
TRUE;display "USEDEF";valid_file "F";width 2.7068in;height 2.7068in;depth
0in;original-width 2.5584in;original-height 2.5584in;cropleft "0";croptop
"1";cropright "1";cropbottom "0";filename '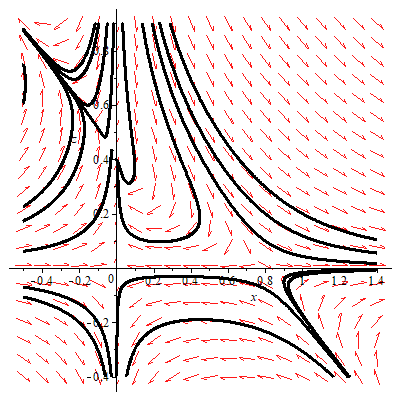';file-properties
"XNPEU";}} \\ 
\text{Figure 6: Theorem \ref{rose} } &  & \text{Figure 7: Theorem \ref%
{yellow} } \\ 
\mathbf{p}=\frac{12}{7}>\mathbf{N}=\frac{11}{7}>\sigma =-\frac{1.82}{7},%
\mathbf{q}=1.3 &  & \mathbf{p}=\frac{29}{19}<\mathbf{-}\sigma =-\frac{29.4}{%
19}<\mathbf{N}=\frac{37}{19},\mathbf{q}=1.3%
\end{array}%
\end{equation*}

\begin{proof}[Proof of Theorem \protect\ref{yellow}]
Here we consider region $\mathbf{C,}$ where $\mathbf{p}<\mathbf{-}\sigma <%
\mathbf{N},$ see figure 7.\medskip

$\bullet $ The point $M_{0}$ is in $Q_{2}$. Indeed $s_{0}=\gamma <0;$ and $%
z_{0}=\mathbf{N}-\mathbf{p}+(\mathbf{p}-1)\left\vert \gamma \right\vert >0.$
It is a saddle point, and the eigenvalues satisfy $\lambda
_{1}<s_{0}<\lambda _{2}.$ So there exist four trajectories $\mathcal{T}%
_{1}^{M},\mathcal{T}_{1}^{\prime M},\mathcal{T}_{2}^{M},\mathcal{T}%
_{2}^{\prime M},$ such that $\mathcal{T}_{1}^{M},\mathcal{T}_{1}^{\prime M},$
(resp. $\mathcal{T}_{2}^{M},\mathcal{T}_{2}^{\prime M})$ converge to $M_{0}$
as $t\rightarrow \infty $ (resp.$-\infty $). associated to $\lambda _{1}<0$
(resp. $\lambda _{2}>0$) with the slopes $\rho _{i}=(\mathbf{p}-1)(\frac{%
\lambda _{i}}{s_{0}}-1)$ $i=1,2$; hence $\rho _{1}>0>\rho _{2},$ and we call 
$\mathcal{T}_{1}^{M},\mathcal{T}_{2}^{M},$ the trajectories such that $%
s(t)>s_{0}$ for any $t\in \mathbb{R}.$In particular $\mathcal{T}_{2}^{M}$
starts in the region $J$ of $Q_{2}$ where $s_{t}>0,z_{t}<0.\ $We check
easily that the region $J$ is positively invariant. Then the trajectory $%
\mathcal{T}_{2}^{M}$ stays in $J,$ so it is bounded, and since $s$ and $z$
are monotone, we get that $\mathcal{T}_{2}^{M}$ converges at $\infty $ to a
fixed point in $\overline{Q_{2}},$ that is to $N_{0}$ if $\mathbf{N}+\sigma
>0,$ and to $(0,0)$ if $\mathbf{N}+\sigma <0.$

$\bullet $ The point $(0,0)$ admits the eigenvalues $\rho _{1}=\frac{\mathbf{%
p}-\mathbf{N}}{\mathbf{p}-1}<0$ and $\rho _{2}=\mathbf{N}+\sigma >0,$ it is
a saddle point, the trajectories issued from this point are the two axes,
they are not admissible.\medskip 

$\bullet $ The point $N_{0}$ is located at the boundary of $Q_{1}$ and $%
Q_{2}.$ The eigenvalues satisfy $l_{1}<0$ and $l_{2}<0,$ so $N_{0}$ is a
sink. In any case, all the corresponding solutions satisfy (\ref{no2}), from
Lemma \ref{link}. In any case, we know that the trajectory $\mathcal{T}%
_{2}^{M}$ converges to $N_{0}.$\medskip 

$\bullet $ The point $A_{0}=(\frac{\mathbf{N}-\mathbf{p}}{\mathbf{p}-1},0).$
is located at  the boundary of $Q_{1}$ and $Q_{4}.$ It has the eigenvalues $%
\mu _{1}=\frac{\mathbf{N}-\mathbf{p}}{\mathbf{p}-1}>0$ and 
\begin{equation*}
\mu _{2}=\mathbf{N}+\sigma -\mathbf{q}\frac{\mathbf{N}-\mathbf{p}}{\mathbf{p}%
-1}<\mathbf{N}+\sigma -(\mathbf{N}-\mathbf{p})=\mathbf{p}+\sigma <0,
\end{equation*}%
so it is a saddle point. As we have done before, we check that there exist
two trajectories converging to $A_{0}$ as $t\rightarrow \infty ,$ one in $%
Q_{1}$ and one in $Q_{4}$. This show the existence of local solutions of the
equations with $\varepsilon =\pm 1,$ such that $\lim_{r\rightarrow \infty
}r^{\frac{\mathbf{N-p}}{\mathbf{p}-1}}w=k>0,$ satisfying (\ref{ao2}).
Clearly those solutions are not global, since there is no fixed point in $%
\overline{Q_{1}}$ and $\overline{Q_{4}}$ attracting at $-\infty ,$ from
Lemma \ref{glo} So the corresponding functions $w$ are only solutions of the
exterior problem.
\end{proof}

\medskip

\begin{proof}[Proof of Theorem \protect\ref{expli}]
Setting $s=\Phi z,$ we get 
\begin{equation*}
\Phi _{t}=\Phi (\frac{\mathbf{p}-\mathbf{N}}{\mathbf{p}-1}-(\mathbf{N}%
+\sigma )+(\frac{\mathbf{p}}{\mathbf{p}-1}+(\mathbf{q}+1)\Phi )z)
\end{equation*}%
Then system admits particular solutions such that $s\equiv -\delta z$ for
some $\delta >0$ if and only if $\frac{\mathbf{p}-\mathbf{N}}{\mathbf{p}-1}=%
\mathbf{N}+\sigma ,$ that means $\sigma =-\mathbf{p}\frac{\mathbf{N}-1}{%
\mathbf{p}-1},$ and $\delta =\frac{\mathbf{p}}{(\mathbf{p}-1)(\mathbf{q}+1)}$%
. They are given explicitely from the equation%
\begin{equation*}
s_{t}=s(\frac{\mathbf{p}-\mathbf{N}}{\mathbf{p}-1}-\frac{\mathbf{q}-\mathbf{p%
}+1}{\mathbf{p}-1}s),
\end{equation*}%
and from (\ref{fid}), $\delta \varepsilon r^{\sigma }w^{q+1}\left\vert
w^{\prime }\right\vert ^{-p}=-1,$ hence $\varepsilon =-1$ and we get $w^{-%
\frac{\mathbf{q}+1}{\mathbf{p}}}w^{\prime }=\pm \delta ^{\frac{1}{p}}r^{-%
\frac{\mathbf{N}-1}{\mathbf{p}-1}},$ hence (\ref{flo}) by integration, with $%
d_{\mathbf{p},\mathbf{q},\mathbf{N}}=\frac{(\mathbf{p}-1)(\mathbf{q}-\mathbf{%
p}+1)}{\mathbf{p}(\mathbf{p}-\mathbf{N)}}(\frac{\mathbf{p}}{(\mathbf{p}-1)(%
\mathbf{q}+1)})^{\frac{1}{\mathbf{p}}}$. These solutions satisfy $%
\lim_{r\longrightarrow 0}w=C$ when $C>0$ and when they are global, $w\sim
_{r\longrightarrow \infty }w^{\ast }.\medskip $

If $\mathbf{p}=\mathbf{N,}$ we get $\sigma =-\mathbf{p=-N}$, and $w^{-\frac{%
\mathbf{q}+1}{\mathbf{N}}}w^{\prime }=\pm \delta ^{\frac{1}{\mathbf{N}}%
}r^{-1},$ then $w^{1-\frac{\mathbf{q}+1}{\mathbf{N}}}=C\pm \delta ^{\frac{1}{%
\mathbf{N}}}(\frac{\mathbf{q}+1}{\mathbf{N}}-1)\ln r,$ where $\delta =\frac{%
\mathbf{N}}{(\mathbf{N}-1)(\mathbf{q}+1)},$ that is $d_{\mathbf{N}}=\frac{%
\mathbf{q}+1-\mathbf{N}}{\mathbf{N}}(\frac{\mathbf{N}}{(\mathbf{N}-1)(%
\mathbf{q}+1)})^{\frac{1}{\mathbf{N}}}.$
\end{proof}

\medskip

\begin{proof}[Proof of Theorem \protect\ref{sigmap}]
Using Lemma \ref{reduc}, we reduce to the case (1)\textbf{\ }$\mathbf{N}>%
\mathbf{p}=-\sigma ,$ see figure 8.\medskip

$\bullet $ Here the point $M_{0}=(0,\mathbf{N}-\mathbf{p})$ coincides with $%
N_{0}$.The eigenvalues are $l_{1}=0$ and $l_{2}=\mathbf{p}-\mathbf{N}<0.$
The corresponding eigenvectors are $\overrightarrow{v_{1}}=(1,-\mathbf{q})$
and $\overrightarrow{v_{2}}=(0,1).$ Corresponding to the eigenvalue $0,$
there exists a central manifold of dimension $1$, tangent to $%
\overrightarrow{v_{1}},$ invariant by the flow. It contains trajectories
converging to $M_{0}$ as $t\longrightarrow \infty $ or $t\longrightarrow
-\infty $. The eigenspace associated to $l_{2}>0$ has dimension $1,$ so it
is precisely the axis $s=0,$ and the trajectories are not admissible. The
system becomes  
\begin{equation*}
\left\{ 
\begin{array}{c}
s_{t}=s(\frac{\mathbf{p}-\mathbf{N}}{\mathbf{p}-1}+s+\frac{z}{\mathbf{p}-1}),
\\ 
z_{t}=z(\mathbf{N}-\mathbf{p}-\mathbf{q}s-z).%
\end{array}%
\right. 
\end{equation*}

$\bullet $ The point $A_{0}=(\frac{\mathbf{N}-\mathbf{p}}{\mathbf{p}-1},0).$
is located at the boundary of $Q_{1}$ and $Q_{4}.$ The eigenvalues are $\mu
_{1}=\frac{\mathbf{N}-\mathbf{p}}{\mathbf{p}-1}>0$ and $\mu _{2}=(\mathbf{q-p%
}+1)\frac{\mathbf{p}-\mathbf{N}}{\mathbf{p}-1}<0,$ with eigenvectors. $%
\overrightarrow{\nu _{1}}=(1,0)$ and $\overrightarrow{\nu _{2}}=(1,-(\mathbf{%
p}-1)(\mathbf{q}-\mathbf{p}+2)).$ Two trajectories associated to $A_{0}$ are
located on the line $\left\{ z=0\right\} $ , so they are not admissible.
There exist two other trajectories, $\mathcal{T}_{1}$ and $\mathcal{T}_{4},$
respectively in $Q_{1}$ and $Q_{4},$ converging to $A_{0}$ at $\infty .$ The
corresponding solutions $w$ satisfy $\lim_{r\longrightarrow \infty }r^{\frac{%
\mathbf{N}-\mathbf{p}}{\mathbf{p}-1}}w=k>0,$ from Lemma \ref{link}, for $%
\varepsilon =\pm 1$.\medskip 

$\bullet $ First consider the trajectories located in $Q_{1},$ corresponding
the case $\varepsilon =1.$ The trajectory $\mathcal{T}_{1}$ has a slope less
that $-(\mathbf{p}-1),$ then it arrives as $t\longrightarrow \infty $ in the
region $\mathcal{R}=\left\{ s_{t}>0\right\} $ in $Q_{1}$ delimitated by the
line $N_{0}A_{0}$ where $s_{t}=0;$ and $\mathcal{R}$ is negatively
invariant, so $\mathcal{T}_{1}$ stays in it. The line $\left\{
z_{t}=0\right\} $ is located under the line $M_{0}A_{0},$ thus $z_{t}<0$ in $%
\mathcal{R}.$ If $\mathcal{T}_{1}$ converges to $M_{0},$ then the region $%
\mathcal{R}_{1}$ delimitated by the two axis and $\mathcal{T}_{1}$ is
bounded with no fixed point inside, and this is contradictory. Then $%
\mathcal{T}_{1}$ is asymptotic to the axis $\left\{ s=0\right\} ,$ and the
solutions are not global. Consider again the region $\mathcal{R}_{1}.$ Then
any trajectory passing by a point of $\mathcal{R}_{1}$ is bounded, and then
converges to $M_{0}$ as $t\longrightarrow \infty ,$ so we get a (2
parameters) family of solutions $w$ in $(r_{0},\infty )$ such that (\ref%
{port}) holds. Indeed the trajectory is tangent to line of direction $(1,-%
\mathbf{q})$ passing by $M_{0},$ which means $\overline{z}\sim -qs,$ and we
have 
\begin{equation*}
\left\{ 
\begin{array}{c}
s_{t}=s(s+\frac{\overline{z}}{\mathbf{p}-1}), \\ 
z_{t}=(\mathbf{N}-\mathbf{p}+\overline{z})(\mathbf{-}\overline{z}-\mathbf{q}%
s).%
\end{array}%
\right. 
\end{equation*}%
Then $s_{t}\sim -\frac{\mathbf{q}}{\mathbf{p}-1}s^{2},$ and by integration $%
s\sim \frac{\mathbf{p}-1}{\mathbf{q}t},$ and $z\sim \mathbf{N}-\mathbf{p,}$
thus by (\ref{wwp}), $w\sim (\frac{\mathbf{p}-1}{\mathbf{q}t})^{\frac{%
\mathbf{p}-1}{\mathbf{q}+1-\mathbf{p}}},$ which is precisely (\ref{port}%
).\medskip 

$\bullet $ Next consider the trajectories located in $Q_{2},$ corresponding
the case for $\varepsilon =-1.$ The line $\left\{ s_{t}=0\right\}
=M_{0}A_{0} $ is located under the line $\left\{ z_{t}=0\right\} $ and the
region $\mathcal{G}$ located between the two lines is negatively invariant.
Any trajectory passing by a point above $\left\{ z_{t}=0\right\} $
necessarily crosses this line since $z$ and $s$ are decreasing, and any
trajectory passing by a point under $\left\{ s_{t}=0\right\} $ cuts this
line, because $s$ and $z$ are increasing. We consider two sets of
trajectories: $\mathcal{U}_{1}$ is the union of trajectories which cut $%
\left\{ z_{t}=0\right\} $ at some point, and $\mathcal{U}_{2}$ is the union
of trajectories which cut $\left\{ s_{t}=0\right\} $ at some point. They are
open in $Q_{2}$, so $\mathcal{U}_{1}\cup \mathcal{U}_{2}\neq Q_{2}.$ Then
there exists some point $P$ such that the trajectory passing by $P$ is
located in $\mathcal{G},$ which is negatively invariant. Then it converges
to the point $M_{0}$ as $t\longrightarrow -\infty .$ So we get solutions $w$
satisfying (\ref{part}).
\end{proof}

\medskip 
\begin{equation*}
\begin{array}{ccc}
\FRAME{itbpF}{2.7068in}{2.7068in}{0in}{}{}{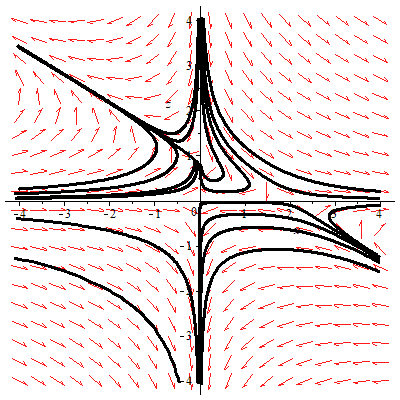}{\special%
{language "Scientific Word";type "GRAPHIC";maintain-aspect-ratio
TRUE;display "USEDEF";valid_file "F";width 2.7068in;height 2.7068in;depth
0in;original-width 2.5584in;original-height 2.5584in;cropleft "0";croptop
"1";cropright "1";cropbottom "0";filename 'sigmap-defin.png';file-properties
"XNPEU";}} &  & \FRAME{itbpF}{2.7068in}{2.7068in}{0in}{}{}{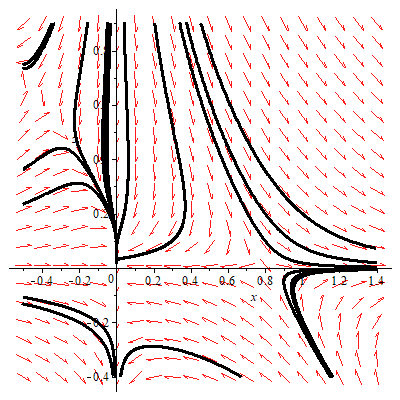}{%
\special{language "Scientific Word";type "GRAPHIC";maintain-aspect-ratio
TRUE;display "USEDEF";valid_file "F";width 2.7068in;height 2.7068in;depth
0in;original-width 2.5584in;original-height 2.5584in;cropleft "0";croptop
"1";cropright "1";cropbottom "0";filename 'sigman.png';file-properties
"XNPEU";}} \\ 
\text{Figure 8: Theorem \ref{sigmap}} &  & \text{Figure 9: Theorem \ref%
{sigman}} \\ 
\mathbf{N}=2.2>\mathbf{p}=-\sigma =1.4,\mathbf{q}=1.1 &  & -\sigma =\mathbf{N%
}=\frac{37}{19}>\mathbf{p}=\frac{29}{19},\mathbf{q}=1.5%
\end{array}%
\end{equation*}

\begin{proof}[Proof of Theorem \protect\ref{sigman}]
We still reduce to the case\textbf{\ }(1) $-\sigma =\mathbf{N}>p,$ see
figure 9. Here the system is 
\begin{equation}
\left\{ 
\begin{array}{c}
s_{t}=s(\frac{\mathbf{p}-\mathbf{N}}{\mathbf{p}-1}+s+\frac{z}{\mathbf{p}-1}),
\\ 
z_{t}=z(-\mathbf{q}s-z).%
\end{array}%
\right.   \label{sih}
\end{equation}

$\bullet $ The point $M_{0}=(s_{0},z_{0})=(\frac{\mathbf{p}-\mathbf{N}}{%
\mathbf{q}+1-\mathbf{p}},\frac{(\mathbf{N}-\mathbf{p})\mathbf{q}}{\mathbf{q}%
+1-\mathbf{p}})$ is in $Q_{2}$, corresponding to particular solutions of the
equation with absorption $\varepsilon =-1.$ $w^{\ast }=a^{\ast }r^{\frac{%
\mathbf{N}-\mathbf{p}}{\mathbf{q}+1-\mathbf{p}}}$; note that $%
\lim_{r\longrightarrow 0}w^{\ast }=0.$

$\bullet $ The point $N_{0}$ coincides with $(0,0)$; the eigenvalues are $%
\rho _{1}=\frac{\mathbf{p}-\mathbf{N}}{\mathbf{p}-1},\rho _{2}=0,$
corresponding eigenvectors $\overrightarrow{v_{1}}=(1,0)$ and $%
\overrightarrow{v_{2}}=(0,1).$ The trajectories associated to $\rho _{1}$
are not admissible. There exist at least one central manifold, relative to $%
\rho _{2},$ tangent to the axis $\left\{ s=0\right\} .$ We are going to
precise it below:

The eigenvalues at point $M_{0}$ are given by (\ref{valp}), their product is 
$-(\mathbf{N}-\mathbf{p})^{2}\mathbf{q}<0.,$ so $M_{0}$ is a saddle point.
There exist four trajectories $\mathcal{T}_{1}^{M},\mathcal{T}_{2}^{M},%
\mathcal{T}_{3}^{M},\mathcal{T}_{4}^{M},$ such that $\mathcal{T}_{1}^{M},%
\mathcal{T}_{2}^{M}$ (resp. $\mathcal{T}_{3}^{M},\mathcal{T}_{4}^{M})$
converge to $M_{0}$ as $t\rightarrow \infty $ (resp.$-\infty $). associated
to $\lambda _{1}<0$ (resp. $\lambda _{2}>0$) with the slope $\rho _{1}=(%
\mathbf{p}-1)(\frac{\lambda _{1}}{s_{0}}-1)>0$ (resp. $\rho _{2}=(\mathbf{p}%
-1)(-\frac{\lambda _{2}}{\left\vert s_{0}\right\vert }-1)<-(\mathbf{p}-1))$.
Then one of the two last trajectories, denoted by $\mathcal{T}_{3}^{M}$ lies
in the bounded region $\mathcal{H}$ of $Q_{2}$ where $s_{t}>0,z_{t}<0$ for
large $t.$ Since this region is positively invariant, the trajectory $%
\mathcal{T}_{3}^{M}$ stays in $\mathcal{H}$, so it is bounded, hence
converges to $(0,0)$ which is the only possible fixed point in $\overline{%
\mathcal{H}}$. Then any trajectory passing by some point of $\mathcal{H}$
converges to $(0,0).$ So there is an infinity of trajectories converging to
this point, staying in $Q_{2}$.\medskip

$\bullet $ The point $A_{0}$ is located at the boundary of $Q_{1}$ and $%
Q_{4}.$ Moreover the eigenvalues at $A_{0}$ are $\mu _{1}=\frac{\mathbf{N}-%
\mathbf{p}}{\mathbf{p}-1}>0$ and $\mu _{2}=\mathbf{q}\frac{\mathbf{p}-%
\mathbf{N}}{\mathbf{p}-1}<0,$ $A_{0}$ is a saddle point. Two trajectories
associated to $\mu _{1}$are located on the line $z=0$ are not admissible.
There are trajectories $\mathcal{T}_{1},\mathcal{T}_{4}$, respectively in $%
Q_{1}$ and $Q_{4}$ converging to $A_{0}$ as $t\longrightarrow \infty $
corresponding to the eigenvector $\overrightarrow{v_{2}}=(\frac{\mu _{1}}{%
\mathbf{p}-1},\mu _{2}-\mu _{1}),$ with the slope $-(\mathbf{p}-1)(\mathbf{q}%
+1),$ the solutions are not global, since there is no fixed point in those
quadrants. Then there exist solutions in $(r_{0},\infty )$ for $\varepsilon
=\pm 1,$ such that $\lim_{r\longrightarrow \infty }r^{\frac{\mathbf{N}-%
\mathbf{p}}{\mathbf{p}-1}}w=k>0.$ Moreover, since $(\mathbf{p}-1)(\mathbf{q}%
+1)>(\mathbf{p}-1),$ $\mathcal{T}_{1}$ is located in the region of $Q_{1}$
where $s_{t}>0$ and $z_{t}<0,$ which is negatively invariant, so it stays in
it. Next consider the region $\mathcal{K}$ in $Q_{1}$ delimitated by the two
axis and $\mathcal{T}_{1}$ (under $\mathcal{T}_{1}$). It is invariant, and $%
z_{t}<0$ in $\mathcal{K}.$ Then consider any trajectory $\mathcal{T}$
passing by a point of $\mathcal{K}.$ It cannot converge to $A_{0}$ because $%
\mathcal{T}_{1}$ is the only trajectory in $Q_{1}$ converging to this point.
Then $\mathcal{T}$ necessarily converges to $(0,0).$ There still exists an
infinity of trajectories converging to this point, staying in $Q_{1}$%
.\medskip 

Thus in $Q_{1}$ (corresponding to $\varepsilon =1)$ as well as in $Q_{2},$
(corresponding to $\varepsilon =-1)$ there is an infinity of trajectories
converging to $(0,0).$ From system (\ref{sih}), since $\lim_{t%
\longrightarrow \infty }\frac{s}{z}=0,$ we get $s_{t}\sim s(\frac{\mathbf{p}-%
\mathbf{N}}{\mathbf{p}-1}),$ and $z_{t}\sim -z^{2}.$ Then by integration, $%
z\sim \frac{1}{t}.\left\vert s\right\vert \leq Ce^{-kt}$ for some $k>0.$
From (\ref{fid}) we get $w^{\frac{\mathbf{q}}{1-\mathbf{p}}}w^{\prime }=(-%
\frac{\mathbf{p}-1}{\mathbf{q}+1-\mathbf{p}}w^{-\frac{\mathbf{q}+1-\mathbf{p}%
}{\mathbf{p}-1}})^{\prime }\sim _{r\longrightarrow \infty }-\varepsilon r^{%
\frac{1-\mathbf{N}}{p-1}}(\ln r)^{\frac{1}{\mathbf{p}-1}},$ then by
integration, either 
\begin{equation*}
w\sim -\varepsilon C(\ln r)^{-\frac{\mathbf{1}}{\mathbf{q}+1-\mathbf{p}}}r^{%
\frac{\mathbf{N}-\mathbf{p}}{\mathbf{q}+1-\mathbf{p}}},\qquad w^{\prime
}\sim -\varepsilon C^{\frac{\mathbf{q}}{\mathbf{p}-1}}(\ln r)^{-\frac{%
\mathbf{1}}{\mathbf{q}+1-\mathbf{p}}}r^{\frac{\mathbf{N}-\mathbf{1-q}}{%
\mathbf{q}+1-\mathbf{p}}},\qquad C=(\frac{\mathbf{q}+1-\mathbf{p}}{\mathbf{N}%
-\mathbf{p}})^{\frac{\mathbf{p}-1}{\mathbf{q}+1-\mathbf{p}}},
\end{equation*}%
or 
\begin{equation*}
\lim_{r\longrightarrow \infty }w=C>0,\qquad w^{\prime }\sim -\varepsilon r^{%
\frac{1-\mathbf{N}}{p-1}}(\ln r)^{\frac{1}{\mathbf{p}-1}}.
\end{equation*}%
The first eventuality is impossible for $\varepsilon =1.$ Let us check if it
is possible for $\varepsilon =-1:$ it would imply that $s=-r\frac{w^{\prime }%
}{w}$ tends to $\mathbf{p}-\mathbf{N}\neq 0,$ which is contradictory. Then
for $\varepsilon =-1,$ for given $r_{0}>0$ there exists an infinity (with 2
parameters) of solutions $w$ in $(r_{0},\infty )$ satisfying (\ref{hec}),
and an infinity (with one parameter) of solutions satisfying (\ref{hoc}).
\end{proof}

\medskip

\begin{proof}[Proof of Theorem \protect\ref{pegaln}]
Here again we reduce to the case (1) $\mathbf{p=N}>-\sigma ,$ see figure 10.
The system reduces to 
\begin{equation*}
\left\{ 
\begin{array}{c}
s_{t}=s(s+\frac{z}{\mathbf{p}-1}), \\ 
z_{t}=z(\mathbf{p}+\sigma -\mathbf{q}s-z).%
\end{array}%
\right. 
\end{equation*}

$\bullet $ The point $M_{0}=(s_{0},z_{0})=(\frac{\mathbf{p}+\sigma }{\mathbf{%
q}+1-\mathbf{p}},-\frac{(\mathbf{p}-1)(\mathbf{p}+\sigma )}{\mathbf{q}+1-%
\mathbf{p}})$ is in $Q_{4},$ corresponding to particular solutions of the
equation with absorption $\varepsilon =-1.$ $w^{\ast }=a^{\ast }r^{-\frac{%
\mathbf{p}+\sigma }{\mathbf{q}+1-\mathbf{p}}}$, which are $\infty $%
-singular.\medskip 

$\bullet $ The point $N_{0}=(\mathbf{p}+\sigma ,0)$ is located at the
boundary of $Q_{1}$ and $Q_{2}$ and admits the eigenvalues $l_{1}=\frac{%
\mathbf{p}+\sigma }{\mathbf{p}-1}>0$ and $l_{2}=-(\mathbf{p}+\sigma )<0,$
with eigenvectors $\overrightarrow{v_{1}}=(\frac{\mathbf{p}}{\mathbf{p}-1},-%
\mathbf{q})$ and $\overrightarrow{v_{2}}=(0,1).$ it is a saddle point; the
trajectories associated to $l_{2}$ are not admissible. There exists two
trajectories, $\mathcal{T}_{1}$ and $\mathcal{T}_{2}$, respectively in $Q_{1}
$ and $Q_{2}$ starting from $N_{0}$ at $-\infty ,$ corresponding to $C^{0}$%
-regular solutions $w.$ The trajectory $\mathcal{T}_{2}$ starts in the
region where $s_{t}<0,z_{t}>0,$ which is positively invariant, so it stays
in it. The trajectory $\mathcal{T}_{1}$ starts in the region where $%
s_{t}>0,z_{t}<0,$ with a slope $-\mathbf{q}\frac{\mathbf{p}-1}{\mathbf{p}},$
above the line $N_{0}M_{0},$ and stays in the region of $Q_{1}$ ahere $%
z_{t}<0,$ and then it is asymptotic to the axis $\left\{ z=0\right\} .$%
\medskip 

$\bullet $ Here the point $A_{0}$ coincides with $(0,0),$ and the
eigenvalues are $\rho _{1}=0,\rho _{2}=\mathbf{N}+\sigma ,$ with
eigenvectors $\overrightarrow{v_{1}}=(1,0)$ and $\overrightarrow{v_{2}}%
=(0,1).$ The trajectories associated to $\rho _{2}$ are not admissible.
There exist at least one central manifold, relative to $\rho _{1},$ tangent
to the axis $\left\{ z=0\right\} .$ We are going to precise it
below.\medskip 

The eigenvalues at point $M_{0}$ are given by (\ref{valp}), their product is 
$s_{0}z_{0}<0,$ then $\lambda _{1}<0<s_{0}<\lambda _{2},$ so $M_{0}$ is a
saddle point, There exist four trajectories $\mathcal{T}_{1}^{M},\mathcal{T}%
_{2}^{M},\mathcal{T}_{3}^{M},\mathcal{T}_{4}^{M},$ such that $\mathcal{T}%
_{1}^{M},\mathcal{T}_{2}^{M}$ (resp. $\mathcal{T}_{3}^{M},\mathcal{T}%
_{4}^{M})$ converge to $M_{0}$ as $t\rightarrow \infty $ (resp.$-\infty $).
associated to $\lambda _{1}<0$ (resp. $\lambda _{2}>0$) with the slope $\rho
_{1}=(\mathbf{p}-1)(\frac{\lambda _{1}}{s_{0}}-1)<0$ (resp. $\rho _{2}=(%
\mathbf{p}-1)(\frac{\lambda _{2}}{s_{0}}-1)>(\mathbf{p}-1)$). Then one of
the two first trajectories, denoted by $\mathcal{T}_{1}^{M}$ lies in the
bounded region $\mathcal{H}$ of $Q_{4}$ where $s_{t}>0,z_{t}<0$ for large $%
t. $ Since this region is negatively invariant, the trajectory $\mathcal{T}%
_{1}^{M}$ stays in $\mathcal{H}$, so it is bounded, hence converges to $%
(0,0) $ as $t\longrightarrow -\infty $ which is the only possible fixed
point in $\overline{\mathcal{H}}$. Moreover any trajectory passing by some
point of $\mathcal{H}$ converges to $(0,0).$ So there is an infinity of
trajectories converging to this point as $t\longrightarrow \infty $, staying
in $Q_{4}$.\medskip

Next consider the other quadrants $Q_{1,}Q_{2}.$ There is no trajectory
converging to $(0,0)$ in $Q_{2}.$ Indeed suppose that a trajectory converges
to $(0,0)$ in $Q_{2}$ as $t\longrightarrow \infty ;$ then it satisfies $%
s_{t}>0,z_{t}<0,$ near $\infty $ which is impossible; if it converges as $%
t\longrightarrow -\infty ,$ then $s_{t}<0,z_{t}>0,$ then it cannot be
tangent to the axis $\left\{ z=0\right\} .$\medskip

Finally consider the quadrant $Q_{1}.\ $The region $\mathcal{K}$ delimitated
by the two axis and the line $N_{0}M_{0}$ is negatively invariant. Then any
trajectory passing by a point of $\mathcal{K}$ converges to $(0,0)$ at $%
-\infty .$ Moreover the region $\mathcal{L}$ delimitated by the two axis and 
$\mathcal{T}_{1}$ is also invariant. Then any trajectory passing by a point
of $\mathcal{L}$ converges to $(0,0)$ at $-\infty .$ Thus there is an
infinity of central manifolds.\medskip

Those trajectories satisfy $\lim_{t\longrightarrow -\infty }\frac{z}{s}=0,$
then $s_{t}\sim s^{2}.$ By integration, $0<s=-\frac{w_{t}}{w}\sim \frac{1}{%
\left\vert t\right\vert }.$ Thus for given $\varepsilon >0,$ there holds $%
C_{1}\left\vert t\right\vert ^{(1+\varepsilon )}\leq w\leq C_{2}\left\vert
t\right\vert ^{(1+\varepsilon )}$ near $-\infty .$ Since $\mathbf{p}=\mathbf{%
N,}$ equation (\ref{sca}) reduces to 
\begin{equation*}
-(N-1)\left\vert w_{t}\right\vert ^{\mathbf{N}-2}w_{tt}=\varepsilon e^{(%
\mathbf{N}+\sigma )t}w^{q}.
\end{equation*}%
Then 
\begin{equation*}
-(\mathbf{N}-1)\left\vert w_{t}\right\vert ^{\mathbf{N}-2-\mathbf{q}}w_{tt}=(%
\frac{\mathbf{N}-1}{\mathbf{q}-\mathbf{N}+1}\left\vert w_{t}\right\vert ^{%
\mathbf{N}-1-\mathbf{q}})_{t}\sim \varepsilon e^{(\mathbf{N}+\sigma
)t}\left\vert t\right\vert ^{q},
\end{equation*}%
hence by integration, either $\lim_{t\longrightarrow -\infty }w_{t}=-c<0,$
or $\frac{\mathbf{N}-1}{\mathbf{q}-\mathbf{N}+1}\left\vert w_{t}\right\vert
^{\mathbf{N}-1-\mathbf{q}}\sim -\frac{\varepsilon }{\mathbf{N}+\sigma }e^{(%
\mathbf{N}+\sigma )t}\left\vert t\right\vert ^{q}$. The last case is
impossible for $\varepsilon =1.$ If $\varepsilon =-1,$ since $\frac{%
\left\vert w_{t}\right\vert }{w}\sim \frac{1}{\left\vert t\right\vert },$ it
implies that $w^{-\frac{\mathbf{q}}{N-1}}\left\vert w_{t}\right\vert =(\frac{%
N-1}{q-N+1}w^{-\frac{q-N+1}{N-1}})_{t}\sim Ce^{\frac{\mathbf{N}+\sigma }{N-1}%
t},$ which by integration contradicts the estimate $C_{1}\left\vert
t\right\vert ^{(1+\varepsilon )}\leq w\leq C_{2}\left\vert t\right\vert
^{(1+\varepsilon )}.$ Then $w\sim _{r\longrightarrow 0}C\left\vert \ln
r\right\vert ,$ and $w^{\prime }\sim _{r\longrightarrow 0}-\frac{C}{r}.$
\end{proof}

\begin{equation*}
\begin{array}{ccc}
\FRAME{itbpF}{2.7068in}{2.7068in}{0in}{}{}{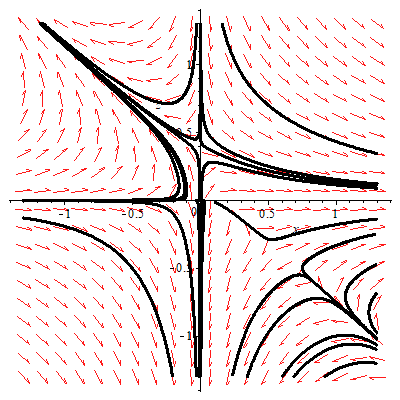}{\special%
{language "Scientific Word";type "GRAPHIC";maintain-aspect-ratio
TRUE;display "USEDEF";valid_file "F";width 2.7068in;height 2.7068in;depth
0in;original-width 2.5584in;original-height 2.5584in;cropleft "0";croptop
"1";cropright "1";cropbottom "0";filename 'regionpegaln.png';file-properties
"XNPEU";}} &  & \FRAME{itbpF}{2.7068in}{2.7068in}{0in}{}{}{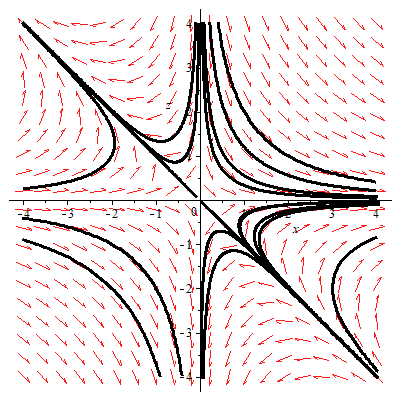%
}{\special{language "Scientific Word";type "GRAPHIC";maintain-aspect-ratio
TRUE;display "USEDEF";valid_file "F";width 2.7068in;height 2.7068in;depth
0in;original-width 2.5584in;original-height 2.5584in;cropleft "0";croptop
"1";cropright "1";cropbottom "0";filename
'pnsigma-defin.png';file-properties "XNPEU";}} \\ 
\text{Figure 10: Theorem \ref{pegaln}} &  & \text{Figure 11: Theorem \ref%
{pnsigma}} \\ 
\mathbf{p=N}=\frac{5}{3}>-\sigma =\frac{19}{15},\mathbf{q}=1.3 &  & \mathbf{%
p=N}=-\sigma =\frac{5}{3},\text{ }\mathbf{q}=1.5%
\end{array}%
\end{equation*}

\begin{proof}[Proof of Theorem \protect\ref{pnsigma}]
Here $\mathbf{p=N}=-\sigma ,$ see figure 11. The system reduces to%
\begin{equation*}
\left\{ 
\begin{array}{c}
s_{t}=s(s+\frac{z}{\mathbf{N}-1}), \\ 
z_{t}=z(-\mathbf{q}s-z).%
\end{array}%
\right. 
\end{equation*}%
The unique fixed point is $(0,0)$, and the two eigenvalues are $0.$ From
Proposition \ref{expli}, there exist solutions such that $s\equiv -\frac{%
\mathbf{N}}{(\mathbf{N}-1)(\mathbf{q}+1)}z,$ corresponding to trajectories
in $Q_{2}$ and $Q_{4}$. We get a family of solutions $w$ given by (\ref{flai}%
). The sign $+$ (resp. $-$) corresponds to the trajectory in $Q_{4}$ (resp. $%
Q_{2}$). It seems that there is no other trajectory converging to $(0,0),$
thus no other solution $w$ can be defined near $0,$ or near $\infty .$%
\medskip 

For $\varepsilon =1$, corresponding to quadrants $Q_{1}$ and $Q_{3}$ all the
trajectories satisfy $s_{t}z_{t}<0,$ hence they cannot converge to the
unique fixed point $(0,0)$ as $t\longrightarrow \pm \infty ,$ from Lemma \ref%
{glo}. Therefore there is no local solution near 0 or $\infty $.
\end{proof}

\end{document}